\documentclass[11pt,a4paper]{article}
\usepackage{epstopdf}
\usepackage{enumerate}
\usepackage{graphicx}           
\graphicspath{ {Images/} }     
\usepackage{amsmath}               
\usepackage{amsfonts}              
\usepackage{amsthm}               
\usepackage{color}
\usepackage{empheq}

\usepackage{amssymb,amsbsy,amsmath,amsfonts,amssymb,amscd,amsthm, mathrsfs}
\usepackage{dsfont}
\usepackage{hyperref}

\newcommand{\supp}{\operatorname{supp}}

\newcommand{\NN}{\mathbb{N}}
\newcommand{\RR}{\mathbb{R}}

\newcommand{\dd}{\,\mathrm{d}}

\newtheorem{theorem}{Theorem}
\newtheorem{proposition}[theorem]{Proposition}
\newtheorem{lemma}[theorem]{Lemma}
\newtheorem{corollary}[theorem]{Corollary}

\theoremstyle{definition}
\newtheorem{definition}[theorem]{Definition}
\newtheorem{remark}[theorem]{Remark}
\newtheorem{cutoff}[theorem]{Cutoff}

\newcommand{\beqn}{\begin{equation}}
\newcommand{\eeqn}{\end{equation}}
\newcommand{\bear}{\begin{eqnarray}}
\newcommand{\eear}{\end{eqnarray}}
\newcommand{\bean}{\begin{eqnarray*}}
\newcommand{\eean}{\end{eqnarray*}}

\begin{document}
\begin{center}
\vspace*{2cm}
\LARGE{On a system of equations for the \\normal fluid-condensate interaction\\ in a Bose gas}
\end{center}
\medskip
\begin{center}
Enrique Cort\'es\footnotemark[1], Miguel Escobedo\footnotemark[2]
\end{center}
\medskip
\begin{abstract}
The existence of global solutions for a  system of differential equations is proved, and some of their properties are described. The system involves a kinetic equation for quantum particles. It is a simplified version of a mathematical description of a weakly interacting dilute gas of bosons in the presence of a condensate near the critical temperature.
\end{abstract}

\footnotetext[1]{BCAM - Basque Center for Applied Mathematics. Alameda de Mazarredo 14,
E--48009 Bilbao, Spain.
E-mail~: {\tt e.cortes.coral@gmail.com}
}
\footnotetext[2]{Departamento de Matem\'aticas, Universidad del
Pa{\'\i}s Vasco, Apartado 644, E--48080 Bilbao, Spain.
E-mail~: {\tt mtpesmam@lg.ehu.es}}

\section{Introduction}

\setcounter{equation}{0}
\setcounter{theorem}{0}

We consider the existence and properties of radially symmetric weak solutions to the following system of differential equations:

\begin{empheq}[left=\empheqlbrace]{align}
&\frac {\partial F} {\partial t}(t, p)=n(t) I_3(F(t))(p)\qquad t>0,\; p\in \RR^3, \label{PA}\\
&n'(t)=-n(t)\int _{ \RR^3 } I_3(F(t))(p)dp \qquad t>0,\label{PB}
\end{empheq}
where
\begin{align}
&I_3(F(t))(p)=\!\!\iint _{(\RR^3)^2}\!\!\big[R(p, p_1, p_2)\!-\!R(p_1, p, p_2)\!-\!R(p_2, p_1, p) \big]dp_1dp_2, \label{E1BCD} \\
&R(p, p_1, p_2)\,=\left[\delta (|p|^2-|p_1|^2-|p_2|^2)  \delta (p-p_1-p_2)\right]\,\times \nonumber \\
&\hskip 2.5cm \times\left[ F_1F_2(1+F)-(1+F_1)(1+F_2)F \right],  \label{S1EA4BC}
\end{align}
and we denote $F=F(t,p)$ and $F_{\ell}=F(t,p_{\ell})$ for $\ell=1, 2$.

The system (\ref{PA}), (\ref{PB}) is motivated by the mathematical description of a weakly interacting dilute gas of bosons. Given such a gas at equilibrium, if its temperature is below the so-called critical temperature $T_c$, a macroscopic density of bosons, called a condensate, appears at the lowest quantum state (cf.\cite{lieb}). A description of the system of particles out of equilibrium at zero temperature has also been rigorously obtained  (\cite{ESY}). The system (\ref{PA}), (\ref{PB})  is more directly related to a gas out of equilibrium and at non zero temperature. The equations that, in the physic's literature, describe a gas in such a situation have not been the object of a mathematical proof; they have rather been deduced on the basis of physical arguments  (cf. \cite{Zoller}, \cite{GNZB, ZNG}, \cite{Stoof2} for example). 
We are particularly interested in the kinetic description of the interaction between the condensate and the particles in the dilute gas, when most of the particles are still in the gas, and so when the system is at a temperature close to  $T_c$.  

\subsection{The Nordheim equation}

The kinetic equation consistently used to describe the evolution of the  distribution function for a spatially homogeneous, weakly interacting dilute gas of bosons of momentum $p_1$ is
\begin{align}
&\frac{\partial F}{\partial t}(t, p_1)=I_4(F(t))(p_1)\qquad  t>0,\; p_1\in \RR^3, \label{S1E0a}
\end{align}
where
\begin{align}
&I_4(F(t))(p_1)=\iiint_{\left(\mathbb{R}^{3}\right)  ^{3}}q(F)  d\nu (p_2, p_3, p_4),  \label{S1Eb}\\
&q(F)= F_{3}F_{4}( 1+ F_{1})(1+F_{2}) -F_{1}F_{2}(1+F_{3})(1+F_{4}), \label{S1Eq}\\
&d\nu (p_1, p_2, p_3)=2a^2 \pi ^{-3} \delta\left(  p_{1}+p_{2}-p_{3}-p_{4}\right)\times \nonumber \\
&\hskip 2.5cm \delta \left(  E(p_1)+ E(p_2)- E(p_3)- E(p_4)\right) dp_{2}dp_{3}dp_{4}.
\end{align}
sometimes called Nordheim equation (\cite{Nordheim1}), (cf. for example  \cite{Zoller}, \cite{GNZB}, \cite{Stoof2}). We are assuming that the particles have mass $m=1/2$ and $E(p)$ denotes  the energy of a particle of momentum $p$.  The constant $a$ is the scattering length that parametrizes the Fermi pseudopotential  of scattering. In the absence of condensate, the energy of the particles is taken to be $E(p)=|p|^2$.  

For a condensed Bose gas, it is necessary to include  the collisions involving the condensate. A kinetic equation is derived in \cite{Eckern} and \cite{Kirkpatrick} describing such processes.  More recently, \cite{ZNG}  extended the treatment to a trapped Bose gas by including Hartree-Fock
corrections to the energy of the excitations, and  have derived coupled kinetic equations for the
distribution functions of the normal and superfluid components. Later on the results where generalized  to low temperatures in \cite{IG}  using  the Bogoliubov-Popov approximation to describe the energy particle. The system is as follows
\begin{empheq}[left=\empheqlbrace]{align}
&\frac {\partial F} {\partial t}(t, p)=I_4(F(t))(p)+ 32 a^2 n(t) \widetilde I_3(F(t))(p)\quad t>0,\; p\in \RR^3,  \label{S1EA1}\\
& n'(t)=-n(t)\int  _{ \RR^3 } \widetilde I_3(F(t))(p)dp\qquad t>0. \label{S1EA1B}
\end{empheq}
(cf. \cite{Eckern}, \cite{GNZB}, \cite{Kirkpatrick} for a deduction based on physic's arguments). The term $I_4(F)$ is exactly as in (\ref{S1Eb}) and   the constant $32 a^2$ comes from the approximation of  the transition probability:
$|\mathcal M(p, p_1, p_2)|^2\approx 32 a^2 n(t)$. The integral collision $\widetilde I_3$ is given by an expression similar to (\ref{E1BCD}), (\ref{S1EA4BC})
 but where the corresponding terms $\widetilde R(p, p_1, p_2)$ are as follows,
 \begin{align}
&\widetilde R(p, p_1, p_2)\,=\left[\delta (E(p)-E(p_1)-E(p_2))  \delta (p-p_1-p_2)\right]\,\times \nonumber \\
&\hskip 1cm \times\left[ F(p_1)F(p_2)(1+F(p))-(1+F(p_1))(1+F(p_2))F(p) \right].   \label{S1EBog}
\end{align}
In presence of a condensate, the energy $E(t, p)$ of the particles at time $t$ is now taken as $E(t, p)=\sqrt{|p|^4+16a\,n(t)|p|^2}$,  where $n(t)$ is the condensate density (\cite{Chiara}, \cite{GNZB}). 
Once equation (\ref{S1EA1}) has been obtained, the equation (\ref{S1EA1B}) is just what is needed in order to ensure that the total number of particles $n(t)+\int  _{ \RR^3 }F(t, p)dp$ in the system is constant in time.  

We are particularly interested in a situation where most of the particles are in the gas, and the condensate density $n$ is very small. The energy of the particles is then usually approximated as $E(t, p)\approx |p|^2+4 a\pi n(t)$ (cf.\cite{GNZB}).
In all what follows we need  the strongest simplification  $E(t, p)\approx |p|^2$ to have the collision integral $I_3$ in (\ref{E1BCD}).
 
Moreover, in the problem (\ref{PA}), (\ref{PB})  only the term that in the equation  (\ref{S1EA1})  describes the interactions involving one particle of the condensate has been kept. The term $I_4$, the same as in equation (\ref{S1E0a}), that only considers interactions between particles in the gas, has been dropped. The term $I_4$ has been studied with detail to prove the existence of solutions to the Nordheim equation (\ref{S1E0a}) and describe some of their properties. The problem (\ref{PA}), (\ref{PB}) only takes into account the collision processes involving a particle of the condensate. 

Since we are only concerned with radial solutions $(F,n)$ of (\ref{PA}), (\ref{PB}), a very natural independent variable is  $x=|p|^2$. But  this introduces a jacobian and then, the most suitable quantity is not always $f(x)=F(p)$ but may be sometimes  $\sqrt x f(x)$.

\subsection{The term $I_4$ and the Nordheim equation}

The local existence of bounded solutions for Nordheim equation (\ref{S1E0a}) was proved in \cite{BE}. Global existence of bounded solutions has been proved in \cite{Lu5} for bounded and suitably small initial data. The existence of radially symmetric weak solutions was first proved in \cite{Lu1} for all initial data $f_0$ in the space of nonnegative radially symmetric measures on $[0, \infty)$. 

For radially symmetric solutions $F(p)=f(x)$, $x=|p|^2$, the expression of the Nordheim equation simplifies because it is possible to perform the angular variables in the collision integral. After rescaling the time variable $t$ (in order to absorb some constants), the Nordheim equation reads:

\begin{align}
&\frac{\partial f}{\partial t}(t, x_1)= J_4(f(t))(x_1),\qquad t>0,\;x_1\geq0, \label{S1E1}
\end{align}
where
\begin{align}
&J_4(f)(x_1)=\iint _{[0,\infty)^2}\hskip -0.4cm  \frac{w(x_1, x_2, x_3)}{\sqrt{x_1}} q(f)(x_1,x_2,x_3)dx_{2}dx_{3}, \label{S1E2}\\ 
&q(f)=(1+  f_{1})(1+f_{2})  f_{3}f_{4}-( 1+ f_{3})(1+f_{4})  f_{1}f_{2},  \label{S1E3}\\
&w(x_1,x_2,x_3)=\min\left\{\sqrt{x_{1}},\sqrt{x_{2}},\sqrt{x_{3}},\sqrt{x_{4}}\right\},\,\,x_{4}=(x_{1}+x_{2}-x_{3})_+. \label{E2'}
\end{align}

The factor $\frac{w}{\sqrt{x_1}}$ in the collision integral comes from the angular integration of the Dirac's delta of the energies $|p _{ \ell }|^2$.

If we denote $\mathscr M _+([0, \infty))$ the space of positive and finite Radon measures on $[0, \infty)$, and define for all $\alpha \in \RR$
\begin{align}
\label{S1E17}
&\mathscr{M}^{\alpha} _+([0, \infty))=\left\{G\in  \mathscr{M}_+([0, \infty)):\; M_{\alpha}(G)<\infty\right\},\\
&M_{\alpha}(G)=\int_{[0,\infty)}x^{\alpha}G(x)dx\qquad\text{(moment of order $\alpha$)},\label{S1E17'}
\end{align}
the definition of weak solution introduced in \cite{Lu1} is the following.

\begin{definition}[Weak radial solutions of (\ref{S1E0a})]
\label{S1D0}
Let $G$ be a  map from $[0, \infty)$ into $\mathscr M_+^1([0, \infty))$ and consider $f$ defined as $\sqrt xf(t)=G(t)$. We say that $f$ is a weak radial solution of (\ref{S1E0a}) if $G$ satisfies:
\begin{align}
&\forall t>0:\,\,\,G(t)\in \mathscr M_+^1([0, \infty)), \label{S1ED1}\\
&\forall  T>0:\,\,\,\sup _{ 0\le t<T } \int  _{ [0, \infty) }(1+x)G(t, x)dx<\infty, \label{S1ED2} \\
& \forall \,\varphi \in C^{1, 1}_b([0, \infty)): \,\, \int  _{ [0, \infty) } \varphi (x)G(t , x) dx\in C^1([0, \infty)),  \label{S1ED3} 
\\
&\frac {d} {dt} \int_{\left[  0,\infty\right)}\varphi(t,x)G(t,x)dx =\mathcal{Q}_4(\varphi,G(t)), \label{S1E4}\\
&\mathcal Q_4(\varphi,G)=\iiint_{\left[  0,\infty\right)^3} \frac{G_{1}G_{2}G_{3}
}{\sqrt{x_{1}x_{2}x_{3}}}w\Delta\varphi \;dx_1dx_2dx_3+ \nonumber \\
&\hskip 2cm +\frac{1}{2}\iiint_{\left[  0,\infty\right)^3} \frac{G_{1}G_{2}}{\sqrt{x_{1}x_{2}}}w \Delta\varphi \;dx_1dx_2dx_3  \label{S1E5}\\
&\Delta\varphi (x_1, x_2, x_3)=\varphi (x_4)+\varphi (x_3)-\varphi (x_2)-\varphi (x_1),\label{S2E1}\\
&w(x_1,x_2,x_3)=\min\{\sqrt{x_1}, \sqrt{x_2}, \sqrt{x_3}, \sqrt{x_4}\},\,\,\,x_4=(x_1+x_2-x_3)_+. \label{S1E6'}
\end{align}
\end{definition}

For all initial data $f_0$ such that $G_0=\sqrt x f_0\in \mathscr M_+^1([0, \infty))$, the existence of a weak solution was proved in \cite{Lu1}.
The moments of order zero and one of $G$ where shown to be constant in time. 
It was shown in \cite{Lu3} that  a definition equivalent to Definition \ref{S1D0} would be to impose $\varphi (0)=0$ to the test functions in Definition \ref{S1D0} and impose the conservation of mass on $G(t)$ for all $t>0$. Further properties of the solutions, such as the gain of moments, asymptotic behavior,   where obtained in a series of articles  \cite{Lu1,Lu2, Lu3, Lu4}

It is proved in Proposition \ref{S1P0} below that if the measure $G$ is written as $G(t)=n(t)\delta _0+g(t)$, where
$n(t)=G(t,\{0\})$, then for all $\varphi \in C^{1,1}_b([0, \infty))$ the term $\mathcal Q_4(\varphi,G)$ may be decomposed as follows:
\begin{align}
\mathcal Q_4(\varphi,G(t))=\mathscr Q _4(\varphi,g(t))+n(t)\mathscr{Q}_3(\varphi,g(t)),\label{S1E13B}
\end{align}
where
\begin{align}
&\mathscr Q_4(\varphi,g)=\iiint_{(0,\infty)^3}\frac{g_{1}g_{2}g_{3}
}{\sqrt{x_{1}x_{2}x_{3}}}w \Delta\varphi \;dx_1dx_2dx_3 \nonumber \\
&\hskip 2cm +\frac{1}{2}\iiint_{(0,\infty)^3}\frac{g_{1}g_{2}}{\sqrt{x_{1}x_{2}}}w \Delta\varphi \;dx_1dx_2dx_3, \label{S1E15}\\
&\mathscr Q _3(\varphi,g)=\mathscr{Q}_3^{(2)}(\varphi,g)-\mathscr{Q}_3^{(1)}(\varphi,g),\label{S1E1Q3}\\
&\mathscr{Q}_3^{(2)}(\varphi,g)=\iint_{(0, \infty)^2} \frac {\Lambda(\varphi)(x, y)} {\sqrt{x y}}g(x)g(y)dxdy,\label{S1E1Q32}\\
&\mathscr{Q}_3^{(1)}(\varphi,g)=\int_{ (0, \infty)}\frac {\mathcal{L}_0(\varphi)(x)} {\sqrt x}g(x)dx, \label{S1E1Q31}\\
&\Lambda(\varphi)(x, y)=\varphi (x+y)+\varphi (|x-y|)-2\varphi (\max\{x, y\}), \label{S1E154}\\
&\mathcal {L}_0(\varphi )(x)=x\big(\varphi (0)+\varphi (x)\big)-2\int _0^x \varphi (y)dy. \label{S1E155}
\end{align}

It was also proved in \cite{Lu1} that as $t\to \infty$, the measure $G$ converges in the weak sense of measures  to one of the measures:
\bear
\label{BED}
G _{ \beta, \mu, C} =\frac {\sqrt x} {e^{\beta x-\mu  }-1}+C\delta _0,\,\,\,\beta >0,\,\,\mu  \le 0, \,\,\,C\ge 0
\eear
where the constants $C$ and $\mu  $ are such that $C\mu  =0$.  

When $C=0$ and $\mu \le 0$, the function  $F_{ \beta , \mu , 0  }(p)=|p|^{-1} G _{ \beta , \mu , 0  }(|p|^2)$ is an equilibrium of the Nordheim equation (\ref{S1E0a}) because
$q(F _{ \beta , \mu , 0 })d\nu \equiv 0$. When $C>0$ and $\mu =0$, then $F _{ \beta , 0, C  }(p)=|p|^{-1} G _{ \beta , 0, C  }(|p|^2)$ is an equilibria of (\ref{S1EA1}) because
$q(f _{ \beta , 0 , 0 }) \equiv 0$ and  $R(p, p', p'')\equiv 0$ for all $(p, p', p'')\in (\RR^3)^3$ for  $f _{ \beta , 0 , 0} $, where  $R(p, p', p'')$ is defined in  (\ref{S1EA4BC}).  It was proved in \cite{Lu1} that $F_{ \beta , \mu, C  }$ is a weak solution of the Nordheim equation  (\ref{S1E1}) if and only if $\mu C=0$. 

On the other hand, it was proved in \cite{EV1} that, given any  $N>0$, $E>0$ there exists initial data $f_{0}\in L^{\infty}\left( 
\RR_+;\left( 1+x\right) ^{\gamma}\right) $ with $\gamma >3$, satisfying
$$
\int_{\mathbb{R}^{+}}f_{0}(x)\sqrt{x }dx=N,\qquad\int_{\mathbb{R}^{+}}f_{0}(x)\sqrt{x^{3}}dx=E,
$$
and such that there exists a global weak solution $f$  and  positive times  $0<T_{\ast}<T^*$ such that:
\bear
\label{S1EC1}
\sup_{0< t\leq T_{\ast}}\left\Vert f\left( t,\cdot\right) \right\Vert
_{L^{\infty}\left( \mathbb{R}^{+}\right) }<\infty,\,\,\,\,\,\sup _{T_*<t\le T^* }\int_{\left\{ 0\right\} }\sqrt{x}f\left( t,x\right)
dx>0.
\eear
Property (\ref{S1EC1}) shows that the  solution $G=\sqrt x f$ of  (\ref{S1ED1})--(\ref{S1E6'}) is a bounded function on the time interval $[0, T^*)$ and a Dirac mass is formed at the origin at some time $T_0$ between $T_*$ and $T^*$. After that time $T_0$, the solution $G$ is such that $G(t, \{0\})>0$. 

In the simplified description of the physical system of particles that we are using, where only the radial density  $G$ of particles of momentum $p$  is considered, the description of the physical Bose-Einstein condensate can just be given by a  Dirac measure at the origin. 

Notwithstanding the similarity of these two phenomena, the extent to which the first one is a truthful mathematical description of the second is not clear.  Nevertheless, we refer to the term $n(t)\delta _0$  that appears in finite time in some of the weak solutions of the Nordheim equation as ``condensate'', with some abuse of language.

\subsection{The term $I_3$ in radial variables.}

The results briefly presented in the previous sub Section describe some of the properties of the weak solutions to the Nordheim equation in terms of the measure $G$.  In particular, the weak convergence of $G$ to the measures defined in (\ref{BED}) shows what is the limit of  $G(t, \{0\})$ as $t\to \infty$.   
To understand better  the dynamics of $G(t, \{0\})\delta _0$ and its interaction with $G(t)-G(t, \{0\})\delta _0$ it seems suitable to  write  $G(t)=G(t, \{0\})\delta _0+g(t)$ and consider the system (\ref{S1EA1}), (\ref{S1EA1B}).  

For radially symmetric functions $F(p)=f(x)$, $x=|p^2|$, the  system (\ref{PA}), (\ref{PB}) reads, after a suitable time rescaling to absorb some constants:
\begin{empheq}[left=\empheqlbrace]{align}
&\frac {\partial f} {\partial t}(t, x)=\frac {n(t)} {\sqrt x} J_3(f(t))(x)\qquad t>0,\;x>0, \label{PR}\\
&n'(t)=-n(t)\int _0^\infty J_3(f(t))(x)dx\qquad t>0, \label{PR19}
\end{empheq}  
where
\begin{align}
&J_3(f)(x)=\int_0^x \Big(f(x-y)f(y)-f(x)\big[1+f(x-y)+f(y)\big]\Big)dy +\nonumber\\
&\hskip 1cm +2\int_x^\infty \Big(f(y)\big[1+f(y-x)+f(x)\big]-f(y-x) f(x)\Big) dy.\label{PR2}
\end{align}
(cf. \cite{ST1} and \cite{Svis} for the isotropic system that also contains the term $J_4(f)$, that comes from $I_4$ in  (\ref{S1EA1})).
Notice that
\begin{align}
&\int _0^\infty J_3(f(t))(x)dx\nonumber\\
&=\int _0^\infty\int _0^\infty \Big(f(t, x)f(t, y)- f(t, x+y)\big[1+f(t, x)+f(t, y)\big]\Big)dxdy \label{S1E9}
\end{align}
whenever the integral in the right hand side is finite, for example, if
$f\in L^1\big(\RR_+, (1+x)dx\big)$. In that case we also have,
\begin{equation}
\int _0^\infty J_3(f(t))(x)dx=M _1(f(t)).
\end{equation}

The factor $x^{-1/2}$ in the right hand side  of (\ref{PR}) comes from the angular integration of the Dirac's measure of energies of $I_3$, just as the 
$\frac{w}{\sqrt{x_1}}$ term of (\ref{S1E2}) in $I_4$.  But since $\frac{w}{\sqrt{x_1}}$ is a bounded function,
it appears that the operator $I_3$ is more singular than $I_4$ for small values of $x$. 

If we denote  $F(t, p)=f(t, |p|^2)=|p|^{-1}g(t, |p|^2)$ and $x=|p^2|$, from the original motivation of the Nordheim equation  it is very natural to expect
$$
\int  _{ \RR^3 }F(t, p)dp= 2\pi\int _0^\infty f(t, x)\sqrt x dx= 2\pi\int_0^\infty g(t, x)dx<\infty, 
$$
(that corresponds to the  number of particles in the normal fluid), and
$$
\int  _{ \RR^3 }F(t, p)|p|^2dp= 2\pi\int _0^\infty f(t, x)x^{3/2}dx=2\pi\int_0^\infty g(t, x)xdx<\infty,
$$
(corresponding to the total energy in the system).  But there is no particular reason to expect 
$$
\int  _{ \RR^3 }F(t, p)\frac {dp} {|p|}= 2\pi\int _0^\infty f(t, x)dx=2\pi\int_0^\infty g(t, x)\frac {dx} {\sqrt x}<\infty.
$$
Without that last condition, the convergence of the integrals in the term $I_3(F(t))$ (cf. (\ref{E1BCD}), (\ref{S1EA4BC})), or in (\ref{PR}), (\ref{PR2}), is delicate.
That difficulty is usually avoided using a suitable weak formulation. 

If we suppose that  $f=x^{-1/2}g\in L^1\big(\RR_+, (1+x)dx\big)$, and multiply the equation (\ref{PR}) by $\sqrt x\, \varphi $, we obtain by Fubini's Theorem, 
\begin{equation}
\label{g11}
\frac {d} {dt}\int_{ [0, \infty) } \varphi (x) g(t, x)dx=n(t)\widetilde{\mathscr Q}_3(\varphi,g(t))\quad\forall  \varphi \in C^1_b([0, \infty)),
\end{equation}
where
\begin{align}
&\widetilde{\mathscr Q}_3(\varphi,g)=\mathscr{Q}_3^{(2)}(\varphi,g)-\widetilde{\mathscr Q}_3^{(1)}(\varphi,g),\label{S1EB2}\\
&\widetilde{\mathscr{Q}}_3^{(1)}(\varphi,g)=\int_{ (0, \infty)} \frac {\mathcal{L}(\varphi )(x)} {\sqrt x}g(x)dx, \label{S1E20R}\\
&\mathcal{L}(\varphi)(x)=x\varphi (x)-2\int _0^x \varphi (y)dy. \label{S1E21R}
\end{align}
Notice that, by (\ref{S1E1Q3}),
\begin{align}
\mathscr{Q}_3(\varphi,g)=\widetilde{\mathscr{Q}}_3(\varphi,g)-\varphi(0)M_{1/2}(g).\label{S1EB1}
\end{align}

A natural weak formulation for $G=n(t)\delta _0+g$ is then obtained by adding (\ref{PR19}) to (\ref{g11}). 
We  then define a weak radially symmetric solution of the Problem (\ref{PA}), (\ref{PB}) as follows.
\begin{definition}[Weak radial solution of (\ref{PA}), (\ref{PB})]
\label{S1D1}
Consider a map $G:[0, T) \to \mathscr M_+^1([0, \infty))$  for some $T\in (0, \infty]$, that we decompose as follows:
$$
\forall t\in [0, T):\,\,\,\,\,G(t)=n(t)\delta _0+g(t), \,\,\,\, \hbox{where}\,\,\,\, n(t)=G(t, \{0\});
$$
and  define  $ F(t, p)=|p|^{-1}g(t, |p|^2)$ for all $t>0$ and $p\in \RR^3$. We say that $(F, n)$ is a weak radial solution of (\ref{PA}), (\ref{PB}) on $(0, T)$ if:

\begin{align}
&\forall T'\in (0, T]:\qquad\sup _{ 0\le t<T' } \int  _{ [0, \infty) }(1+x)G(t, x)dx<\infty, \label{S1ED2S} \\
& \forall \varphi \in C^{1}_b([0, \infty)): \quad t\mapsto \int  _{ [0, \infty) } \varphi (x)G(t, x) dx \in W^{1,\infty } _{ loc }([0, T)),  \label{S1ED3S}
\end{align}
and for a.e. $t\in (0, T)$
\begin{equation}
\frac {d} {dt}\int_{ [0, \infty) } \varphi (x) G(t, x)dx=n(t) \mathscr{Q}_3(\varphi,g(t))\quad\forall  \varphi \in C^1_b([0, \infty)), \label{S1E16}\\
\end{equation}
where $\mathscr Q _3(\varphi,g)$ is defined by (\ref{S1E1Q3})-(\ref{S1E1Q31}).
\end{definition}

We  show in Proposition \ref{S1P0} that the Definition \ref{S1D1} substantially  coincides with the Definition \ref{S1D0} of radial weak solution of (\ref{S1E0a}) when the term $\mathscr Q_4(\varphi,g)$ in (\ref{S1E13}) is dropped (cf. Remark \ref{SXR1}). As a  consequence, the measures $f _{ \beta , 0, C  }(p)$
defined above are weak radial solutions of (\ref{PA}), (\ref{PB}) (cf. Proposition \ref{equilibria}).

\subsection{Main results}
The existence of weak radial solutions for the Cauchy problem associated with the system (\ref{PA}), (\ref{PB}) is given in the following Theorem.

\begin{theorem}[Existence result]
\label{S1T1}
Suppose that $G_0\in \mathscr{M}^1_+([0,\infty))$ satisfies $G_0(\{0\})>0$, and define $F_0(p)=|p|^{-1}g_0(|p|^2)$, where $g_0=G_0-G_0(\{0\})\delta _0$. Then, there exists  a weak radial solution $(F, n)$ of  (\ref{PA}), (\ref{PB}) on $(0, \infty)$ such that $F(t, p)=|p|^{-1}g(t, |p|^2)$, where $G=n\delta _0+g$  satisfies:
\bear
\label{S1T1E0}
G\in C\big([0,\infty),\mathscr{M}_+^1([0,\infty))\big),\quad G(0)=G_0
\eear 
and:
\begin{enumerate}[(i)]
\item$G$ conserves the total number of particles $N$ and energy $E$:
\begin{align}
&M_0(G(t))=M_0(G_0)=N\qquad\forall t\geq 0,\label{S1E210}\\
&M_1(G(t))=M_1(G_0)=E\qquad\forall t\geq 0.\label{S1E220}
\end{align}

\item For all $\alpha \geq 3$,  if $M_\alpha (G_0)<\infty$, then $G\in C \big((0, \infty), \mathscr{M}_+^{\alpha}([0,\infty))\big)$ and 
\begin{flalign}
\label{S1E23}
&M_{\alpha}(G(t))\le \left(M_{\alpha}(G_0)^{\frac{2}{\alpha-1}}
+\alpha2^{\alpha-1}E^{\frac{\alpha+1}{\alpha-1}}\tau(t)\right)^{\frac{\alpha-1}{2}}\quad\forall t>0,&\\
&\text{where}\quad\tau (t)=\int _0^t G(s, \{0\})ds.& \label{E657tau}
\end{flalign}

\item For all $\alpha \geq3$,
\begin{equation}
\label{S5Ealphahh}
\quad M_{\alpha}(G(t))\leq C(\alpha,E)\left(\frac{1}{1-e^{-\gamma(\alpha,E)\tau(t)}}\right)^{2(\alpha-1)}\quad\forall t >0,
\end{equation}
where $\tau(t)$ is given by (\ref{E657tau}), and the constants $C(\alpha,E)$ and $\gamma (\alpha,E)$ are defined in Theorem \ref{S5T5R}.

\item If $\alpha \in (1, 3]$ and
\begin{align}
E> C(\alpha)N^{5/3},\label{PRO112}
\end{align}
\begin{flalign}
&\text{where}\qquad
\label{PRO115}
C(\alpha)=
\begin{cases}
\Big(\frac{(2^{\alpha}-2)(\alpha+1)}{(\alpha-1)}\Big)^{\frac{2}{3}}&\text{if}\quad\alpha\in(1,2],\\
\big(\alpha(\alpha+1)\big)^{\frac{2}{3}}&\text{if}\quad\alpha\in(2,3],
\end{cases}
&
\end{flalign}
then $M_{\alpha}(G(t))$ is a decreasing function on $(0,\infty)$.
\end{enumerate}
\end{theorem}

The next result is a property satisfied by all  the weak radial solutions of  (\ref{PA}), (\ref{PB}).
\begin{theorem}
\label{S1Treg}
Let $G_0$ be  as in Theorem \ref{S1T1}, and $G$ a weak radial solution of  (\ref{PA}), (\ref{PB}).
Then for all $T>0$, $R>0$ and $\alpha\in\left(-\frac{1}{2},\infty\right)$,
\begin{align}
&\int_0^TG(t,\{0\})\int_{(0,R]}x^{\alpha}G(t,x) dxdt\leq \nonumber \\
&\,\,\, \le \frac{2R^{\frac{1}{2}+\alpha}}{1-\left(\frac{2}{3}\right)^{\frac{1}{2}+\alpha}}\left(\int_0^T\!\! G(t,\{0\})dt\right)^{\frac{1}{2}}\left(\frac{\sqrt{E}}{2}\int_0^T\!\!G(t,\{0\})dt+\sqrt{N}\right).\label{MNEG2}
\end{align}
\end{theorem}
The only
possible algebraic behavior for such a measure $G$ near the origin is then $x^{-1/2}$.
\begin{remark}
\label{S1R1}
The functions $F _{ \beta , 0, C }$ defined above are weak radial solutions of (\ref{PA}), (\ref{PB}) for all $\beta >0$ and $C\ge 0$ (cf. Proposition \ref{equilibria}). Since
\bean
\int  _{ (0, \infty) } x^{\alpha}  G _{ \beta , 0, C}\;dx<\infty\quad\Longleftrightarrow\quad\alpha >-1/2,
\eean
the estimate (\ref{MNEG2}) can not hold for all radial weak solutions if $\alpha\leq -1/2$.
\end{remark}

In the next two results we describe the evolution of the measure at the origin $n(t)=G(t,\{0\})$ by taking the limit $\varepsilon\to 0$ in the weak formulation 
(\ref{S1E16}) for test functions $\varphi_{\varepsilon}$ as follows:
\begin{remark}
\label{TEST}
Given $\varphi\in C^1_b([0,\infty))$ nonnegative, convex, with $\varphi(0)=1$ and 
$\lim_{x\to\infty}\sqrt{x}\varphi(x)=0$, denote $\varphi_{\varepsilon}(x)=\varphi(x/\varepsilon)$ for $\varepsilon>0$. Notice that for any 
$G\in\mathscr{M}_+([0,\infty))$,
\begin{align}
\label{TEST2}
G(\{0\})=\lim_{\varepsilon\to 0}\int_{[0,\infty)}\varphi_{\varepsilon}(x)dG(x).
\end{align}
The standard example is $\varphi_{\varepsilon}(x)=(1-x/\varepsilon)^2_+$.
\end{remark}

\begin{theorem}
\label{THn01}
Let $G$ be the solution of (\ref{S1E16}) obtained in Theorem \ref{S1T1}, with initial data $G_0\in \mathscr{M}^1_+([0,\infty))$ such that $N=M_0(G_0)>0$, $E=M_1(G_0)>0$ and $G_0(\{0\})>0$. Denote $G(t)=n(t)\delta _0+g(t)$, with
$n(t)=G(t, \{0\})$. Then $n$ is right continuous and a.e. differentiable on $[0,\infty)$. Moreover, there exists a positive measure $\mu$ on $[0,\infty)$ whose cumulative distribution function is given by
\begin{align}
\label{ZE01}
\mu((0,t])=\lim_{\varepsilon\to 0}\int_{0}^{t}n(s)\mathscr{Q}_3^{(2)}(\varphi_{\varepsilon},g(s))ds
\end{align}
for any $\varphi_{\varepsilon}$ as in Remark \ref{TEST},
and  such that:
\begin{align}
\label{ZE02}
n(t)-n(0)+\int_{0}^{t}n(s)M_{1/2}(g(s))ds=\mu((0,t])\qquad\forall t>0.
\end{align}
\end{theorem}

\begin{theorem}
\label{MU1}
Let $G$ and $\mu$ be as in Theorem \ref{THn01}. Then
\begin{align}
\label{ZE00}
0<\mu((0,t]))<\infty\qquad\forall t>0.
\end{align}
\end{theorem}
The measure $\mu$ in (\ref{ZE01}) depends on the atomic part of  $g$, and on the behaviour of $g$ at the origin (it seems to be actually related with its moment of order $-1/2$ c.f. Proposition \ref{LM-1/2} and Remark \ref{HH}). This measure $\mu $ appears as a source term in the equation (\ref{ZE02}) for $n$.
Given  the function $n$, the equation (\ref{PR}) satisfied by $g$ on $(0, \infty)$ has also a natural weak formulation by itself. In terms of $g(t)$, where $g(t)=G(t)-G(t, \{0\})\delta _0$ and $\sqrt x f(t, x)=g(t, x)$ it reads 

\bear
\label{2E765}
\frac{d}{dt}\int_{[0,\infty)} \!\!\!\!\!\varphi(x)g(t,x)dx=n(t)\mathscr{Q}_3(\varphi,g(t)),\;\forall \varphi\in C_b^1([0, \infty)),\,\varphi (0)=0.
\eear

In the next result we describe the relation between a weak solution $(F,n)$ of (\ref{PA}), (\ref{PB}), where $F(t, p)=|p|^{-1}g(t, |p|^2) $, $G(t)=n(t)\delta _0+g(t)$, $n(t)=G(t, \{0\})$, and a pair $(g, n)$ where $g$ is a weak radial solution of the equation (\ref{PA}) and $n$ satisfies (\ref{PB}). 
\begin{theorem}
\label{EQUIV}
Suppose that $G\in C\big([0,\infty),\mathscr{M}_+([0,\infty))\big)$ is such that $G(0)=G_0\in \mathscr{M}^1_+([0,\infty))$ with $G_0(\{0\})>0$, and denote  $G(t)=n(t)\delta _0+g(t)$ with $n(t)=G(t, \{0\})$.\\
(i) If $(F, n)$  is a weak radial solution of (\ref{PA}), (\ref{PB}) and $F(t, p)=|p|^{-1}g(t, |p|^2)$, then $n$ is given by (\ref{ZE02}), (\ref{ZE01}),
and $g$ satisfies (\ref{2E765}) for a.e. $t>0$.\\
(ii) On the other hand, if  $g$ satisfies (\ref{S1ED2S}), (\ref{S1ED3S}) and (\ref{2E765}) for some nonnegative bounded function $n$,  
then the limit in (\ref{ZE01}) exists. If $n$ also satisfies
\begin{equation}
\label{ene1}
n(t)=n(0)+\lim_{\varepsilon\to 0}\int_0^t n(s)\mathscr{Q}_3^{(2)}(\varphi_{\varepsilon},g(s))ds-\int_0^t n(s)M_{1/2}(g(s))ds
\end{equation}
and $F(t, p)=|p|^{-1}g(t, |p|^2)$, then $(F, n)$ is a weak radial solution of (\ref{PA}), (\ref{PB}).
\end{theorem}

If in the Definition \ref{S1D1} only test functions satisfying $\varphi (0)=0$ are taken, it becomes necessary to introduce some other condition to the system. Otherwise the system would be reduced to  find $g$ satisfying (\ref{S1ED3S})--(\ref{S1E16})  for a given function $n(t)$ and  for test functions such that  $\varphi (0)=0$. If we impose just the conservation of mass, we prove below (Corollary \ref{S1C31}) that we recover a solution that satisfies the Definition \ref{S1D1}.

\begin{corollary}
\label{S1C31}
If $g$ satisfies (\ref{S1ED2S}), (\ref{S1ED3S}) and (\ref{2E765}) for some nonnegative bounded function $n=n(t)$ such that 
\begin{equation}
\label{MBC}
n(t)+\int  _{ (0, \infty) }g(t, x)dx= constant
\end{equation}
and $F(t, p)=|p|^{-1}g(t, |p|^2)$, then $(F, n)$ is a weak radial solution of (\ref{PA}), (\ref{PB}).
\end{corollary}

In our last  result we show that, under some sufficient conditions, the condensate density $n(t)$ tends to zero as $t\to\infty$, fast enough to be  integrable.
\begin{theorem} 
\label{S1T5}
Suppose that  $G_0\in \mathscr M_+^1([0, \infty))$  satisfies $G_0(\{0\})>0$ and let $(F,n)$ be the weak radial solution of (\ref{PA}), (\ref{PB}) obtained in Theorem \ref{S1T1}. Let us call $N=M_0(G_0)$ and $E=M_1(G_0)$. If condition (\ref{PRO112}), (\ref{PRO115}) hold for some $\alpha \in (1, 3]$,
then, for all $t_0>0$,
\begin{equation}
\label{S1ET5B}
\int _{ t_0 }^{\infty} n(t)dt\leq M_{\alpha}(G(t_0)) C(N,E,\alpha)
\end{equation}
for some explicit constant $C(N,E,\alpha)$ given in (\ref{PRO116}), and
\begin{equation}
\label{S1ET5C}
\lim_{t\to \infty }n(t)=0.
\end{equation}
\end{theorem}

\begin{remark}
\label{S1R2}
The quantity $E/N^{5/3}$ has a very precise  interpretation in physical terms.  Suppose that $T$ is the temperature of  a system of  particles at equilibrium with total number of particles $N$ and total energy $E$. And denote $T_c$ the critical temperature, that is the temperature at which the ground state of the system becomes macroscopically occupied. Then:
\bean
\frac {E} {N^{5/3}} = b\, \frac {T} {T_c},\qquad\hbox{where}\qquad
b=\frac {3}{(2\pi )^{\frac {1} {3}}}\frac  {\zeta (5/2)}{\zeta (3/2)^{5/3}}.
\eean
and  condition  (\ref{PRO112}) implies
$$
\frac {T} {T_c}=\frac {1} {b}\,\frac {E} {N^{5/3}}> \frac {C(\alpha )} {b}.
$$
The function $C(\alpha )/b$ is continuous and strictly increasing on $[1, 3]$ and its limit as $\alpha \to 1^+$ is $\log(16)^{2/3}/b\approx 4.48403$. Condition (\ref{PRO112}) means that, when at equilibrium, the  system of particles would be at a temperature  clearly above the critical temperature. Anyway,  the solution $F$ of the problem (\ref{PA}), (\ref{PB}) may be far from any real distribution of particles of the original system of particles.
\end{remark}

\subsection{Some arguments of the proofs.}
It is very natural to make the following change of variables in  problem (\ref{S1E16}). Given $G(t)=n(t)\delta _0+g(t)$, where $n(t)=G(t, \{0\})$, we define
\bear
\label{S1E45}
&&H(\tau)=G(t),\qquad\text{where}\qquad\tau =\int _0^t n(s) ds. \label{S1E45b}
\eear
In terms of $H$, (\ref{S1E16}) reads
\begin{equation}
\frac {d} {d\tau }\int_{ [0, \infty) } \varphi (x)H(\tau , x)dx= \mathscr Q _3(\varphi,H(\tau))\quad\forall \varphi \in C^1_b([0, \infty)). \label{S1E16Ha}
\end{equation}
To obtain a measure $H$ that satisfies (\ref{S1E16Ha}), we first find $h$ satisfying
\begin{align}
\frac {d} {d\tau }\int_{ [0, \infty) } \varphi (x)h(\tau , x)dx= \widetilde{\mathscr Q}_3(\varphi,h(\tau))
\quad\forall \varphi \in C^1_b([0, \infty)),\label{S1E16ha}
\end{align}
where $\widetilde{\mathscr Q}_3$ is given in (\ref{S1EB2})--(\ref{S1E21R}).
Then we define $H$ as
\bear
\label{S1EdecompH}
H(\tau)=h(\tau)-\left(\int_0^{\tau}M_{1/2}(h(\sigma))d\sigma\right)\delta_0.
\eear
By (\ref{S1EB1}), the measure  $H$ will satisfy (\ref{S1E16Ha}).

As it will be seen in Section \ref{existenceH}, all the arguments are much simpler and clear in the equation for $H$ than in the equation for $G$. In particular, the measure $\lambda$, that corresponds to the measure $\mu$ of Theorem \ref{THn01}, appears as the Lebesgue-Stieltjes measure associated to $m(\tau )=h(\tau , \{0\})$.

The proofs of Theorem \ref{S1T1} and Theorem \ref{S1Treg} make great use of the change of variables (\ref{S1E45b}).
Several of our arguments will need the measure $h(\tau )$ to satisfy only one inequality in (\ref{S1E16ha}). This requires the following:
\begin{definition}
\label{SUPER}
A  function $h:[0,\infty)\to\mathscr{M}_+([0,\infty))$ is said to be a super solution of (\ref{S1E16ha}) if
\begin{empheq}[left=\empheqlbrace]{align}
&\forall  \varphi \in C^1_b([0, \infty))\;\hbox{nonnegative, convex and decreasing} :\nonumber\\
&\frac {d} {d\tau }\int_{ [0, \infty) } \varphi (x)h(\tau,x)dx\ge \mathscr{Q}_3^{(2)}(\varphi,h(\tau))\qquad a.e.\;\tau>0.\label{S1E16haB}
\end{empheq}
\end{definition}
The operator $\mathscr Q_3^{(2)}$ is considered in \cite{KIER} and \cite{AV1}, where a problem similar to (\ref{S1E16ha}) is studied, with
$\widetilde{\mathscr{Q}}_3$  replaced by $\mathscr Q_3^{(2)}$ and for which, the property of instantaneous condensation is proved. We extend this result to the solutions $h$ of the problem (\ref{S1E16ha}) with the whole $\widetilde{\mathscr{Q}}_3$, using similar arguments (monotonicity,  convexity of  test functions) and taking care of the linear term.
 
Theorem \ref{S1T1} is deduced from the corresponding existence result of $h$, that is proved using very classical arguments: regularization of the problem,  fixed point, a priori estimates and passage to the limit. Then, the delicate point is to invert the change of variables (\ref{S1E45b}) in order to obtain a global in time nonnegative solution $G$.

The Plan of the article is the following. In Section \ref{model} we prove Proposition \ref{S1P0}. Section \ref{existenceH} is devoted to the proof of the existence of the measure $H$. In Section \ref{SectionC} we obtain several properties of $h(\tau , \{0\})$. In Section \ref{sectionG} we prove Theorem \ref{S1T1} (existence for the measure $G$) and Theorem \ref{S1Treg}. The contents of Section \ref{SectionK} are the
proofs of Theorem \ref{THn01}, Theorem \ref{MU1}, Theorem \ref{EQUIV} and Corollary \ref{S1C31}. Finally in Section \ref{SectionD} we prove Theorem \ref{S1T5}. Several technical results are presented in an Appendix.

\section{On weak formulations.}
\label{model}

\setcounter{equation}{0}
\setcounter{theorem}{0}

We  deduce first a  detailed  expression of  the weak formulation of (\ref{S1E1}) for a radial measure $G$.
\begin{proposition}
\label{S1P0}
Let $G$ satisfy (\ref{S1ED1})--(\ref{S1E6'}) for some $T>0$, and write 
$G(t)=n(t)\delta_0+g(t)$, where $n(t)=G(t,\{0\})$. Then, for all $\varphi \in C^{1,1}_b([0, \infty))$ and for all $t\in (0, T)$:
\begin{align}
&\frac {d} {dt}\int _{ [0, \infty) } \varphi (x)G(t, x)dx=\mathscr{Q}_4(\varphi,g(t))+n(t)\mathscr{Q}_3(\varphi,g(t)),\label{S1E13}
\end{align}
where $\mathscr Q_4(\varphi,g)$ and $\mathscr Q _3(\varphi,g)$ are defined in (\ref{S1E15})--(\ref{S1E155}).
\end{proposition}

\begin{remark}
\label{SXR1}
If the term $\mathscr Q_4(\varphi,g)$ in (\ref{S1E13}) is dropped, we recover the equation (\ref{S1E16})  that defines a radial weak solution of (\ref{PA}), (\ref{PB}).
\end{remark}

\begin{proof}
[\upshape\bfseries{Proof of Proposition \ref{S1P0}}]
We may rewrite $\mathcal{Q}_4(\varphi,G)$ in (\ref{S1E5}) as 
\begin{align*}
\mathcal{Q}_4(\varphi,G)&=\iiint_{[0,\infty)^3}\Phi_{\varphi}\;dG_1dG_2dG_3
+\frac{1}{2}\iiint_{[0,\infty)^3}\sqrt{x_3}\Phi_{\varphi}\;dG_1dG_2dx_3,
\end{align*}
where $\Phi_{\varphi}$ is as in Lemma \ref{S2L1}, and we have used notation $dG$ instead of $Gdx$. Then we decompose 
$[0,\infty)^3=(0,\infty)^3\cup A\cup P$, where, for $\{i,j,k\}=\{1,2,3\}$,
\begin{align*}
&A=\{(x_1,x_2,x_3)\in\partial[0,\infty)^3\;:\; x_i=x_j=0,\; x_k> 0\}\cup\{(0,0,0)\},\\
&P=\{(x_1,x_2,x_3)\in\partial[0,\infty)^3\;:\; x_i=0,\; (x_j,x_k)\in (0,\infty)^2\}.
\end{align*}
Let $\varphi\in C^{1.1}_b([0,\infty)$. By (\ref{S2E2}) in Lemma \ref{representation of Deltavarphi} and the definition (\ref{S2E3}) of $W$, it follows that $\Phi_{\varphi}\equiv 0$ on $A$. Hence, recalling the definition (\ref{S1E15}) of $\mathscr{Q}_4(\varphi,g)$ and the definition of $\Phi_{\varphi}$ in Lemma \ref{representation of Deltavarphi}, we have
\begin{align}
\label{999C}
\mathcal{Q}_4(\varphi,G)=\mathscr{Q}_4(\varphi,g)+\iiint_{P}\Phi_{\varphi}\;dG_1dG_2dG_3
+\frac{1}{2}\iiint_{P}\sqrt{x_3}\Phi_{\varphi}\;dG_1dG_2dx_3.
\end{align}
We now study the integral over $P$ for the cubic and the quadratic terms in (\ref{999C}).\\
(a) The cubic term.
Since $\Phi_{\varphi}$ is symmetric in the $x_1$, $x_2$ variables, and $\Phi_{\varphi}$ is uniformly continuous on $[0,\infty)^3$ by Lemma
\ref{S2L1}, then
\begin{align}
\iiint_{P}\Phi_{\varphi}\;dG_1dG_2dG_3=&2 \iiint_{\{x_2=0,\;x_1>0,\,x_3>0\}}\Phi_{\varphi}\;dG_1dG_2dG_3\nonumber\\
&+\iiint_{\{x_3=0,\;x_1>0,\,x_2>0\}}\Phi_{\varphi}\;dG_1 dG_2 dG_3\nonumber\\
=&2G(t,\{0\})\iint_{(0,\infty)^2}\Phi_{\varphi}(x_1,0,x_3)\;dG_1dG_3\nonumber\\
&+G(t,\{0\})\iint_{(0,\infty)^2}\Phi_{\varphi}(x_1,x_2,0)\;dG_1 dG_2.\label{mid step}
\end{align}
Using now the definition of $\Phi_{\varphi}$, we have
\begin{align}
\label{line 1}
&2\iint_{(0,\infty)^2}\Phi_{\varphi}(x_1,0,x_3)\;dG_1dG_3\\
&=2\iint_{\{x_1>x_3>0\}}\big[\varphi(x_1-x_3)+\varphi(x_3)-\varphi(0)-\varphi(x_1)\big] \frac{dG_1dG_3}{\sqrt{x_1x_3}}\nonumber\\
&=\iint\limits_{(0,\infty)^2}\big[\varphi(|x_1-x_3|)+\varphi(\min\{x_1,x_3\})-\varphi(0)-\varphi(\max\{x_1,x_3\})\big] \frac{dG_1dG_3}{\sqrt{x_1x_3}}.\nonumber
\end{align}
and
\begin{align}
\label{line 2}
&\iint_{(0,\infty)^2}\Phi_{\varphi}(x_1,x_2,0)\,dG_1 dG_2\\
&=\iint\limits_{(0,\infty)^2}\big[\varphi(x_1+x_2)+\varphi(0)-\varphi(\min\{x_1,x_2\})-\varphi(\max\{x_1,x_2\})\big]\frac{dG_1 dG_2}{\sqrt{x_1x_2}}.\nonumber
\end{align}
Notice in (\ref{line 1}) that $\varphi(|x_1-x_3|)+\varphi(\min\{x_1,x_3\})-\varphi(0)-\varphi(\max\{x_1,x_3\})=0$ on the diagonal $\{x_1=x_3>0\}$.
Then, using (\ref{line 1}) (changing the labels $x_3$ by $x_2$) and (\ref{line 2}) in (\ref{mid step}), and recalling the definition (\ref{S1E154}) of $\Lambda(\varphi)$, we obtain
\begin{align}
\label{999A}
\iiint_{P}&\Phi_\varphi\;dG_1dG_2dG_3=G(t,\{0\})\iint_{(0,\infty)^2}\frac{\Lambda(\varphi)(x_1,x_2)}{\sqrt{x_1x_2}}dG_1 dG_2.
\end{align}
(b) The quadratic term.
Again, by the symmetry of $\Phi_{\varphi}$ in $x_1$, $x_2$, and the continuity of $\Phi_{\varphi}$ on $[0,\infty)^3$, we obtain
\begin{align}
\label{999B}
\frac{1}{2}\iiint_{P}\sqrt{x_3}\,\Phi_{\varphi}\;dG_1dG_2d x_3
&=\iiint_{\{x_2=0,\;x_1>0,\,x_3>0\}}\sqrt{x_3}\,\Phi_{\varphi} \,dG_1dG_2d x_3\nonumber\\
&=G(t,\{0\})\iint_{(0,\infty)^2}\sqrt{x_3}\,\Phi_{\varphi}(x_1,0,x_3)\,dG_1d x_3\nonumber\\
&=G(t,\{0\})\iint_{\{x_1>x_3>0\}}\frac{\Delta\varphi(x_1,0,x_3)}{\sqrt{x_1}}dG_1d x_3\nonumber\\
&=-G(t,\{0\})\int_{(0,\infty)}\frac{\mathcal{L}_0(\varphi)(x_1)}{\sqrt{x_1}}dG_1,
\end{align}
where $\mathcal{L}_0(\varphi)$ is given in (\ref{S1E155}).
Using (\ref{999A}) and (\ref{999B}) in (\ref{999C}), the result follows.
\end{proof}

\begin{proposition}
\label{equilibria}
For all $C> 0$ and all $\beta >0$, the measure $f _{ \beta , 0, C }$ is a radial weak solutions of (\ref{PA}),(\ref{PB}).
\end{proposition}

\begin{proof}
By Proposition \ref{S1P0},
\bean
\mathcal Q_4(\varphi,G _{ \beta , 0, C })=\mathscr{Q}_4(\varphi,G _{ \beta , 0, 0 })+C\mathscr{Q}_3(\varphi,G _{ \beta , 0, 0 }).
\eean
We already know by Theorem  5 of \cite{Lu1} that
$\mathcal{Q}_4(\varphi,G _{ \beta , 0, C })=0$ for all  $\varphi \in  C^{1,1}([0, \infty))$. 
Since $\mathscr{Q}_4(\varphi,G _{ \beta , 0, 0 })\equiv \mathcal{Q}_4(\varphi,G _{ \beta , 0, 0 })$, we deduce
$\mathscr{Q}_4(\varphi,G _{ \beta , 0, 0 })=0$ for all $\varphi \in  C^{1,1}([0, \infty))$. Then, since $C>0$,
$$\mathscr{Q}_3(\varphi,G _{ \beta , 0, 0 })=0\quad\forall \varphi \in  C^{1,1}([0, \infty)).$$
\end{proof}

\section{Existence of solutions $H$ to (\ref{S1E16Ha})}
\label{existenceH}
\setcounter{equation}{0}
\setcounter{theorem}{0}

The main result of this Section is the following,
\begin{theorem}
\label{S5T5R}
Let $h_0\in\mathscr{M}^1_+([0,\infty))$ with $N=M_0(h_0)>0$ and $E=M_1(h_0)>0$. Then, there exists $h\in C\big((0,\infty), \mathscr{M}_+^\alpha([0,\infty))\big)$ for any $\alpha\geq1$,  
that satisfies the following properties: for all $\varphi\in C^1_b([0,\infty)$
\begin{align}
\label{lip loc h}
(i)\quad&  \tau\mapsto\int_{[0,\infty)}\varphi(x)h(\tau,x)dx\in W^{1,\infty} _{loc}([0,\infty)),\\
\label{AUXW}  
(ii)\quad&\frac{d}{d \tau}\int_{[0,\infty)}\varphi(x)h(\tau,x)\,dx=\widetilde{\mathscr{Q}}_3(\varphi,h(\tau))\quad a.e.\,\tau>0,\\
(iii)\quad&  h(0)=h_0,\\
\label{MMI}
(iv)\quad& M_0(h(\tau))\le\bigg(\frac{\sqrt{E}}{2}\tau+\sqrt{N}\bigg)^2\quad\forall\tau \ge 0,\\
 \label{EE}
(v)\quad&M_1(h(\tau))=E\quad\forall\tau \ge 0,\\
(vi)\quad& \text{For all }\alpha \geq 3,\,\,\hbox{if}\,\,M_{\alpha}(h_0)<\infty,\,\,\hbox{then}  \nonumber\\
\label{MAh}
&\quad M_{\alpha}(h(\tau))
\leq \left(M_{\alpha}(h_0)^{\frac{2}{\alpha-1}}+\alpha2^{\alpha-1}E^{\frac{\alpha+1}{\alpha-1}}\tau\right)^{\frac{\alpha-1}{2}}\quad\forall\tau\geq 0,\\
\label{S5Ealpha }
(vii)\quad& M_{\alpha}(h(\tau))\leq C(\alpha,E)\left(\frac{1}{1-e^{-\gamma(\alpha,E)\tau}}\right)^{2(\alpha-1)}\,\,\forall \alpha\geq 3,\;\forall\tau>0,
\end{align}
where $C=C(\alpha,E)$ is the unique positive root of the algebraic equation
\begin{equation}
\label{algebraic equation}
2^{\alpha-2}(\alpha+1)E^{\frac{2\alpha+3}{2(\alpha-1)}}(1+C)=C^{\frac{2\alpha-1}{2(\alpha-1)}},
\end{equation}
and $\gamma=\gamma(\alpha,E)$:
\begin{equation}
\label{gamma}
\gamma=\frac{1}{2(\alpha+1)}\left(\frac{C}{E}\right)^{\frac{1}{2(\alpha-1)}}.
\end{equation}

\end{theorem}
The proof of Theorem \ref{S5T5R} is in three steps. In the first, a regularized problem is solved (Theorem \ref{Ex1T2}). Then, using an approximation argument, a solution is obtained that satisfies (\ref{lip loc h})--(\ref{MAh}) but not yet (\ref{S5Ealpha }) (Theorem \ref{Ex1T1}). The Theorem  \ref{S5T5R} is proved with  a second approximation argument on the initial data.

As a Corollary, we obtain the measure $H$ (not necessarily positive). 

\begin{corollary}
\label{S5C52R}
Suppose that $h_0\in\mathscr{M}^1_+([0,\infty))$ with $N=M_0(h_0)>0$ and $E=M_1(h_0)>0$, consider $h$ given by Theorem \ref{S5T5R},
and define, for $\tau\geq 0$
\begin{align}
\label{DEFH}
H(\tau)=h(\tau)-\left(\int_0^\tau M_{1/2}(h(\sigma))d\sigma\right)\delta_0.
\end{align}
Then $H\in C\big([0,\infty),\mathscr{M}^1([0,\infty))\big)$ and for all $\tau\in[0,\infty)$ and $\varphi\in C^1_b([0,\infty))$:
\begin{align}
\label{lip loc H}
(i)\quad&\tau\mapsto\int_{[0,\infty)}\varphi(x)H(\tau,x)dx\in W^{1,\infty} _{loc}([0,\infty)),\\
(ii)\quad&
\label{AUXWH}
\frac{d}{d \tau}\int_{[0,\infty)}\varphi(x)H(\tau,x)\,dx=\mathscr{Q}_3(\varphi,H(\tau))\quad a.e.\,\tau>0,\\
\label{S5IDh}
(iii)\quad& H(0)=h_0,\\
(iv)\quad& \label{MMH}M_0(H(\tau))=N\quad\forall \tau \ge 0,\\
(v)\quad& \label{EEH} M_1(H(\tau))=E\quad\forall \tau \ge 0,\\
(vi)\quad& \forall \alpha \geq3, \,\,\hbox{if}\,\,  M_{\alpha}(h_0)<\infty\,\hbox{ then, for all}\,\tau >0, \nonumber\\
&\label{MAH} \hskip 1.5cm  M_{\alpha}(H(\tau))\leq \left(M_{\alpha}(h_0)^{\frac{2}{\alpha-1}}+\alpha2^{\alpha-1}E^{\frac{\alpha+1}{\alpha-1}}\tau\right)^{\frac{\alpha-1}{2}},\\
(vii)\quad&\label{S5EalphaR }
M_{\alpha}(H(\tau))\leq C(\alpha,E)\left(\frac{1}{1-e^{-\gamma(\alpha,E)\tau}}\right)^{2(\alpha-1)},\,\,\forall \alpha\geq 3,
\end{align}
where the constants $C(\alpha,E)$ and $\gamma (\alpha,E)$ are defined in Theorem \ref{S5T5R}.
\end{corollary}

\begin{remark}
Under the hypothesis that all the moments of the initial data $h_0$  are bounded it is easy to obtain the estimate (\ref{S5Ealpha }) using the weak formulation  (\ref{AUXW}). However, it is not so easy using the regularized weak formulation (\ref{REGWEAK}) below. For that reason, we first want to obtain a solution $h$ satisfying (\ref{AUXW}) with an initial data with bounded moments of all order.
\end{remark}

\subsection{A first result.}
\begin{theorem}
\label{Ex1T1}
For any $h_0\in\mathscr{M}^1_+([0,\infty))$ with $N=M_0(h_0)$ and $E=M_1(h_0)$, there exists
$h\in C\big([0,\infty), \mathscr{M}_+^1 ([0,\infty))\big)$ that satisfies (\ref{lip loc h})--(\ref{MAh}).
\end{theorem}

The proof of  Theorem \ref{Ex1T1} is made in two steps. We first solve a regularised version of (\ref{AUXW}). Then, in a second step, we use an approximation argument. More precisely, we consider the following cutoff:
\begin{cutoff}
\label{cut-off}
For every $n\in \NN$ let $\phi _n\in  C_c([0,\infty))$ be such that $\supp\phi_n=[0,n+1]$,
$\phi _n(x)\le x^{-1/2}$ for all $x>0$ and $\phi_n(x)=x^{-1/2}$ for all $x\in\left(\frac {1} {n}, n \right)$,
in such a way that:
\begin{align}
\forall x>0\,\,\,\,\,\lim _{ n\to \infty }\phi _n(x)=\frac {1} {\sqrt x}.
\end{align}
\end{cutoff}

\subsection{Regularised problem}
We now solve in Theorem \ref{Ex1T2} a regularised version of (\ref{AUXW}) with the operator $\widetilde{\mathscr{Q}}_{3,n}$ defined in 
(\ref{Aq3tilden})--(\ref{Q3n1w}). The solution $h_n$ is obtained as a mild solution to the equation
\begin{align}
\label{Ex1EApn}
\frac {\partial h_n} {\partial\tau}(\tau,x)=J_{3,n}(h_n(\tau))(x),
\end{align} 
where $J_{3,n}$ is defined in (\ref{A1E32})-(\ref{A1E35}), and corresponds to a regularised version of the term $J_3$ defined in (\ref{PR2}). Namely,
$J_{3,n}(h)=J_3(h\phi_n)$, where $\phi_n$ is as in Cutoff \ref{cut-off}.
\begin{theorem}
\label{Ex1T2}
For any $n\in\NN$ and any nonnegative function  $h_0\in C_c([0,\infty))$,
there exists a unique nonnegative function $h_n\in C\big([0,\infty), L^{\infty}(\RR_+)\cap L^1_{x}(\RR_+)\big)$ such that for all $\tau\in[0,\infty)$ 
and all $\varphi\in  L^1_{loc}(\RR_+)$:
\begin{align}
&\tau\mapsto\int_{[0,\infty)}\varphi(x)h(\tau,x)dx\in W^{1,\infty}_{ loc  }([0,\infty)) \label{S5E765}\\
&\frac{d}{d\tau}\int_0^{\infty}\varphi(x)h_n(\tau,x)dx=\widetilde{\mathscr{Q}}_{3,n}(\varphi,h_n(\tau)).\label{REGWEAK}\\
&h_n(0,x)=h_0(x)\label{datan}
\end{align}
Moreover, if we denote by $N=M_0(h_0)$ and $E=M_1(h_0)$, then for every $\tau\in[0,\infty)$ and $\alpha\geq 3$:
\begin{align}
&M_0(h_n(\tau))\leq\bigg(\frac{E}{2}\tau+\sqrt{N}\bigg)^2, \label{mass inequality} \\
&M_1(h_n(\tau))=E, \label{conservation of energy}\\
&M_{\alpha}(h_n(\tau))\leq\left(M_{\alpha}(h_0)^{\frac{2}{\alpha-1}}+\alpha 2^{\alpha-1}E^{\frac{\alpha+1}{\alpha-1}}\tau
\right)^{\frac{\alpha-1}{2}}. \label{MAhn}
\end{align}
Furthermore, there exist two positive constants $C _{ 1, n }$ and $C _{ 2, n }$ depending on $n$ and $\|h_0\| _{ L^\infty\cap L^1_x }$ such that for all $\tau >0$:
\begin{equation}
\label{a priori sup norm}\|h_n(\tau)\|_{\infty}\leq C _{ 1, n }e^{C _{ 2, n }(\tau^2+\tau)}.
\end{equation}
\end{theorem}
\begin{proof}
Using
(\ref{A1E32}) we write equation (\ref{Ex1EApn}) as
\bear
\label{Ex1EApn2}
\frac {\partial h_n} {\partial \tau }+h_nA_n(h_n)=K_n(h_n)+L_n(h_n),
\eear
and the solution $h_n$ is obtained as a fixed point of the operator:
\begin{align}
R_n(h_n)(\tau,x)=&h_0(x)S_n(0, \tau; x)\nonumber \\
&+\int_0^\tau S_n(\sigma , \tau; x)\big(K_n(h_n)(\sigma,x)+L_n(h_n)(\sigma,x)\big)d\sigma, \label{Ex1E1}\\
S_n(\sigma,\tau; x)=&e^{-\int_\sigma ^\tau A_n(h_n)(\sigma,x)d\sigma}\label{Ex1E4}
\end{align} 
on
\begin{align}
B(T):=\Big\{h\in C\big([0,T],L^{\infty}&(\RR_+)\cap L^1_x(\RR_+)\big): h\geq 0\quad\text{and} \nonumber\\
&\sup_{\tau\in[0,T]}\|h(\tau)\|_{L^{\infty}\cap L^1_x}\leq 2\|h_0\|_{L^{\infty}\cap L^1_x}\Big\}. \label{Ex1E2}
\end{align}
Let us show first that $R_n$ sends $B(T)$ into itself. Let $r_0:=\|h_0\|_{L^\infty\cap L^1_x}$ and for an arbitrary $T>0$, let $h\in  B(T)$. 
By Proposition \ref{well defined operators}  with $\rho (x)=x$, 
\begin{align*}
&R_n(h)(\tau,x)\geq 0\qquad\forall\tau\in[0,T],\;\forall x\in\RR_+,\\
&R_n(h)\in C\big([0,T],L^{\infty}(\RR_+)\cap L^1_x(\RR_+)\big).
\end{align*}
Moreover, using (\ref{SAE100}) and  (\ref{bound L}):
\begin{align*}
\sup _{ \tau \in [0,T] }\|R_n(h)(\tau)\|_{L^{\infty}\cap L^1_x}
\leq r_0+T\,C(n)(4r_0^2+2r_0).
\end{align*}
If $T$ satisfies:
\bear
T\le \frac {1} {C(n)(4r_0+2)}
\label{Ex1E257}
\eear
then 
$R_n(h)\in B(T)$.

To prove that $R_n$ is a contraction, let $h_1\in B(T)$, $h_2\in B(T) $ and write:

\begin{align*}
&\big|R_n(h_1)(\tau,x)-R_n(h_2)(\tau,x)\big|\le h_0(x)\left|  S_1(0, \tau ; x)-S_2(0, \tau ; x)\right| +\\
&+\int_0^{\tau}\left|  S_1(\sigma , \tau ; x)-S_2(\sigma , \tau ; x) \right|
\big(K_n(h_1)(\sigma,x)+L_n(h_1)(\sigma,x)\big)d\sigma\\
&+\int_0^{\tau}\big|K_n(h_1)(\sigma,x)-K_n(h_2)(\sigma,x)\big|d\sigma\\
&+\int_0^{\tau} \big|L_n(h_1)(\sigma,x)-L_n(h_2)(\sigma,x)\big|d\sigma.
\end{align*}
By (\ref{SaE121}), for all $\sigma \ge 0$ and $\tau \ge0$
\begin{align}
\left|  S_1(\sigma , \tau ; x)-S_2(\sigma , \tau ; x) \right| & \le   \int_0^\tau |A_n(h_1)(\sigma,x)-A_n(h_2)(\sigma,x)|d\sigma \nonumber \\
& \le C(n)\,\tau  \sup_{\tau\in[0,T]}\|h_1(\tau)-h_2(\tau)\|_{\infty}.   \label{Ex1E258}
\end{align}
Using now (\ref{Ex1E258}) and (\ref{SAE100})--(\ref{SaE121}), we deduce:
\begin{align*}
&\|R_n(h_1)(\tau)-R_n(h_2)(\tau)\| _{ L^\infty \cap L^1_x }\le  C_1
\sup_{\tau\in[0,T]}\|h_1(\tau)-h_2(\tau)\|_{\infty},\\
&C_1\equiv C_1(n, T, r_0)=C(n)T\left(1+ 3r_0+2Tr_0(1+2r_0)\right).
\end{align*}
If (\ref{Ex1E257}) holds and
\begin{align*}
C(n)T\left(1+ 3r_0+2Tr_0(1+2r_0)\right)<1,
\end{align*}
$R_n$ will be a contraction from $B(T)$ into itself. This is achieved, for example, as soon as:
\begin{align*}
T<\min \left\{ \frac {1} {2r_0(1+2r_0)}, \frac {1} {2C(n) (1+2r_0)}\right\}=\kappa _{ r_0 }.
\end{align*}
The fixed point $h_n$  of $R_n$ in $B(T)$  is then a  mild solution of  (\ref{Ex1EApn}), that can be extended to a maximal interval of existence $[0, T _{ n, \max })$. 

We claim now that $h_n$ satisfies (\ref{S5E765}), (\ref{REGWEAK}).
Since $h_n$ is a mild solution of (\ref{Ex1EApn}):
\begin{equation}
\label{Ex1mildE}
h_n(\tau , x)=h_0(x)S_n(0, \tau; x)+\int_0^\tau S_n(\sigma , \tau; x)\big(K_n(h_n)(\sigma,x)+L_n(h_n)(\sigma,x)\big)d\sigma 
\end{equation}
We multiply this equation by $\varphi\in L^1_{loc}(\RR_+)$ and integrate on $(0, \infty)$:
\begin{align*}
\int _0^\infty & h_n(\tau , x)\varphi ( x)dx=\int _0^\infty h_0(x)S_n(0, \tau ; x) \varphi (x)dx+\\
&+\int _0^\tau \int _0^\infty S_n(\sigma , \tau; x)\big(K_n(h_n)(\sigma,x)+L_n(h_n)(\sigma,x)\big)\varphi (x)dx d\sigma.
\end{align*}
Using Lemma \ref{well defined operators} and $h_0\in C_c([0,\infty))$, it follows that the integrals above are well define.
It also follows from Lemma \ref{well defined operators} and (\ref{Ex1E4}) that $\tau \mapsto \int _0^\infty h_n(\tau , x)\varphi (x)dx$ is  locally Lipschitz on $(0, T _{ n, \max })$, and:
\begin{align*}
\frac {d} {dt}\int _0^\infty &h_n(\tau , x)\varphi (x)dx= \int _0^\infty h_0(x)(S_n(0, \tau ; x)) _{ \tau  } \varphi  (x)dx+\\
&+\int _0^\infty\big(K_n(h_n)(\tau ,x)+L_n(h_n)(\tau ,x)\big)\varphi (x) dx+\\
&+\int _0^\tau \int _0^\infty (S_n(\sigma , \tau; x))_\tau \big(K_n(h_n)(\sigma,x)+L_n(h_n)(\sigma,x)\big)\varphi (x)dx d\sigma.
\end{align*}
We use now that $(S_n(\sigma , \tau; x))_\tau=-A_n(h_n)(\tau , x)S_n(\sigma , \tau; x)$
 and the identity (\ref{Ex1mildE}) to deduce:
\begin{align*}
\frac {d} {dt}\int _0^\infty h_n(\tau , x)\varphi (x)dx=&\int _0^\infty\!\!\big(K_n(h_n)(\tau ,x)+L_n(h_n)(\tau, x)\big)\varphi (x) dx- \nonumber\\
&-\int _0^\infty A_n(h_n)h_n(\tau , x)\varphi (\tau , x)dx,
\end{align*}
that is (\ref{REGWEAK}).

Suppose now that $T _{ n, \max }<\infty$ and 
$$
\sup_{\tau \in [0, T _{ n, \max })}\|h_n(\tau)\|_{L^{\infty}\cap L^1_x}<\infty.
$$
Then there is an increasing sequence 
$\tau_j\rightarrow T_{n,\max}$  as $j\rightarrow\infty$  and $L>0$ such that
$$
\sup_{j }\|h_n(\tau_j)\|_{L^{\infty}\cap L^1_x}\le L<\infty.
$$

Fix $\delta >0$ such that $
\delta <\kappa _{ r_0+1 }$.
Starting with the initial value $h(\tau _j)$ we have a mild solution $h_j$ defined on $[0, \delta]$. Gluing together $h$ with $h_j$ we obtain a mild solution on
$[0, t_j+\delta]$. For $j$ large enough, $t_j+\delta >T _{ n, \max }$, and this is a contradiction. Therefore, either $T _{ n, \max }=\infty$ or, if 
$T _{ n, \max }=\infty$, then $\limsup \|h_n(\tau)\|_{L^{\infty}\cap L^1_x}=\infty$ as $\tau\to T_{n,\max}$.
 
Let us prove now the estimates (\ref{mass inequality}), (\ref{conservation of energy}) and (\ref{a priori sup norm}), first for all $\tau\in (0, T _{ n, \max })$. Then,  the property $T _{ n, \max }=\infty$ will follow. We start proving (\ref{conservation of energy}). To this end we use (\ref{REGWEAK}) with $\varphi=x$. Since in that case 
$\Lambda(\varphi)(x,y)=0$ and $\mathcal{L}(\varphi)(x)=0$, 
(\ref{conservation of energy}) is immediate. To prove (\ref {mass inequality}), we use  (\ref{REGWEAK})  with $\varphi =1$. 
Then, $\Lambda(\varphi)(x,y)=0$ and $\mathcal{L}(\varphi)(x)=-x$
 and then, using $\phi _n\le x^{-1/2}$, H\"older inequality and  (\ref{conservation of energy}): 
\begin{align*}
\frac{d}{d\tau}\left(\int_0^\infty h_n(\tau,x)d x\right)^{1/2}\leq\frac{\sqrt{E}}{2},
\end{align*}
from where (\ref {mass inequality}) follows. 

In order to prove \eqref{a priori sup norm} we use \eqref{mass inequality}:
\begin{align*}
\|K_n(h_n)(\sigma)\|_{\infty}&\leq\|\phi_n\|_{\infty}^2\|h_n(\sigma)\|_1\|h_n(\sigma)\|_{\infty}\\
&\leq\|\phi_n\|_{\infty}^2\bigg(\frac{\sqrt{E}}{2}\sigma+\sqrt{N}\bigg)^2\|h_n(\sigma)\|_{\infty},
\end{align*}
which combined with the estimate $\|L_n(h_n)(\sigma)\|_{\infty}\leq 2\|\phi_n\|_1\|h_n(\sigma)\|_{\infty}$, gives 
\begin{align*}
\|h_n(\tau)\|_{\infty}&\leq\|h_0\|_{\infty}
+\int_0^{\tau}\big(\|K_n(h_n)(\sigma)\|_{\infty}+\|L_n(h_n)(\sigma)\|_{\infty}\big)d\sigma\\
&\leq \|h_0\|_{\infty}+C(n,h_0)\int_0^{\tau}(\sigma^2+1)\|h_n(\sigma)\|_{\infty}d\sigma.
\end{align*}
where
\bean
C(n,h_0)=\max\left\{\|\phi _n\|_1\|\phi _n\| ^2_{ \infty } \|h_0\|_1,\, \frac {\|\phi _n\| ^2_{ \infty }} {4}\|h_0\| _{ L^1_x }\right\}.
\eean
Then \eqref{a priori sup norm} follows from Gronwall's inequality.

For the proof of (\ref{MAhn}) we use (\ref{REGWEAK}) with $\varphi(x)=x^{\alpha}$ for $\alpha\geq 3$:
\bear
\label{S5E902}
\frac {d} {d\tau }M_{\alpha}(h_n(\tau))=\widetilde{\mathscr{Q}}_{3,n}(\varphi,h_n(\tau)).
\eear
Since:
\begin{align}
\label{MAL}
\mathcal{L}(\varphi)(x)=\left(\frac{\alpha-1}{\alpha+1}\right)x^{\alpha+1}\geq 0,
\end{align}
we have,
\bean
\frac{d}{d\tau}M_{\alpha}(h_n(\tau))\leq 
 2\int_0^{\infty}\!\!\!\!\int_0^x \Lambda(\varphi)(x,y)\phi_n(x)\phi_n(y)h_n(\tau,x)h_n(\tau,x) dydx.
\eean
Then, we write
$\Lambda(\varphi)(x,y)=x^{\alpha}\big((1+z)^{\alpha}+(1-z)^{\alpha}-2\big),$
where $z=y/x$, and by Taylor's expansion around $z=0$:
\begin{align*}
u(z)\leq\frac{\|u''\|_{\infty}}{2}z^2\leq \alpha(\alpha-1)2^{\alpha-3}z^2.
\end{align*}
Hence for all $0\leq y\leq x$:
\begin{align}
\label{MAQ}
\Lambda(\varphi)(x,y)&\leq C_{\alpha}x^{\alpha-2}y^2,\qquad\text{where}\qquad C_{\alpha}=\alpha(\alpha-1)2^{\alpha-3},
\end{align}
and then, using $\phi_n(x)\phi_n(y)\leq y^{-1}$ and (\ref{conservation of energy}),
\begin{align*}
\frac{d}{d\tau}M_{\alpha}(h_n(\tau) \leq 2C_{\alpha}M_{\alpha-2}(h_n(\tau))E.
\end{align*}
Since  by Holder's inequality and (\ref{conservation of energy})
\begin{align*}
M_{\alpha-2}(h_n(\tau))\leq E^{\frac{2}{\alpha-1}}M_{\alpha}(h_n(\tau))^{\frac{\alpha-3}{\alpha-1}},
\end{align*} 
we deduce
\begin{align*}
\frac{d}{d\tau}\left(M_{\alpha}(h_n(\tau))^{\frac{2}{\alpha-1}}\right)\leq\frac{4C_{\alpha}}{\alpha-1}E^{\frac{\alpha+1}{\alpha-1}},
\end{align*}
and (\ref{MAhn}) follows.
\end{proof}
\subsection{Proof of Theorem \ref{Ex1T1}.}

The solution $h$ whose existence is claimed in Theorem \ref{Ex1T1}  is obtained as the limit of a subsequence of solutions  $(h_n) _{ n\in \NN }$  to the regularized problems obtained in Theorem \ref{Ex1T2}. We first prove the following Lemma.

\begin{lemma}
\label{precomp}
Let $h_0\in C_c([0,\infty))$ be nonnegative with $N=M_0(h_0)>0$ and $E=M_1(h_0)>0$, and consider $(h_n)_{ n\in \NN }$ the sequence of functions  given by Theorem \ref{Ex1T2}. Then for every $\tau\in[0,\infty)$ there exists a subsequence, still denoted $(h_n(\tau))_{n\in\NN}$, and a measure $h(\tau)\in\mathscr{M}^1_+([0,\infty))$ such that, as $n\to \infty$,  
$h_n(\tau)$ converges to $h(\tau)$ in the following sense:
\begin{align}
\label{growth condition}
&\forall\varphi \in C([0, \infty));\;\exists\theta\in [0, 1): \quad \sup _{ x\ge 0} \frac {\varphi (x)} {1+x^\theta}<\infty,\\
\label{sense of convergence}
&\lim _{ n\to \infty }\int_{[0,\infty)}\varphi(x)h_{n}(\tau,x)d x=\int_{[0,\infty)}\varphi(x)h(\tau,x)d x.
\end{align}
Moreover, for every $\tau\in[0,\infty)$:
\begin{align}
\label{pre conservation laws for widetildeG}
&M_0(h(\tau))\leq\bigg(\frac{\sqrt{E}}{2}\tau+\sqrt{N}\bigg)^2,\\
\label{pre energy}
&M_1(h(\tau))\leq E.
\end{align}
\end{lemma}

\begin{proof}
Let us prove first the convergence for a subsequence of $(h_n(\tau))_{n\in\NN}$. For every $\tau\geq 0$ we have by (\ref{mass inequality}) that 
\begin{align*}
\sup_{n\in\NN}\int_0^\infty h_n(\tau,x)d x\leq \bigg(\frac{\sqrt{E}}{2}\tau+\sqrt{N}\bigg)^2.
\end{align*}
Therefore, there exists a subsequence, still denoted $(h_n(\tau))_{n\in\NN}$, and a measure $h(\tau)$ such that 
 $(h_n(\tau))_{n\in\NN}$  converges to $h(\tau)$ in the weak* topology of $\mathscr{M}([0,\infty))$, as $n\to \infty$: 
\begin{equation}
\label{weak* 1}
\lim _{ n\to \infty  }\int_{[0,\infty)}\!\!\!\varphi(x)h_{n}(\tau,x)d x=
\int_{[0,\infty)}\!\!\varphi(x)h(\tau,x)d x,\,\,
\forall\varphi\in C_0([0,\infty)).
\end{equation}
Since for all $n\in\NN$, $h_n(\tau)$ is nonnegative, then $h(\tau)$ is a positive measure. Also by weak* convergence and \eqref{mass inequality} we deduce that $h(\tau)$ is a finite measure:
\begin{align}
\label{weak mass}
\int_{[0,\infty)}h(\tau,x)d x\leq\liminf_{n\rightarrow\infty}\int_0^\infty h_{n}(\tau,x)d x\leq\bigg(\frac{\sqrt{E}}{2}\tau+\sqrt{N}\bigg)^2.
\end{align}
Moreover, by (\ref{conservation of energy}) we also have that the sequence $(h_n(\tau))_{n\in\NN}$ is bounded in $L^1_x(\RR_+)$. 
Hence there exists a subsequence (not relabelled) that converges to a measure 
$\nu(\tau)$ in the weak* topology of $\mathscr{M}([0,\infty))$, i.e., such that
\begin{equation}
\label{weak* 2}
\lim _{ n\to \infty  }\int_0^\infty\varphi(x)\,x\,h_{n}(\tau,x)d x=\int_{[0,\infty)}\!\!\!\varphi(x)\nu(\tau,x)dx,\,\forall\varphi\in C_0([0,\infty)).
\end{equation}
Again, since $h_n(\tau)$ is nonnegative for all $n\in\NN$ then $\nu(\tau)$ is a positive measure. Also by weak* convergence and \eqref{conservation of energy} we have
\begin{align}
\label{weak energy}
\int_{[0,\infty)}\nu(\tau,x)d x\leq\liminf_{n\rightarrow\infty}\int_0^\infty\,x\,h_{n}(\tau,x)d x=E.
\end{align}
Let us show now that $\nu(\tau)=x\,h(\tau)$. This will follow from 
\begin{align}
\label{equality measures}
\forall \varphi\in C_0([0,\infty)):  \int_{[0,\infty)}\varphi(x)\nu(\tau,x)d x=\int_{[0,\infty)}\varphi(x)\,x\,h(\tau,x)d x
\end{align}
 In a first step we show that (\ref{equality measures}) holds for $\varphi\in C_c([0,\infty))$ and then we use a density argument. 
Let $\varepsilon>0$ and 
$\varphi\in C_c([0,\infty))$. Using \eqref{weak* 2} with test function $\varphi$, and \eqref{weak* 1} with test function $x\varphi(x)$, we deduce that 
\begin{align*}
\bigg|\int_{[0,\infty)}&\varphi(x)\nu(\tau,x)d x-\int_{[0,\infty)}\varphi(x)\,x\,h(\tau,x)d x\bigg |\\
&\leq \bigg|\int_0^\infty\varphi(x)\nu(\tau,x)d x-\int_{[0,\infty)}\varphi(x)\,x\,h_n(\tau,x)d x\bigg|\\
&+\bigg|\int_0^\infty\varphi(x)\,x\,h_n(\tau,x)d x-\int_{[0,\infty)}\varphi(x)\,x\,h(\tau,x)d x\bigg|<\varepsilon
\end{align*}
for $n$ large enough. Hence \eqref{equality measures} holds for all $\varphi\in C_c([0,\infty))$. Now let $\varphi\in C_0([0,\infty))$ and consider a sequence 
$(\varphi_k)_{k\in\NN}\subset C_c([0,\infty))$ such that \\$\|\varphi_k-\varphi\|_{\infty}\rightarrow0$ as $k\rightarrow\infty$. 
Using \eqref{equality measures} with $\varphi_k$ and the bounds \eqref{weak mass} and \eqref{weak energy}, we deduce that

\begin{align*}
\bigg|\int_{[0,\infty)}&\varphi(x)	\nu(\tau,x)d x-\int_{[0,\infty)}\varphi(x)\,x\,h(\tau,x)d x \bigg|\\
&\leq\int_{[0,\infty)}\big|\varphi(x)-\varphi_k(x)\big|\nu(\tau,x)d x\\
&+\bigg|\int_{[0,\infty)}\varphi_k(x)\nu(\tau,x)d x-\int_{[0,\infty)}\varphi_k(x)\,x\,h(\tau,x)d x \bigg|\\
&+\int_{[0,\infty)}\big|\varphi_k(x)-\varphi(x)\big|\,x\,h(\tau,x)d x<\varepsilon
\end{align*}
for $k$ large enough. Therefore \eqref{equality measures} holds for all $\varphi \in C_0([0,\infty))$, i.e., $\nu(\tau)=x\,h(\tau)$. Hence we rewrite \eqref{weak* 2} as 
\begin{align}
\lim _{ n\to \infty }\int_0^{\infty}\varphi(x)x\,h_n(\tau,x)d x =\int_{[0,\infty)} \!\!\!\!\!\!\varphi(x)x\,h(\tau,x)dx, \,\,\forall\varphi\in C_0([0,\infty)).
\end{align}
Let us show now (\ref{growth condition}), (\ref{sense of convergence}).
Let  then  $\varphi\in C([0, \infty))$  be any nonnegative test function that satisfies (\ref{growth condition}). We denote $(\zeta _j) _{ j\in \NN }$ a sequence of nonnegative and nonincreasing functions  of  $C_c^\infty([0, \infty))$ such that:
\begin{equation*}
\zeta_j(x)= 1\,\,\hbox{if}\,\,\,x\in [0, j),\qquad
\zeta_j(x)= 0\,\,\,\hbox{if}\,\,x>j+1,
\end{equation*}
and  define $\varphi_j=\varphi\,\zeta_j$. Then for every $n,\,j\in\NN$:
\begin{align}
\label{est 1. lemma convergence for tau fixed}
\bigg|\int_0^{\infty}&\varphi(x)h_n(\tau,x)d x-\int_{[0,\infty)}\varphi(x)h(\tau,x)d x \bigg|\\
&\leq\int_0^{\infty}\big|\varphi(x)-\varphi_j(x)\big|h_n(\tau,x)d x\nonumber\\
&+\bigg|\int_0^{\infty}\varphi_j(x)h_n(\tau,x)d x-\int_{[0,\infty)}\varphi_j(x)h(\tau,x)d x \bigg|\nonumber\\
&+\int_{[0,\infty)}\big|\varphi_j(x)-\varphi(x)\big|h(\tau,x)d x\nonumber
\end{align}
Since $\varphi_j\in C_0([0,\infty))$, using \eqref{weak* 2}, the second term in the right hand side of 
\eqref{est 1. lemma convergence for tau fixed} converges to zero as $n\rightarrow\infty$ for every $j\in\NN$. The first and the third term in the right hand side of \eqref{est 1. lemma convergence for tau fixed} are treated in the same way, using that $\varphi _j(x)=\varphi (x)$ for all $x\in [0, j)$. For instance, in the first term:
\begin{align*}
\int_0^{\infty}\big|\varphi(x)&-\varphi_j(x)\big|h_n(\tau,x)d x
=\int_j^{\infty}\big|\varphi(x)-\varphi_j(x)\big|h_n(\tau,x)d x\\
&\leq 2\int_j^{\infty}|\varphi(x)| h_n(\tau,x)d x \leq 2C\int_j^{\infty}(1+x^{\theta})h_n(\tau,x)d x\\
&\leq 2C\left(\frac{1+j^{\theta}}{j}\right)\int_j^{\infty}x\,h_n(\tau,x)d x\leq 2C\left(\frac{1+j^{\theta}}{j}\right)E.
\end{align*}
Therefore this term is small provided $j$ is large enough. In conclusion, the difference in \eqref{est 1. lemma convergence for tau fixed} is less than $\varepsilon$ for $n$ sufficiently large, i.e., \eqref{sense of convergence} holds.
\end{proof}

\begin{remark}
\label{remark. weak* topology generated by a distance}
The so-called narrow topology $\sigma(\mathscr{M}([0,\infty)),C_b([0,\infty)))$ on $\mathscr{M}_+([0,\infty))$ is generated by the metric
$d(\mu,\nu)=\|\mu-\nu\|_0$, where
\begin{align*}
\|\mu\|_0=\sup\left\{\int_{[0,\infty)}\varphi d\mu:\varphi\in\text{Lip}_1([0,\infty)),\;\|\varphi\|_{\infty}\leq 1\right\},
\end{align*}
(cf. \cite{BOG} Theorem 8.3.2).
\end{remark}

Using this Remark, Lemma \ref{precomp} and the Arzel\`{a}-Ascoli's Theorem we prove now the following:
\begin{proposition}
\label{equicont}
Let $h_0$ and $(h_n) _{ n\in \NN }$ be as in Lemma \ref{precomp}.
Then there exist a subsequence (not relabelled) and $h\in C\big([0,\infty),\mathscr{M}_+([0,\infty))\big)$ such that 
\begin{align}
\label{S5E90}
h_n\xrightarrow[n\rightarrow\infty]{}h\quad\text{in}\quad C\big([0,\infty),\mathscr{M}_+([0,\infty))\big).
\end{align}
Moreover, if we denote by $N=M_0(h_0)$ and $E=M_1(h_0)$, then for all $\tau\geq 0$ 
\begin{align}
\label{MASS IN}
&M_0(h(\tau))\leq\bigg(\frac{\sqrt{E}}{2}\tau+\sqrt{N}\bigg)^2,\\
\label{EN IN}
&M_1(h(\tau))\leq E,
\end{align}
and for all $\varphi\in C([0,\infty))$ satisfying the growth condition (\ref{growth condition}):
\begin{equation}
\label{sense of convergence 2}
\lim _{ n\to \infty  }\int_0^{\infty}\varphi(x)h_{n}(\tau,x)d x=\int_{[0,\infty)}\varphi(x)h(\tau,x)d x.
\end{equation}
\end{proposition}

\begin{proof}[\upshape\bfseries{Proof of Proposition \ref{equicont}}]
By Lemma \ref{precomp} the sequence 
$(h_n(\tau))_{n\in\NN}$ is relatively compact in $\mathscr{M}([0,\infty))$ for every $\tau\in[0,\infty)$. Let us show now that $(h_n)_{n\in\NN}$ is also equicontinuous. To this end let $\tau_2\geq\tau_1\geq 0$, and consider $\varphi$ as in Remark \ref{remark. weak* topology generated by a distance}, i.e., 
$\varphi\in\text{Lip}([0,\infty))$ with Lipschitz constant $\text{Lip}(\varphi)\leq 1$, and $\|\varphi\|_{\infty}\leq 1$.
Then, using $\phi_n(x)\leq x^{-1/2}$, (\ref{lemma regularity 1}) and (\ref{lemma regularity 4}) in Lemma \ref{lemma regularity}, we have
\begin{align}
&\bigg|\int_0^{\infty}\varphi(x)h_n(\tau_1,x)d x-\int_0^{\infty}\varphi(x)h_n(\tau_2,x)d x\bigg|\nonumber \\
&\leq\int_{\tau_1}^{\tau_2}\big|\widetilde{\mathscr{Q}}_{3,n}(\varphi,h_n(\sigma))\big|d\sigma\leq 2\int_{\tau_1}^{\tau_2}\bigg(\int_0^{\infty}h_n(\sigma,x)d x\bigg)^2d\sigma\nonumber\\
&+4\int_{\tau_1}^{\tau_2}\int_0^{\infty}\sqrt{x}\,h_n(\sigma,x)d xd\sigma.\label{equicontinuity 1}
\end{align}
Using H\"{o}lder's inequality and the estimates (\ref{mass inequality}) and (\ref{conservation of energy}) in (\ref{equicontinuity 1}), it follows that 
\begin{align*}
&\bigg|\int_0^{\infty}\varphi(x)h_n(\tau_1,x)d x-\int_0^{\infty}\varphi(x)h_n(\tau_2,x)d x\bigg|\\
&\leq 2\int_{\tau_1}^{\tau_2}\bigg(\frac{\sqrt{E}}{2}\sigma+\sqrt{N}\bigg)^4d\sigma
+4\sqrt{E}\int_{\tau_1}^{\tau_2}\bigg(\frac{\sqrt{E}}{2}\sigma+\sqrt{N}\bigg)d\sigma\quad\forall n\in\NN.\nonumber
\end{align*}
We then deduce using Remark \ref{remark. weak* topology generated by a distance} that $(h_n)_{n\in\NN}$ is equicontinuous.
It then follows from Arzel\`{a}-Ascoli's Theorem (cf. for example \cite{Roy})  that there exists $h\in C\big([0,\infty),\mathscr{M}_+([0,\infty))\big)$ such that
$h_n\rightarrow h$ in $C\big([0,T], \mathscr{M}_+([0,\infty))\big)$, for every $T>0$, as $n\rightarrow\infty$. 

The estimates 
(\ref{MASS IN}), (\ref{EN IN}) and the convergence (\ref{sense of convergence 2}) are deduced in the same way as in the Proof of  Lemma \ref{precomp}.
\end{proof}

\begin{proof}[\upshape\bfseries{Proof of Theorem \ref{Ex1T1}}]
By Corollary \ref{APD1}, there exists a sequence of nonnegative function $(h_{0,n})_{n\in\NN}\in C_c([0,\infty))$ that approximate $h_0$ in the weak* topology of the space $C_b([0,\infty))^*$.
Let then  $(h_n) _{ n\in \NN }\subset  C\big([0,\infty),\mathscr{M}_+([0,\infty))\big)$ be the sequence of solutions to (\ref{S5E765}), (\ref{REGWEAK})
obtained by Theorem \ref{Ex1T2} with the initial data $h_{0,n}$.  By Proposition \ref{equicont}  there exists a subsequence, still denoted $(h_n) _{ n\in \NN }$, and $h\in  C\big([0,\infty), \mathscr{M}_+([0,\infty))\big)$ such that
$h_n$ converges to $h$ in the topology of $C\big([0,\infty),\mathscr{M}_+([0,\infty))\big)$. 

By (\ref{REGWEAK}) and (\ref{datan}), for all  $\varphi \in C^1_b([0, \infty))$ and $\tau >0$:
\begin{equation}
\label{S5E456}
\int_0^{\infty}\varphi (x)h_n(\tau , x)dx-\int_0^{\infty}\varphi (x)h _{ 0, n }( x)dx=\int _0^\tau \widetilde{\mathscr{Q}}_{3,n}(\varphi,h_n(\sigma))d\sigma. 
\end{equation}
By construction, for every $\varphi\in C_b^1([0,\infty))$ and every $\tau\in[0,\infty)$:
\begin{align}
\label{LIM55}
\lim _{ n\to \infty  }\int_0^{\infty}\varphi(x)h_n(\tau,x)d x=\int_{[0,\infty)}\varphi(x)h(\tau,x)d x.
\end{align}
We prove now the convergence of the linear term: for all $\varphi\in C_b^1([0,\infty))$ and $\tau\in[0,\infty)$
\begin{align}
\label{linear limit 3}
\lim_{n\to\infty}\widetilde{\mathscr{Q}}_{3,n}^{(1)}(\varphi, h_n(\tau))=\widetilde{\mathscr{Q}}_3^{(1)}(\varphi, h_n(\tau)).
\end{align}
By definition:
\begin{align}
&\Big|\widetilde{\mathscr{Q}}_3^{(1)}(\varphi,h(\tau))- \widetilde{\mathscr{Q}}_{3,n}^{(1)}(\varphi,h_n(\tau))\Big|\nonumber \\
&\leq  \bigg|\int_0^{\infty}\frac{\mathcal{L}(\varphi)(x)}{\sqrt{x}}h(\tau,x)d x-\int_0^{\infty}\frac{\mathcal{L}(\varphi)(x)}{\sqrt{x}}h_n(\tau,x)d x \bigg|\nonumber\\
&+\int_0^{\infty}\bigg|\mathcal{L}(\varphi)(x)\phi_n(x)-\frac{\mathcal{L}(\varphi)(x)}{\sqrt{x}}\bigg|h_n(\tau,x)d x. \label{linear limit 1}
\end{align}
From Lemma \ref{lemma regularity} (iii) and (\ref{sense of convergence 2}):
\begin{equation}
\lim _{ n\to \infty } \bigg|\int_0^{\infty}\frac{\mathcal{L}(\varphi)(x)}{\sqrt{x}}h(\tau,x)d x-\int_0^{\infty}\frac{\mathcal{L}(\varphi)(x)}{\sqrt{x}}h_n(\tau,x)d x \bigg|=0
\end{equation}
For the second term  in the right hand side of (\ref {linear limit 1}) we split the integral $\int _0^\infty$ in two:  $\int_0^R$ and  $\int_R^{\infty}$ for   $R>0$,  
and apply (\ref{lemma regularity 4}). We obtain:
\begin{align}
\label{linear limit 2}
\int_0^{\infty}&\bigg|\mathcal{L}(\varphi)(x)\phi_n(x)-\frac{\mathcal{L}(\varphi)(x)}{\sqrt{x}}\bigg|h_n(\tau,x)d x  \\
&\leq \bigg\|\mathcal{L}(\varphi)(x)\phi_n(x)-\frac{\mathcal{L}(\varphi)(x)}{\sqrt{x}}\bigg\|_{C([0,R])}\int_0^R h_n(\tau,x)d x+\nonumber\\
&\hskip 4.5cm +4\|\varphi\|_{\infty}\int_R^{\infty}\sqrt{x}\;h_n(\tau,x)d x\nonumber.
\end{align}
By (\ref{conservation of energy}), for any $\varepsilon >0$ and $R> (E/\varepsilon )^2$:
\begin{align*}
\int_R^{\infty}\sqrt{x}\;h_n(\tau,x)d x \leq\frac{E}{\sqrt{R}}<\varepsilon \qquad\forall n\in\NN.
\end{align*}
Then by Lemma \ref{regularised operators converge uniformly} and (\ref{mass inequality}), the part on $[0,R]$ converges to zero as $n\to\infty$. 
Since $R>0$ is arbitrary we finally deduce that (\ref{linear limit 2}) converges to zero as $n\to\infty$.
Therefore (\ref{linear limit 3}) holds.

Let us prove now the convergence of the quadratic term: for all $\varphi\in C_b^1([0,\infty))$ and all $\tau\in[0,\infty)$:
\begin{align}
\label{quadratic limit 2}
\lim_{n\to\infty}\mathscr{Q}_{3,n}^{(2)}(\varphi, h_n(\tau))=\mathscr{Q}_3^{(2)}(\varphi, h_n(\tau)).
\end{align}
As before
\begin{align}
\label{quadratic limit 1}
&\bigg|\mathscr{Q}_{3}^{(2)}(\varphi,h(\tau))-\mathscr{Q}_{3,n}^{(2)}(\varphi,h_n(\tau))\bigg|\\
&\leq\bigg|\mathscr{Q}_{3}^{(2)}(\varphi,h(\tau))-\int_0^{\infty}\!\!\!\!\int_0^{\infty}\frac{\Lambda(\varphi)(x,y)}{\sqrt{xy}}h_n(\tau,x)h_n(\tau,y)d xd y\bigg|\nonumber\\
&+\int_0^{\infty}\!\!\!\!\int_0^{\infty}\bigg|\Lambda(\varphi)(x,y)\phi_n(x)\phi_n(y)-\frac{\Lambda(\varphi)(x,y)}{\sqrt{xy}}\bigg|h_n(\tau,x)h_n(\tau,y)d xd y\nonumber.
\end{align}
It follows from Lemma \ref{lemma regularity} (ii) and (\ref{sense of convergence 2})
that the first term in the right hand side above converges to zero as $n\rightarrow\infty$.
For the second term we proceed as before. For any $R>0$ we split the double integral:
\begin{align*}
&\int_0^{\infty}\!\!\!\!\int_0^{\infty}\bigg|\Lambda(\varphi)(x,y)\phi_n(x)\phi_n(y)-\frac{\Lambda(\varphi)(x,y)}{\sqrt{xy}}\bigg|h_n(\tau,x)h_n(\tau,y)d xd y\\
&\leq \bigg\|\Lambda(\varphi)(x,y)\phi_n(x)\phi_n(y)-\frac{\Lambda(\varphi)(x,y)}{\sqrt{xy}}\bigg\|_{C([0,R]^2)}\left(\int_0^R h_n(\tau,x)dx\right)^2\nonumber\\
&+\iint_{(0,\infty)^2\setminus (0,R)^2}\bigg|\Lambda(\varphi)(x,y)\phi_n(x)\phi_n(y)-\frac{\Lambda(\varphi)(x,y)}{\sqrt{xy}}\bigg|h_n(\tau,x)h_n(\tau,y)d xd y\\
&=I_1+I_2.
\end{align*}
By Lemma \ref{regularised operators converge uniformly} and (\ref{mass inequality}), $I_1$ converges to zero as 
$n\to\infty$.
For the term $I_2$ we use (\ref{lemma regularity 1}) in Lemma \ref{lemma regularity} and the estimates (\ref{conservation of energy}) and (\ref{mass inequality}):
\begin{align*}
&\int_R^{\infty}\!\!\!\!\int_R^{\infty}\bigg|\Lambda(\varphi)(x,y)\phi_n(x)\phi_n(y)-\frac{\Lambda(\varphi)(x,y)}{\sqrt{xy}}\bigg|
h_n(\tau,x)h_n(\tau,y)d xd y\\
&\leq 4\|\varphi'\|_{\infty}\bigg(\int_R^{\infty}h_n(\tau,x)d x\bigg)^2
\leq \frac{4\|\varphi'\|_{\infty}E^2}{R^2}\qquad\forall n\in\NN
\end{align*}
and
\begin{align*}
&2\int_R^{\infty}\!\!\!\int_0^R\bigg|\Lambda(\varphi)(x)\phi_n(x)\phi_n(y)-\frac{\Lambda(\varphi)(x)}{\sqrt{xy}}\bigg|
h_n(\tau,x)h_n(\tau,y)d xd y\\
&\leq4\|\varphi'\|_{\infty}\int_R^\infty\int_0^Rh_n(\tau,x)h_n(\tau,y)d xd y
\leq \frac{4 \|\varphi'\|_{\infty}E}{R}\bigg(\frac{\sqrt{E}}{2}\tau+\sqrt{N}\bigg)^2.
\end{align*}
Since $R>0$ is arbitrary we deduce that $I_2$ also converges to zero as $n\to\infty$.
We then conclude that (\ref{quadratic limit 2}) holds.

Combining (\ref{linear limit 3}) and (\ref{quadratic limit 2}) it follows that for all $\varphi\in C_b^1([0,\infty))$ and all $\tau\in[0,\infty)$: 
\begin{align}
\label{LIMQ3N}
\lim _{ n\to \infty  }\widetilde{\mathscr{Q}}_{3,n}(\varphi,h_n(\tau))=\widetilde{\mathscr{Q}}_3(\varphi,h(\tau)).
\end{align}
Moreover, using $\phi_n(x)\leq x^{-1/2}$, (\ref{lemma regularity 1}) and (\ref{lemma regularity 4}) in Lemma \ref{lemma regularity}, and the estimates
(\ref{mass inequality}) and (\ref{conservation of energy}), we have for all $\varphi\in C_b^1([0,\infty))$, all $\tau\in[0,\infty)$ and all $n\in\NN$:

\begin{align*}
&\Big|\widetilde{\mathscr{Q}}_{3,n}(\varphi,h_n(\tau))\Big|\leq\\
&\leq 2\|\varphi'\|_{\infty}\bigg(\int_0^{\infty}h_n(\tau,x)dx\bigg)^2+4\|\varphi\|_{\infty}\int_0^{\infty}\sqrt{x}\,h_n(\tau,x)dx\\
&\leq2\|\varphi'\|_{\infty}\bigg(\frac{\sqrt{E}}{2}\tau+\sqrt{N}\bigg)^4+4\|\varphi\|_{\infty}\sqrt{E}\bigg(\frac{\sqrt{E}}{2}\tau+\sqrt{N}\bigg).
\end{align*}
By  (\ref{LIMQ3N}) and the dominated convergence Theorem:
\begin{align}
\label{CQ3I}
\lim _{ n\to \infty  }\int_0^{\tau}\widetilde{\mathscr{Q}}_{3,n}(\varphi,h_n(\sigma))d\sigma=\int_0^{\tau}\widetilde{\mathscr{Q}}_3(\varphi,h(\sigma))d\sigma.
\end{align}
Using now (\ref{LIM55}) and (\ref{CQ3I}), we may pass to the limit as $n\to \infty$ in (\ref{S5E456}) for all $\varphi\in C_b^1([0,\infty))$ and all $\tau\in[0,\infty)$ to obtain:
\begin{equation}
\label{S5E963}
\int_{[0,\infty)}\varphi(x)h(\tau,x)dx=\int_{[0,\infty)}\varphi(x)h_0(x)dx+\int_0^{\tau}\widetilde{\mathscr{Q}}_3(\varphi,h(\sigma))d\sigma.
\end{equation}
The map $\tau\mapsto\int_{[0,\infty)}\varphi(x)h(\tau,x)dx$ is then locally Lipschitz on $[0,\infty)$, and  $h$ satisfies  (\ref{lip loc h}), (\ref{AUXW}) for all
$\varphi\in C_b^1([0,\infty))$ and for a.e. $\tau\in[0,\infty)$. It also follows from (\ref{S5E963}) that $h(0)=h_0$ in $\mathscr M_+$.

The property (\ref{MMI}) follows from (\ref{MASS IN}). The  conservation of energy (\ref{EE}) is obtained as follows. We already know by (\ref{EN IN}) that $M_1(h(\tau ))\le E$. On the other hand, let  $\varphi_k\in C^1_b([0,\infty))$ be a concave test function such that $\varphi_k(x)=x$ for $x\in[0,k)$ and $\varphi_k(x)=k+1$ for $x\geq k+2$. Notice that there exists a positive constant $C$ such that
\begin{align}
\label{SUP9}
\sup_{k\in\NN}\|\varphi_k'\|_{\infty}\leq C.
\end{align}
By Remark \ref{concave-negativity}, $\widetilde{\mathscr{Q}}_3^{(1)}(\varphi_k,h)\leq 0$ and $\mathscr{Q}_3^{(2)}(\varphi_k,h)\leq 0$ for all $k\in\NN$,
and then, from (\ref{S5E963}):
\begin{align}
\label{INW1}
\int_{[0,\infty)}\!\!\!\varphi_k(x)h(\tau,x)dx\geq\int_{[0,\infty)}\!\!\!\varphi_k(x)h_0(x)dx+\int_0^{\tau}\mathscr{Q}_{3}^{(2)}(\varphi_k,h(\sigma))d\sigma.
\end{align}
We now prove that for all $\tau\in[0,\infty)$:
\begin{align}
\label{DOMQ32}
\lim _{ k\to \infty }\int_0^{\tau}\mathscr{Q}_{3}^{(2)}(\varphi_k,h(\sigma))d\sigma =0.
\end{align}
Notice that $\Lambda(\varphi_k)(x,y)\to 0$ as $k\to\infty$, since $\varphi_k(x)\to x$. Then, using (\ref{lemma regularity 1}) in Lemma \ref{lemma regularity}, (\ref{SUP9}) and (\ref{mass inequality}), we deduce for all $\tau\in[0,\infty)$ and $\sigma \in (0, \tau )$:
\begin{align}
&\lim _{ k\to \infty }\mathscr{Q}_{3}^{(2)}(\varphi_k,h(\sigma))=0\\
&\Big| \mathscr{Q}_{3}^{(2)}(\varphi_k,h(\tau))\Big|\leq 2C\bigg(\frac{\sqrt{E}}{2}\tau+\sqrt{N}\bigg)^4\qquad\forall k\in\NN.
\end{align}
and  (\ref{DOMQ32}) follows from the dominated convergence Theorem. We take now limits in (\ref{INW1}) as $k\to \infty$. By (\ref{DOMQ32}) and the monotone convergence Theorem we obtain that $M_1(h(\tau ))\ge E$ and then $M_1(h(\tau ))=E$ for all $\tau >0$.

We assume now that $M_\alpha (h_0)<\infty $ for some $\alpha \geq 3$ and prove (\ref{MAh}). By  (\ref{MAhn}) and 
(\ref{APD56}) in Corollary \ref{APD1}:
\begin{align}
M_{\alpha}(h(\tau))&\leq  \liminf_{ n\to \infty } \left(M_{\alpha}(h _{ 0, n })^{\frac{2}{\alpha-1}}+\alpha 2^{\alpha-1}M_1(h_{0,n})^{\frac{\alpha+1}{\alpha-1}}\tau
\right)^{\frac{\alpha-1}{2}}\\
&\leq\left(M_{\alpha}(h _{ 0})^{\frac{2}{\alpha-1}}+\alpha 2^{\alpha-1}E^{\frac{\alpha+1}{\alpha-1}}\tau
\right)^{\frac{\alpha-1}{2}}.
\end{align}
\end{proof}

\subsection{ Proof of Theorem \ref{S5T5R}}
\begin{proof}[\upshape\bfseries{Proof of Theorem \ref{S5T5R}}] 
Consider again the sequence of initial data $h _{ 0, n }$  used in the proof of Theorem \ref{Ex1T1} and the sequence of solutions  $h _{ n }$  obtained  by Theorem \ref{Ex1T1}. Using  (\ref{MAhn})  we know that $M_\alpha (h_n(\tau ))<\infty$ for $\tau >0$ and $n\in \NN$. 

Our first step is to prove that (\ref{AUXW}) holds also true for $\varphi (x)=x^\alpha $. Notice that $h_n$ solves now the equation (\ref{AUXW}), with the operator $\widetilde{\mathscr Q}_3$ in the right-hand side, whose kernel is not compactly supported and the argument in the proof of (\ref{MAhn}) must be slightly modified.

In order to use (\ref{AUXW}) we consider a sequence $(\varphi_k)\subset C^1_b([0,\infty))$ such that:\begin{align}
\label{M7a5}
&\varphi_k\to\varphi\quad\text{as}\quad k\to\infty\\
\label{M7a6}
&\varphi_k\leq\varphi_{k+1}\leq \varphi\\
\label{M7a7}
&\varphi'\geq\varphi_k'\geq 0.
\end{align}
Let us prove by the dominated convergence Theorem  that for all $\tau\geq 0$:
\begin{align}
\label{Q332}
(i)\quad&\widetilde{\mathscr{Q}}_3(\varphi,h_n)\in L^1(0,\tau),\\
\label{Q333}
(ii)\quad&\lim _{ k\to \infty }\int_0^{\tau}\widetilde{\mathscr{Q}}_3(\varphi_k,h_n(\sigma))d\sigma=\int_0^{\tau}\widetilde{\mathscr{Q}}_3(\varphi,h_n(\sigma))d\sigma.
\end{align}
To this end we first observe that, for $x\ge y>0$:
\begin{align}
\label{MAJ1}
\lim _{ k\to \infty }\Lambda(\varphi_k)(x, y)=\Lambda(\varphi)\qquad\text{and}\qquad \lim _{ n\to \infty  }\mathcal{L}(\varphi_k)=\mathcal{L}(\varphi)(x)
\end{align}
and, by the mean value Theorem:
\begin{align*}
\frac{\Lambda(\varphi_k)(x,y)}{\sqrt{xy}}\leq \varphi_k'(\xi_1)-\varphi_k'(\xi_2)
\end{align*}
for some $\xi_1\in (x,x+y)$ and $\xi_2\in(x-y,x)$. Using then (\ref{M7a7}):
\begin{align}
\label{MAJ2}
\frac{|\Lambda(\varphi_k)(x,y)|}{\sqrt{xy}}\leq \alpha\big(2^{\alpha-1}+1\big)x^{\alpha-1}\qquad\forall k\in\NN,
\end{align}
and by (\ref{M7a6}):
\begin{align*}
\frac{|\mathcal{L}(\varphi_k)(x)|}{\sqrt{x}}\leq \left(\frac{\alpha+3}{\alpha+1}\right)x^{\alpha+\frac{1}{2}}\qquad\forall k\in\NN.
\end{align*}
Since by Theorem \ref{Ex1T1}: $M_{\alpha-1}(h_n(\tau))<\infty$ and $M_{\alpha+1/2}(h_n(\tau))<\infty$, for every fixed $n$  we may apply the Lebesgue's convergence Theorem 
to  the sequences $\left\{\frac{\Lambda(\varphi_k)(x,y)}{\sqrt{xy}}h_n(\sigma , x)h_n(\sigma , y)\right\} _{ k\in \NN }$ and
$\left\{\frac{\mathcal{L}(\varphi_k)(x)}{\sqrt{x}}h_n(\sigma , x)\right\} _{ k\in \NN }$ and obtain (\ref{Q332}), (\ref{Q333}).

We use now $\varphi_k$ in (\ref{AUXW}) and take the limit $k\to\infty$. We obtain from (\ref{M7a5}), (\ref{M7a6}), (\ref{Q333}) and monotone convergence:
\begin{align}
M_{\alpha}(h_n(\tau))=M_{\alpha}(h _{ 0, n })+\int_0^{\tau}\widetilde{\mathscr{Q}}_3(\varphi,h_n(\sigma))d\sigma\qquad\forall\tau\geq 0,
\end{align}
and then, using (\ref{Q332}):
\begin{align}
\label{EQMA}
\frac{d}{d\tau}M_{\alpha}(h_n(\tau))=\widetilde{\mathscr{Q}}_3(\varphi,h_n(\tau))\qquad a.e.\;\tau>0.
\end{align}
If we use  (\ref{MAL}) and (\ref{MAQ}) in the right hand side of (\ref{EQMA}), we obtain
\begin{align*}
\frac{d}{d\tau}M_{\alpha}(h_n)\leq 2^{\alpha-2}\alpha(\alpha-1)E_n M_{\alpha-2}(h_n)-\left(\frac{\alpha-1}{\alpha+1}\right)M_{\alpha+\frac{1}{2}}(h_n),
\end{align*}
where $E_n=M_1(h_{0,n})$.
Now by H\"{o}lder's inequality:
\begin{align*}
&M_{\alpha-2}(h_n)\leq E_n^{2/(\alpha-1)}M_{\alpha}(h_n)^{(\alpha-3)/(\alpha-1)}\\
&M_{\alpha}(h_n)\leq E_n^{1/(2\alpha-1)}M_{\alpha+\frac{1}{2}}(h_n)^{2(\alpha-1)/(2\alpha-1)}.
\end{align*}
Then we obtain
\begin{align*}
\frac{d}{d\tau}M_{\alpha}(h_n)&\leq 2^{\alpha-2}\alpha(\alpha-1)E_n^{1+2/(\alpha-1)}M_{\alpha}(h_n)^{(\alpha-3)/(\alpha-1)}\\
&-\left(\frac{\alpha-1}{\alpha+1}\right)E_n^{-1/(2(\alpha-1))}M_{\alpha}(h_n)^{(2\alpha-1)/(2(\alpha-1))}.\nonumber
\end{align*}
Since $(\alpha-3)/(\alpha-1)\in [0,1)$ then
\begin{align*}
M_{\alpha}(h_n)^{(\alpha-3)/(\alpha-1)}\leq 1+M_{\alpha}(h_n),
\end{align*}
and :
\begin{align}
\label{MOM19}
\frac{d}{d\tau}M_{\alpha}(h_n)&\leq 2^{\alpha-2}\alpha(\alpha-1)E_n^{1+2/(\alpha-1)}\big(1+M_{\alpha}(h_n)\big)\\
&-\left(\frac{\alpha-1}{\alpha+1}\right)E_n^{-1/(2(\alpha-1))}M_{\alpha}(h_n)^{(2\alpha-1)/(2(\alpha-1))},\nonumber
\end{align}
where $(2\alpha-1)/(2(\alpha-1))>1$. If we define:
\begin{align*}
&u(\sigma)=M_{\alpha}(h_n(\tau)),\quad\sigma=C_1\tau,\quad q=2(\alpha-1),\\
&C_1=2^{\alpha-2}\alpha(\alpha-1)E^{1+2/(\alpha-1)}\\
&C_2=\left(\frac{\alpha-1}{\alpha+1}\right)E^{-1/(2(\alpha-1))},\quad C=\frac {C_2} {C_1}.
\end{align*}
We deduce from (\ref{MOM19}) that
\begin{align}
u'\leq1+u-Cu^{1+1/q},
\end{align}
and then by  Lemma 6.3 in \cite{Bob}, for every $n\in \NN$:
\begin{equation}
M_{\alpha}(h_n(\tau))\leq C(\alpha,E_n)\left(\frac{1}{1-e^{-\gamma(\alpha,E_n)\tau}}\right)^{2(\alpha-1)},
\end{equation}
where the constants $C(\alpha,E_n)$ and $\gamma (\alpha,E_n)$ are defined as in Theorem \ref{S5T5R}. We may argue  now as in the proof of Theorem \ref{Ex1T1}
and  pass to the limit along a subsequence to obtain a limit $h\in C\big([0,\infty), \mathscr{M}_+ ([0,\infty))\big)$ satisfying (\ref{lip loc h})--(\ref{EE}) and (\ref{S5Ealpha }). Using (\ref{S5Ealpha }) and $h\in C\big([0,\infty), \mathscr{M}_+ ([0,\infty))\big)$ we deduce  as in the proof of Theorem \ref{Ex1T1} that 
$h\in C\big((0,\infty), \mathscr{M}_+^\alpha  ([0,\infty))\big)$ for all $\alpha\geq 1$.
\end{proof}

\begin{proof}[\upshape\bfseries{Proof of Corollary \ref{S5C52R}}] 
We first observe that  the map $\tau\mapsto M_{1/2}(h(\tau))$ is locally bounded. Indeed by H\"{o}lder's inequality, (\ref{MMI}) and (\ref{EE}):
$$
M_{1/2}(h(\tau))\leq\sqrt{M_1(h(\tau))M_0(h(\tau))}\leq\sqrt{E}\bigg(\frac{\sqrt{E}}{2}\tau+\sqrt{N}\bigg).
$$
Then by (\ref{lip loc h}) it follows (\ref{lip loc H}). Now for all $\varphi\in C^1_b([0,\infty))$ and for a.e. $\tau> 0$, we deduce from (\ref{AUXW}): 
\begin{align*}
\quad\frac{d}{d\tau}\int_{[0,\infty)}\varphi(x)H(\tau,x)dx&=\widetilde{\mathscr{Q}}_3(\varphi,h(\tau))-\varphi(0)M_{1/2}(h(\tau))\\
&=\mathscr{Q}_3(\varphi,h(\tau)).
\end{align*}
Since $H=h$ on $(0,\infty)$ then $\mathscr{Q}_3(\varphi,H)\equiv\mathscr{Q}_3(\varphi,h)$, and therefore (\ref{AUXWH}) holds.

Now for the initial data: $H(0)=h(0)=h_0$.
The conservation of mass (\ref{MMH}) follows from (\ref{AUXWH}) for $\varphi=1$, since $\Lambda(\varphi)=0=\mathcal{L}_0(\varphi)$. The conservation of energy (\ref{EEH}) follows directly from (\ref{EE}) since $H=h$ on $(0,\infty)$. 
\end{proof}

\section{Properties of $h(\tau , \{0\})$.}
\label{SectionC}
\setcounter{equation}{0}
\setcounter{theorem}{0}

In all this Section we denote
\begin{equation}
\label{mZ}
m(\tau )=h(\tau , \{0\}).
\end{equation}
The main result of this Section  is the following.
\begin{theorem}
\label{S1T4h}
Suppose that $h\in C([0, \infty); \mathscr{M}_+^1([0,\infty))$ is a solution of (\ref{S1E16ha}) with  $h(0)=h_0\in\mathscr{M}_+^1([0,\infty))$, $N=M_0(h_0)>0$ and $E=M_1(h_0)>0$. 
Then $m$ is right continuous, a.e. differentiable and strictly increasing on $[0,\infty)$.
\end{theorem}
We begin with the following properties of the function $m$ defined in (\ref{mZ}).
\begin{lemma}
\label{S6L1'}
The function $m$ is nondecreasing, a.e. differentiable and right continuous on $[0,\infty)$.
\end{lemma}

\begin{proof}
Given any $\varphi_{\varepsilon}$ as in Remark \ref{TEST}, then for all $\tau \ge 0$
\begin{align}
&m(\tau)=\lim _{ \varepsilon \to 0 }\int_{[0,\infty)}\varphi_{\varepsilon}(x)h(\tau,x)dx,\label{S6L1E1}
\end{align}
and by (\ref{S1E16ha})-(\ref{S1EB2})
\begin{align}
\frac{d}{d\tau}\int_{[0,\infty)}\varphi_{\varepsilon}(x)h(\tau,x)dx=\mathscr{Q}_3^{(2)}(\varphi_{\varepsilon},h(\tau))
-\widetilde{\mathscr{Q}}_3^{(1)}(\varphi_{\varepsilon},h(\tau)). \label{S6E47}
\end{align}
Since $\varphi_{\varepsilon}$ is convex, nonnegative and decreasing, it follows from Lemma \ref{convex-positivity} 
that $\mathscr{Q}_3^{(2)}(\varphi_{\varepsilon},h)\geq 0$ and 
$\widetilde{\mathscr{Q}}_3^{(1)}(\varphi_{\varepsilon},h)\leq 0$ for all $\varepsilon>0$. 
Then by (\ref{S6E47})
\begin{equation*}
\int_{[0,\infty)}\varphi_{\varepsilon}(x)h(\tau_2,x)d x\geq\int_{[0,\infty)}\varphi_{\varepsilon}(x)h(\tau_1,x)dx\qquad\forall\tau_2\geq\tau_1\geq 0.
\end{equation*}
Letting $\varepsilon\to 0$ it follows from (\ref{S6L1E1}) that $m$ is nondecreasing on $[0,\infty)$ and, for any $\tau\geq 0$ and $\delta>0$,
\begin{align}
\label{right continuity 1}
\liminf_{\delta\rightarrow 0^+}m(\tau+\delta)\geq m(\tau).
\end{align}
Using Lebesgue's Theorem (cf. \cite{Roy}), $m$ is a.e. differentiable on $[0,\infty)$. 
On the other hand, if in (\ref{S6E47}) the term $\widetilde{\mathscr{Q}}_3^{(1)}(\varphi_{\varepsilon},h)$ is dropped,
\begin{align*}
\int_{[0,\infty)}\varphi_{\varepsilon}(x)h(\tau+\delta,x)&d x\leq\int_{[0,\infty)}\varphi_{\varepsilon}(x)h(\tau,x)d x
+\int_{\tau}^{\tau+\delta}\mathscr{Q}_3^{(2)}(\varphi_{\varepsilon},h(\sigma))d\sigma.
\end{align*}
Using $\mathds{1}_{\{0\}}\leq\varphi_{\varepsilon}$ for all $\varepsilon>0$, and  the bound \eqref{lemma regularity 1}, we deduce
\begin{equation*}
m(\tau+\delta)\leq\int_{[0,\infty)}\varphi_{\varepsilon}(x)h(\tau,x)d x
+\frac {2\delta} {\varepsilon }(M_0(h(\tau)))^2.
\end{equation*}
If we take now superior limits as $\delta\to 0^+$
at $\varepsilon >0$ fixed,
\begin{align*}
\limsup_{\delta\rightarrow 0^+}m(\tau+\delta)&\leq\int_{[0,\infty)}\varphi_{\varepsilon}(x)h(\tau,x)d x\qquad\forall\varepsilon>0.
\end{align*}
We may pass now to the limit  as $\varepsilon \to 0$ in the right hand side above
and use (\ref{S6L1E1}) to get,
\begin{align}
\label{right continuity 2}
\limsup_{\delta\rightarrow 0^+}m(\tau+\delta)\leq m(\tau).
\end{align}
The right continuity then follows from \eqref{right continuity 1} and \eqref{right continuity 2}.
\end{proof}

\begin{corollary}
\label{S6L1}
The map $\tau\mapsto H(\tau,\{0\})$, defined for all $\tau \ge 0$, is right continuous on $[0,\infty)$ and 
\begin{align}
\label{JIM}
\limsup_{\delta\rightarrow 0^+}H(\tau-\delta,\{0\})\leq H(\tau,\{0\})\qquad\forall\tau> 0.
\end{align}
\end{corollary}

\begin{proof}
By construction (cf.(\ref{S1EdecompH})) it follows
\begin{align*}
H(\tau,\{0\})=m(\tau)-\int_0^{\tau}M_{1/2}(h(\sigma))d\sigma.
\end{align*}
Since $M_{1/2}(h)\in L^1_{loc}(\RR_+)$ then $\tau\mapsto\int_0^{\tau}M_{1/2}(h(\sigma))d\sigma$ is absolutely continuous, 
and since $m$ is right continuous by Lemma \ref{S6L1'}, it follows that $\tau\mapsto H(\tau,\{0\})$ is also right continuous. 
To prove (\ref{JIM}) we use the continuity of $\tau\mapsto\int_0^{\tau}M_{1/2}(h(\sigma))d\sigma$ and the monotonicity of $h(\tau,\{0\})$:
for all $\tau>0$ and $\delta\in(0,\tau)$,
\begin{align*}
\limsup_{\delta\to 0^+} H(\tau-\delta,\{0\})&=\limsup_{\delta\to 0^+} m(\tau-\delta)-\int_0^{\tau}M_{1/2}(h(\sigma))d\sigma\\
&\leq m(\tau)-\int_0^{\tau}M_{1/2}(h(\sigma))d\sigma=H(\tau,\{0\}).
\end{align*}
\end{proof} 

\begin{remark}
\label{S6Ej}
We do not know if the map $\tau\mapsto H(\tau,\{0\})$ is continuous. By property \eqref{JIM} however, $H(\tau,\{0\})$ does not decrease through the possible discontinuities.
\end{remark}

The proof of Theorem \ref{S1T4h} closely follows the proof of Proposition 1.21 in \cite{AV1} (see also \cite{KIER}, Ch. 3), where the authors proved the same result for the equation without the linear term $\widetilde{\mathscr Q}_3^{(1)}$. The main arguments in the proof are, on the one hand, the invariance of the problem (\ref{S1E16ha}) with respect to time translation and under a suitable  scaling transformation. On the other hand, the fact that 
$\Lambda(\varphi)\ge 0$ on $\RR_+^2$ for convex test functions $\varphi $, and that the map $\tau\mapsto\mathscr{Q}_3^{(2)}(\varphi,h(\tau))$ is locally integrable on $[0, T)$. When the linear term $\widetilde{\mathscr Q}_3^{(1)}$ is added, a slight modification of these argument still leads to the proof. 
Since by Lemma \ref{convex-positivity}, for all nonnegative, convex decreasing test function $\varphi \in C_b^1([0, \infty))$, we have
$\widetilde{\mathscr{Q}}_3^{(1)}(\varphi,h)\leq 0$, then solutions $h$ to (\ref{S1E16ha}) are also super solutions (cf. Definition \ref{SUPER}).

\begin{proposition}
\label{S5P1}
Let $h$ be a super solution of (\ref{S1E16ha}). Then for any $R>0$ and $\theta\in(0,1)$
\begin{align}
\int_{[0,R]}h(\tau ,x)d x\geq (1-\theta)\int_{[0,\theta R]}h(\tau_0 ,x)d x\qquad\forall \tau\geq\tau_0\geq 0.\label{S5EP12}
\end{align}
\end{proposition}

\begin{proof}
Chose $\varphi _R(x)=(1-x/R)_+$ for $R>0$, and consider a sequence $(\varphi_{R,n})_{n\in\NN}\subset C_b^1([0, \infty))$ such that $\varphi_{R,n}\to\varphi_R$,
$\varphi _{R,n}\le\varphi_{R}$ and $\varphi_{R,n}(0)=1$ for all $n\in\NN$. Since by convexity $\mathscr{Q}_3^{(2)}(\varphi_{R,n},h)\geq 0$, then for all
$\tau$ and $\tau_0$ with $\tau\geq\tau_0\geq 0$,
\begin{align*}
\int_{ [0, \infty) } &\varphi_{ R, n } (x)h(\tau  , x)dx\ge \int_{ [0, \infty) } \varphi _{ R, n }(x)h(\tau_0  , x)dx\\
&\ge  \int_{ [0,\theta R] } \varphi _{ R, n }(x)h(\tau_0  , x)dx
\ge \varphi_{ R, n }(\theta R) \int_{ [0,\theta R]}h(\tau_0  , x)dx,
\end{align*}
and (\ref{S5EP12}) follows since, if we let $n\to\infty$,
$$
\int_{[0,R]}h(\tau ,x)d x\geq \int_{ [0,\infty) }\varphi_R(x)h(\tau  , x)dx
\ge \varphi _R(\theta R) \int_{ [0,\theta R]}h(\tau_0, x)dx.
$$
\end{proof}

\begin{lemma}
\label{S5L1}
Let $h$ be a super solution of (\ref{S1E16ha}). Let $R>0$ and consider a sequence $R:=a_0< a_1< a_2<...< a_n<...$ such that 
$|a_i-a_{i-1}|\leq\frac{R}{2}$ for all $i\in\{1, 2, 3,... \}$. Then for all $\tau \geq\tau_0 \geq 0$ there holds
\begin{align}
\int_{[0,R]}h(\tau,x)d x\geq\sum_{i=1}^\infty \frac{1}{2a_i}\int_{\tau _0}^{\tau}\bigg(\int_{(a_{i-1},a_i]}h(\sigma,x)d x\bigg)^2d\sigma. \label{S5E97}
\end{align}
\end{lemma}

\begin{proof}
We chose $\varphi _R$ and $\varphi _{R, n }$ as in the proof of Proposition \ref{S5P1} above. Since $h$ is a super solution of (\ref{S1E16ha}), then
for all $n\in\NN$,
$$
\frac {d} {d\tau }\int  _{ [0, \infty) }h(\tau , x)\varphi  _{ R, n }(x)dx\ge 
{\mathscr Q}_3^{(2)}(\varphi  _{ R, n },h(\tau)).
$$
We have now:
\bean
&&{\mathscr Q}_3^{(2)}(\varphi  _{ R, n },h(\tau)) \ge \iint _{ (R, \infty)^2 }h(\tau, x)h(\tau , y)
\frac {\varphi  _{ R, n }(|x-y|)} {\sqrt {x y}}dxdy\\
&&\ge \sum_{ i=1 } ^\infty\frac { \varphi  _{ R, n }(R/2)} {a_i} \iint _{(a _{ i-1 }, a_i]^2}h(\tau, x)h(\tau , y)dxdy\\
&&=\sum_{ i=1 } ^\infty\frac { \varphi  _{ R, n }(R/2)} {a_i} \left(\int _{(a _{ i-1 }, a_i]}h(\tau, x)dx\right)^2.
\eean
Estimate (\ref{S5E97}) follows in the limit $n\to\infty$, since $\varphi_{ R, n }(R/2)\to1/2.$
\end{proof}

\begin{proposition}
\label{S5P2}
Let $h$ be a super solution of (\ref{S1E16ha}) with initial data $h_0\in\mathscr{M}_+^1([0,\infty))$, and denote $N=M_0(h_0)$ and $E=M_1(h_0)$. Then for all $R>0$, 
$\alpha\in\left(-\frac{1}{2},\infty\right)$, and $\tau_1$ and $\tau_2$ with $0\leq\tau_1\leq\tau_2$:
\begin{align}
\label{MNEG6}
\int_{\tau_1}^{\tau_2}\int_{(0,R]}x^{\alpha}h(\tau,x)dxd\tau
\leq\frac{2R^{\frac{1}{2}+\alpha}\sqrt{\tau_2-\tau_1}}{1-\left(\frac{2}{3}\right)^{\frac{1}{2}+\alpha}}\left(\frac{\sqrt{E}}{2}\tau_2+\sqrt{N}\right).
\end{align}
\end{proposition}

\begin{proof}
Since h is a super solution of (\ref{S1E16ha}), if we chose $\varphi(x)=(1-x/r)^2_+$ for $r>0$, then 
\begin{align}
\label{MNEG4}
\int_{[0,\infty)}\varphi(x)h(\tau_2,x)dx\geq\int_{\tau_1}^{\tau_2}\mathscr{Q}_3^{(2)}(\varphi,h(\tau))d\tau.
\end{align}
Since 
$\supp\Lambda(\varphi)=\{(x,y)\in[0,\infty)^2:|x-y|\leq r\}$
and
$\Lambda(\varphi)(x,y)=\varphi(|x-y|)$ for all $(x,y)\in[r,\infty)^2$,
then for all $\tau\geq 0$:
\begin{align*}
\mathscr{Q}_3^{(2)}(\varphi,h(\tau))&\geq \iint_{\left(r,\frac{3r}{2}\right]^2}\frac{\varphi(|x-y|)}{\sqrt{xy}}h(\tau,x)h(\tau,y)dxdy\\
&\geq\frac{1}{4}\left(\int_{\left(r,\frac{3r}{2}\right]}\frac{h(\tau,x)}{\sqrt{x}}dx\right)^2.
\end{align*}
If we use that $\varphi\leq1$ 
in the left hand side of (\ref{MNEG4}), and  the estimate above in the right hand side, then
\begin{align*}
\int_{\tau_1}^{\tau_2}\left(\int_{\left(r,\frac{3r}{2}\right]}\frac{h(\tau,x)}{\sqrt{x}}dx\right)^2d\tau\leq 4 M_0(h(\tau_2)).
\end{align*} 
Since for any $\alpha\in(-1/2,\infty)$
\begin{align*}
\int_{\left(r,\frac{3r}{2}\right]}\frac{h(\tau,x)}{\sqrt{x}}dx
\geq \left(\frac{3r}{2}\right)^{-\alpha-\frac{1}{2}}\int_{\left(r,\frac{3r}{2}\right]}x^{\alpha}h(\tau,x)dx,
\end{align*}
we then obtain
\begin{align}
\label{MNEG5}
\int_{\tau_1}^{\tau_2}\left(\int_{\left(r,\frac{3r}{2}\right]}x^{\alpha}h(\tau,x)dx\right)^2d\tau\leq 4M_0(h(\tau_2))\left(\frac{3r}{2}\right)^{1+2\alpha}.
\end{align}
For any given $R>0$, using the decomposition 
$$(0,R]=\bigcup_{k=0}^\infty(a_{k+1},a_k],\qquad a_k=\left(\frac{2}{3}\right)^{k}\!\!\!\!R,$$
and Cauchy-Schwarz inequality we obtain
\begin{align*}
\int_{\tau_1}^{\tau_2}\int_{(0,R]}x^{\alpha}h(\tau,x)dxd\tau
\leq\sqrt{\tau_2-\tau_1}\sum_{k=0}^{\infty}\left(\int_{\tau_1}^{\tau_2}\bigg(\int_{(a_{k+1},a_k]}\!\!\!\!\!\!\!\!\!\!\!\!x^{\alpha}h(\tau,x)dx\bigg)^2d\tau\right)^{\frac{1}{2}}.
\end{align*}
If we chose $r=a_{k+1}$ so that $(a_{k+1},a_k]=(r,(3/2)r]$ for every $k\in\NN$, then by (\ref{MNEG5}) we deduce
\begin{align*}
\int_{\tau_1}^{\tau_2}\int_{(0,R]}x^{\alpha}h(\tau,x)dxd\tau
\leq2\sqrt{(\tau_2-\tau_1)M_0(h(\tau_2))}\sum_{k=0}^{\infty}a_k^{\frac{1}{2}+\alpha}.
\end{align*}
Using the estimate (\ref{MMI}) for $M_0(h(\tau_2))$ and 
\begin{align*}
\sum_{k=0}^{\infty}a_k^{\frac{1}{2}+\alpha}
=\frac{R^{\frac{1}{2}+\alpha}}{1-\left(\frac{2}{3}\right)^{\frac{1}{2}+\alpha}},
\end{align*}
we finally obtain (\ref{MNEG6}).
\end{proof}

\begin{lemma}
\label{lemma 2^n}
Let $h$ be a super solution of (\ref{S1E16ha}). Then for all $r>0$, $\tau\geq \tau_0\geq 0$ and $n\in\NN$:
\begin{equation}
\label{e1}
\int_{[0,r]}h(\tau,x)dx \geq\frac{1}{4^{n+1}r}\int_{\tau_0}^{\tau}\bigg(\int_{(r,r2^n]}h(\sigma,x)dx\bigg)^2d\sigma.
\end{equation}
\end{lemma}

\begin{proof}
Consider the decomposition
$$
\left(r,2^nr\right]=\bigcup_{i=3}^{2^{n+1}}\left(\frac{r}{2}(i-1),\frac{r}{2}i\right].
$$
Then by Lemma \ref{S5L1}, and Lemma 3.12 in \cite{AV1}, we have
\begin{align*}
\int_{[0,r]}h(\tau,x)dx&\geq\int_{\tau_0}^{\tau}\sum_{i=3}^{2^{n+1}}\frac{1}{ri}\bigg(\int_{\left(\frac{r}{2}(i-1),\frac{r}{2}i\right]}
h(\sigma,x)dx\bigg)^2d\sigma\\
&\geq\int_{\tau_0}^{\tau}\frac{1}{r}\Bigg(\sum_{i=3}^{2^{n+1}}i\Bigg)^{-1}\bigg(\int_{(r,r2^n]}h(\sigma,x)dx\bigg)^2d\sigma\\
&\geq\frac{1}{(2^n-1)(2^{n+1}+3)r}\int_{\tau_0}^{\tau}\bigg(\int_{(r,r2^n]}h(\sigma,x)dx\bigg)^2d\sigma.
\end{align*}
Notice that $(2^n-1)(2^{n+1}+3)\leq 4^{n+1}$.
\end{proof}

The next Lemma takes into account the linear term $\widetilde{\mathscr Q}_3^{(1)}$.
\begin{lemma}
\label{S5L47} 
Let $h$ be a solution of (\ref{S1E16ha}) with initial data $h_0\in \mathscr{M}_+^1([0,\infty))$ satisfying
\begin{equation}
m_0=\int_{(0,\infty)}h_0(x)dx>0. \label{S5EL467}
\end{equation}
Then, for any $\tau _0\ge 0$  there exist $R_1>0$, $C_1>0$  such that
\begin{equation}
\int_{[0,r]}h(\tau ,x)d x\geq C_1\,r\qquad\forall r\in[0,R_1],\quad\forall\tau\geq \tau _0.
\end{equation}
\end{lemma} 

\begin{proof}
By (\ref{S5EL467}), there exist $0<a\leq b<\infty$ such that 
\begin{align}
\label{mass of g0}
\int_{(a,b]}h_0(x)dx>\frac{m_0}{2}.
\end{align}
We prove now
\bear
\exists T'>0;\,\forall \tau\in [0, T'):\,\,\,\int_{\left(\frac{a}{2},2b\right]}h(\tau,x)dx\geq\frac{m_0}{4}. \label{S5EL468}
\eear
To this end we use  (\ref{S1E16ha}) with a test function
 $\varphi\in C^1_c([0,\infty))$ such that $0\leq\varphi\leq 1$, $\varphi=1$ on $(a,b]$ and $\varphi=0$ on $[0,\infty)\setminus\left(\frac{a}{2},2b\right]$ and  (\ref{mass of g0}) to obtain:
\begin{align}
\label{es1}
\int_{\left(\frac{a}{2},2b\right]} h(\tau,x)dx
&\geq\frac{m_0}{2}+\int_0^{\tau}\widetilde{\mathscr{Q}}_3(\varphi,h(\sigma))d\sigma.
\end{align}
Now using (\ref{lemma regularity 1}) and (\ref{MMI}) we deduce
$$
\left|\mathscr{Q}_3^{(2)}(\varphi,h(\sigma))\right|\leq 2\|\varphi'\|_{\infty}\bigg(\frac{\sqrt{M_1(h_0)}}{2}\sigma+\sqrt{M_0(h_0)}\bigg)^4.
$$
Using now $\frac{|\mathcal{L}(\varphi)(x)|}{\sqrt{x}}\leq 3\|\varphi\|_{\infty}\sqrt{x}$
and $M_{1/2}(h)\leq \sqrt{M_0(h)M_1(h)}$,
we have by the conservation of energy and the mass inequality
$$
\left|\widetilde{\mathscr{Q}}_3^{(1)}(\varphi,h(\sigma))\right|\leq 2\|\varphi\|_{\infty}\sqrt{M_1(h_0)}\left(\frac{M_1(h_0)}{2}\sigma+\sqrt{M_0(h_0)}\right).
$$
It follows that  $ \widetilde{\mathscr{Q}}_3(\varphi,h)\in L^1_{loc}(\RR_+)$
and  we deduce  (\ref{S5EL468}) from (\ref{es1}).\\

By Lemma \ref{lemma 2^n} and (\ref{S5EL468}), for any $r\in\left(0,\frac{a}{2}\right]$ and $n\in\NN$ such that $r2^n\in(2b,3b]$ we have
\begin{align}
\int_{[0,r]}h(\tau,x)dx
&\geq\int_0^\tau\frac{1}{4^{n+1}r}\left(\int_{\left(\frac{a}{2},2b\right]}h(\sigma,x)dx\right)^2d\sigma \nonumber\\
&\geq\frac{\tau}{4^{n+1}r}\left(\frac{m_0}{4}\right)^2  \ge  \frac{m_0^2}{4^3(3b)^2} \tau\,r \qquad\forall\tau\in[0,T']. \label{1}
\end{align}
where $\left(\frac{a}{2},2b\right]\subset (r,r2^n]$ has been used. 

For any given   $\tau _0\ge 0$  define $\tau'=\min\{\tau_0,T'\}$. Then by (\ref{S5EP12}) in Proposition \ref{S5P1} 
with $\theta=\frac{1}{2}$ and $R=2r$, we deduce from (\ref{1}):
\begin{equation}
\label{2}
\int_{[0,2r]}h(\tau,x)dx \geq\frac{C\tau'}{2}r \qquad\forall\tau\geq \tau'.
\end{equation}
and this proves the Lemma, where $R_1=a/2$ and $C_1=C\tau '/4$.
\end{proof}

\begin{proposition}
\label{S4P47}
Let $h$ and $h_0$ be as in Lemma \ref{S5L47}. For all $L>0$ and  every $\tau _1>0$  there exists $R_0=R_0(h,L,\tau _1)>0$ such that
\begin{align}
\label{S4EP47}
\int_{[0,R_0]}h(\tau ,x)dx\geq LR_0\qquad\forall\tau\geq\tau_1.
\end{align}
\end{proposition}

\begin{proof}
By Lemma \ref{S5L47} for $\tau _0=\frac{\tau_1}{2}$
\begin{align}
\label{by the lemma}
\exists C_1>0, \, \exists R_1>0;\,\,\int_{[0,r]}h(\tau,x)dx &\geq C_1r,\quad\forall r\in[0,R_1],\,\,\,\forall\tau\geq\frac{\tau_1}{2}.
\end{align}
Now fix an integer $p\geq 2$ such that $C_1p \geq 8 L$. We divide the proof in two parts.  Assume first :
\begin{equation}
\label{assumption 1}
\exists r'\in (0, R_1],\; \exists \tau'\in\left[\frac{\tau_1}{2},\tau_1\right]:\,\,\,\int_{\left[0,\frac{r'}{p}\right]}h(\tau',x)dx\geq \frac{C_1r'}{2}.
\end{equation}
It follows from lemma \ref{S5P1} with $\theta=\frac{1}{2}$ and $R=\frac{2r'}{p}$ that
$$
\int_{\left[0,\frac{2r'}{p}\right]}h(\tau,x)dx\geq\frac{C_1r'}{4}\qquad\forall \tau\geq\tau',
$$
If we take $R_0:=\frac{2r'}{p}$, we have, by our choice of $p$,
$$
\int_{[0,R_0]}h(\tau,x)dx\geq \frac{C_1p}{8}R_0\geq LR_0\qquad\forall \tau\geq\tau',
$$
so  (\ref{S4EP47})  holds.

Assume now that \eqref{assumption 1} does not hold,
then, by (\ref{by the lemma}):
\begin{equation}
\label{assumption 2}
\int_{\left(\frac{r}{p},r\right]}h(\tau,x)dx\geq\frac{C_1r}{2}\qquad\forall r\in(0,R_1],\quad\forall\tau\in\left[\frac{\tau_1}{2},\tau_1\right].
\end{equation}
Take now any $r\in\left(0,\frac{R_1}{p}\right]$, let $n\in\NN$ be the largest integer such that 
$r p^n\in\left(\frac{R_1}{p},R_1\right]$, and consider now the following decomposition
$$(r,rp^n]=\bigcup_{i=p+1}^{p^{n+1}}\left(\frac{r}{p}(i-1),\frac{r}{p}i\right]
=\bigcup_{k=1}^n\bigcup_{i=p^k+1}^{p^{k+1}}\left(\frac{r}{p}(i-1),\frac{r}{p}i\right].$$
By lemma \ref{S5L1} on $(\tau _1/2, \tau _1)$ with $a_i=ri/p$, $i=p+1, \cdots, p^{n+1}$:
\begin{align}
\label{ue1}
\int_{[0,r]}h(&\tau_1,x)dx
\geq\int_{\frac{\tau_1}{2}}^{\tau_1}\left[\frac{p}{2r}\sum_{i=p+1}^{p^{n+1}}\frac{1}{i}\left(\int_{\left(\frac{r}{p}(i-1),\frac{r}{p}i\right]} h(\sigma,x)dx\right)^2\right]d\sigma\nonumber\\
&=\int_{\frac{\tau_1}{2}}^{\tau_1}\left[\frac{p}{2r}\sum_{k=1}^n\sum_{i=p^k+1}^{p^{k+1}}\frac{1}{i}\left(\int_{\left(\frac{r}{p}(i-1),\frac{r}{p}i\right]} h(\sigma,x)dx\right)^2\right]d\sigma\nonumber\\
&\geq\int_{\frac{\tau_1}{2}}^{\tau_1}\left[\frac{1}{2r}\sum_{k=1}^n\frac{1}{p^k}\sum_{i=p^k+1}^{p^{k+1}}\left(\int_{\left(\frac{r}{p}(i-1),\frac{r}{p}i\right]} h(\sigma,x)dx\right)^2\right]d\sigma.
\end{align}
We use now  Lemma  3.12 in \cite{AV1}
\begin{align*}
&\sum_{i=p^k+1}^{p^{k+1}}\left(\int_{\left(\frac{r}{p}(i-1),\frac{r}{p}i\right]} h(\sigma,x)dx\right)^2
\geq\frac{1}{p^k(p-1)}\times \\ 
&\times \left(\sum_{i=p^k+1}^{p^{k+1}}\int_{\left(\frac{r}{p}(i-1),\frac{r}{p}i\right]} h(\sigma,x)dx\right)^2
\geq\frac{1}{p^{k+1}}\left(\int_{(rp^{k-1},rp^k]}h(\sigma,x)dx\right)^2
\end{align*}
and deduce
\begin{align*}
\int_{[0,r]}h(\tau_1,x)dx\geq\int_{\frac{\tau_1}{2}}^{\tau_1}\left[\frac{1}{2r}\sum_{k=1}^n\frac{1}{p^{2k+1}}\left(\int_{(rp^{k-1},rp^k]}h(\sigma,x)dx\right)^2\right]d\sigma.
\end{align*}
Due to the choice of the integer $n$, $r p^k\in (0, R_1]$ for all $k=1, \cdots, n$, and we can use (\ref{assumption 2}) on each interval
$(rp^{k-1},rp^k]$ to obtain:
\begin{align*}
\int_{[0,r]}h(\tau_1,x)dx
&\geq\int_{\frac{\tau_1}{2}}^{\tau_1}\left[ \frac{1}{2r}\sum_{k=1}^n\frac{1}{p^{2k+1}}\left(\frac{C_1rp^k}{2}\right)^2\right]d\sigma
=\frac{\tau_1C_1^2n}{16p}\;r.
\end{align*}
It then follows from lemma \ref{S5P1} with $\theta=\frac{1}{2}$ and $R=2r$ that
\begin{align}
\label{estimate unbounded}
\int_{[0,2r]}h(\tau,x)dx\geq\frac{\tau_1C_1^2n}{32p}\;r\qquad\forall\tau\geq\tau_1.
\end{align}
Since $r p^n\in\left(\frac{R_1}{p},R_1\right]$, then $n\geq\frac{\log\left(\frac{R_1}{rp}\right)}{\log(p)}$, and
we chose $r>0$ small enough in order to have $r\in (0, R_1/p)$ and 
$$\frac{\tau_1C_1^2}{64p}\frac{\log\left(\frac{R_1}{rp}\right)}{\log p}\geq L;$$
and set $R_0:=2r$. The result then follows from (\ref{estimate unbounded}).
\end{proof}

\begin{lemma}
\label{rescaled}
Let $h$ be a solution of (\ref{S1E16ha}) and, for any $\kappa>0$ and $\lambda>0$, consider the rescaled measure
$h_{\kappa,\lambda}$ defined as:
\begin{equation}
\label{S5Escaled}
\int_{[0,\infty)} \!\!\!\! h_{\kappa,\lambda}(\tau,x)\varphi(x)dx
=\kappa\!\int_{[0,\infty)}\!\!\!\!h(\kappa\lambda\tau, x)\varphi\left(\frac{x}{\lambda}\right)dx,\;\forall \varphi \in C_b([0, \infty)).
\end{equation}
Then $h_{\kappa,\lambda}$ is a super solution of (\ref{S1E16ha}).
\end{lemma}

\begin{proof}
Let $\varphi\in C^1_b([0,\infty))$ be nonnegative, convex and decreasing, $\psi(x)=\varphi(x/\lambda)$, and $\eta=\kappa\lambda\tau$.
By Lemma \ref{convex-positivity}, 
$\widetilde{\mathscr{Q}}_3^{(1)}(\psi,h)\leq 0$, and by (\ref{S1E16ha})
\begin{align*}
\frac{d}{d\eta}&\int_{[0,\infty)}\psi(x)h(\eta, x)dx\geq \mathscr{Q}_3^{(2)}(\psi,h(\eta)).
\end{align*}
Since $\mathscr{Q}_3^{(2)}(\psi,h(\eta))=\kappa^{-2}\lambda^{-1}\mathscr{Q}_3^{(2)}(\varphi,h_{\kappa,\lambda}(\tau))$, then
\begin{align*}
\frac{d}{d\tau}\int_{[0,\infty)} \varphi(x)h_{\kappa,\lambda}(\tau,x)dx&=\kappa^2\lambda\frac{d}{d\eta}\int_{[0,\infty)}\psi(x)h(\eta, x)dx
\geq\mathscr{Q}_3^{(2)}(\varphi,h_{\kappa,\lambda}(\tau)).
\end{align*}
\end{proof}

\begin{lemma}
\label{concentration lemma}
Let $h$ be a super solution of (\ref{S1E16ha}).
Suppose that there exists $\tau'>0$ such that
\begin{equation}
\label{S5HC}
\int_{[0,1]}h(\tau,x)dx\geq 1\qquad\forall \tau\geq\tau'.
\end{equation}
Then for any given $\delta>0$ there exist  $\tau _0$ such that
\begin{align}
&\tau '\leq\tau _0\leq\tau '+T_0(\delta),\qquad T_0(\delta )=\frac{64}{\delta^3}\left(1-\frac{\delta}{2}\right) \label{S5EX2}\\
&\hbox{and}\,\,\,\,\int_{\left[0,\frac{\delta}{4}\right]}h(\tau_0,x)dx\geq 1-\frac{\delta}{2}. \label{S5EX3}
\end{align}
\end{lemma}

\begin{proof}
 The statement of the Lemma is  equivalent to show that the following set
$$
A:=\left\{\tau\in[\tau', \tau '+T_0( \delta )]:\int_ {\left[0,\frac{\delta}{4}\right]}h(\tau,x)dx\geq 1-\frac{\delta}{2}\right\}.
$$
is non empty, where $T_0(\delta )$ is defined in (\ref{S5EX2}). 
To this end we first apply  Lemma \ref{S5L1}  with  $a_0=\frac{\delta}{4}$, $a_i=\frac{\delta}{4}\left(1+\frac{i}{2}\right)$ for $i\in\{1,...,n-1\}$ and $a_n=1$. The number $n$ is chosen to be the largest integer such that $a_{n-1}<1$, which implies
\begin{equation}
\label{S5e0}
\frac{1}{n+1}>\frac{\delta}{8}.
\end{equation}
Then, using ${a_i}^{-1}\geq 1$ for all $i\in\{1,...,n\}$:
\begin{align*}
\int_{\left[0,\frac{\delta}{4}\right]}h(\tau,x)dx
\geq\frac{1}{2}\int_{\tau'}^{\tau} \sum_{i=1}^n\bigg(\int_{(a_{i-1},a_i]}h(\sigma,x)dx\bigg)^2d\sigma,\,\,\,\forall \tau >\tau '.
\end{align*}
Since by Lemma 3.12 in \cite{AV1} and (\ref{S5e0}):
\begin{align*}
\sum_{i=1}^n\bigg(\int_{(a_{i-1},a_i]}h(\sigma,x)dx\bigg)^2
\geq\frac{\delta}{8}\bigg(\int_{\left(\frac{\delta}{4},1\right]}h(\sigma,x)dx\bigg)^2,
\end{align*}
we obtain, for all $\tau >\tau '$
\begin{align}
\label{estimate for contradiction}
\int_{\left[0,\frac{\delta}{4}\right]}h(\tau,x)dx
\geq\frac{\delta}{16}\int_{\tau'}^{\tau}\bigg(\int_{\left(\frac{\delta}{4},1\right]}h(\sigma,x)dx\bigg)^2d\sigma.
\end{align}
Arguing by contradiction suppose that $A=\emptyset$:
\begin{align*}
\int_{\left(0,\frac{\delta}{4}\right]}h(\tau ,x)dx< 1-\frac{\delta}{2}\qquad\forall\tau \in[\tau',\tau '+T_0(\delta )]
\end{align*}
and by (\ref{S5HC}):
$$
\int_{\left(\frac{\delta}{4},1\right]}h(\tau ,x)dx\geq \frac{\delta}{2}\qquad\forall\tau \in[\tau',\tau '+T_0(\delta )].
$$
It  follows from (\ref{estimate for contradiction}) that
$
1-\frac{\delta}{2}>\frac{\delta^3}{64}(\tau-\tau')$ for all $\tau\in[\tau',\tau '+T_0(\delta )]$
which is a contradiction for $\tau=\tau '+T_0(\delta )$. 
\end{proof}

\begin{proposition}
\label{S5LB}
Let $h$ be a solution of  (\ref{S1E16ha}). Suppose that there exist $m$, $R>0$ such that
\begin{equation}
\int_{[0,R]}h(\tau,x)dx\geq m\qquad\forall\tau\in[0,\infty).
\end{equation}
Then given any $\alpha\in(0,1)$ there exists $T_*=T_*(\alpha)>0$ such that
\begin{align}
\label{S5ELB}
\int_{[0,r]}h(\tau,x)dx\geq \frac{m}{(2R)^{\alpha}}\,r^{\alpha}\qquad\forall r\in[0,R],\quad\forall\tau\in\bigg[\frac{R T_*}{m},\infty\bigg).
\end{align}
\end{proposition}
\begin{proof}
We argue by induction and define first the scaled measure  $h_1=h _{ \kappa_1, \lambda _1 } $, defined as in (\ref{S5Escaled}), that satisfies condition (\ref{S5HC}) for $\kappa_1=\frac {1} {m},\,\,\,\lambda _1=R.$
From Lemma \ref{rescaled}, and Lemma \ref{concentration lemma} with $\tau '=0$, we deduce that for all $\delta \in (0, 1)$ there exists  $\tau _1>0$ such that:
\begin{align*}
&0\leq\tau _1\leq T_0(\delta),\quad \int_{\left[0,\frac{\delta}{4}\right]}h_1(\tau_1,x)dx\geq 1-\frac{\delta}{2}.
\end{align*}
Then from Lemma \ref{rescaled}, and Proposition \ref{S5P1} with $\theta=\delta /2$ and $R=1/2$,
\begin{align}
\int_{\left[0,\frac{1}{2}\right]}h_1(\tau,x)dx\geq \left(1-\frac{\delta}{2}\right)^2,\,\,\,\forall \tau \geq T_0(\delta ),\nonumber\\
\label{S5EIt1}
\int_{\left[0,\frac{R}{2}\right]}h\left(\tau , x\right)dx\geq m\left(1-\delta \right),\,\,\,\forall \tau \geq \frac {R} {m}T_0(\delta ).
\end{align}
Exactly as before we now define $h_2=h _{ \kappa_2, \lambda _2 } $ as in (\ref{S5Escaled}), that satisfies condition (\ref{S5HC}) for
$\kappa_2=\frac {1} {m(1-\delta )^2},\,\,\,\lambda _2=\frac {R} {2},\,\,\,\tau '= 2(1-\delta )T_0(\delta ).$
The same argument gives then:
\begin{equation}
\label{S5EIt2}
\int_{\left[0,\frac{R}{4}\right]}h\left(\tau , x\right)dx\geq m\left(1-\delta \right)^2,\quad\forall \tau \geq \frac {R T_0(\delta )} {m}\left(1+\frac {1} {2(1-\delta )} \right).
\end{equation}
We deduce after $n$ iterations
\begin{align}
\label{S5EItn}
\int_{\left[0,\frac{R}{2^n}\right]}h\left(\tau , x\right)dx\geq m\left(1-\delta \right)^{n},\quad\forall \tau \geq \frac {R T_0(\delta )} {m}
\sum_{ k=0 }^{n-1}\frac {1} {2^{k}(1-\delta )^k}
\end{align}
If we chose $\delta=1-2^{-\alpha }$, for any $0<\alpha <1$, we may  define
\bear
\label{S5ET*}
T_*=T_0(\delta )\sum_{ k=0 }^{\infty}2^{-(1-\alpha)  k}= \frac {T_0(\delta )} {1-2^{-(1-\alpha) }}.
\eear
Since for any $r\in (0, R)$ there exists $n\in \NN$ such that $r\in \left(\frac{R}{2^n} ,\frac{R}{2^{n-1}}\right]$,
\begin{align*}
\int_{\left[0, r \right]}h\left(\tau , x\right)dx\geq m2^{-\alpha n} 
,\,\,\,\forall \tau > \frac {R T_*} {m}
\end{align*}
and using $2^{-n}>r/2R$, (\ref{S5ELB}) follows.
\end{proof}

\begin{proposition}
\label{S5PLB}
Let $h$ be a solution of  (\ref{S1E16ha}).  Then, for all $\tau_0>0$ and for any  $\alpha\in(0,1)$ there exists 
$R_*=R_*(h,\tau_0,\alpha)>0$ such that
\begin{align}
\int_{[0,r]}h(\tau,x)\dd x\geq C\,r^{\alpha}\qquad\forall r\in[0,R_*]\quad\forall \tau\in[\tau_0,\infty),
\end{align} 
where $C=\frac{T_*(\alpha)}{\tau_0}(2R_*)^{1-\alpha}$, and $T_*(\alpha)$ is given by Proposition \ref{S5LB}. 
\end{proposition}
\begin{proof}
By Proposition \ref{S4P47} with $L>0$ and for $\tau _1=\tau _0/2 $, there exists $R_0(h,L,\tau _1)>0$ such that
\begin{align*}
\int_{[0,R_0]}h(\tau ,x)dx\geq LR_0\qquad\forall\tau\geq\frac{\tau_0}{2}.
\end{align*}
Then by Proposition \ref{S5LB}, with $m=LR_0$ and $R=R_0$, we obtain that for any given $\alpha\in(0,1)$ there exists $T_*=T_*(\alpha)>0$ such that
\begin{align*}
\int_{[0,r]} h(\tau,x)\dd x\geq \frac{LR_0}{(2R_0)^{\alpha}}\,r^{\alpha}\qquad\forall r\in[0,R_0],\quad\forall \tau\in\bigg[\frac{\tau_0}{2}+\frac{ T_*}{L},\infty\bigg).
\end{align*}
If we chose $L=2T_*/\tau_0$, then the Proposition follows with $R_*=R_0$.
\end{proof}

\begin{proof}[\upshape{\bfseries{Proof of Theorem \ref{S1T4h}}}]
By Lemma \ref{S6L1'} the map $\tau\mapsto h(\tau,\{0\})$ is right continuous, nondecreasing and a.e. differentiable on $[0,\infty)$. It remains to prove that it is actually strictly increasing.
We  first suppose that  $h_0$ is such that 
\bear
\label{S5E54}
\int_{\{0\}}h_0(x)dx=0,\qquad \int_{(0,\infty)}h_0(x)dx>0,
\eear
and prove 
\begin{equation}
\label{S1ET4h}
h(\tau,\{0\})>0\qquad\forall \tau>0.
\end{equation}
Arguing by contradiction, if we suppose that there exists $\tau_0>0$ such that $h(\tau_0,\{0\})=0$, by monotonicity 
$h(\tau,\{0\})=0$ for all $\tau\in[0,\tau_0]$. In particular
\begin{align}
\int_{\frac{\tau_0}{2}}^{\tau_0}\int_{[0,r]}h(\sigma,x)dxd\sigma=\int_{\frac{\tau_0}{2}}^{\tau_0}\int_{(0,r]}h(\sigma,x) dxd\sigma
\end{align}
for all $r>0$. Now using  Proposition \ref{S5P2} with $\alpha=0$, and Proposition \ref{S5PLB},  we deduce that, for any $\alpha\in(0,1/2)$, there exists $R_*=R_*(h, \tau _0/2, \alpha )$ such that
\begin{align*}
&C_2\,r^{\alpha}\leq \int_{\frac{\tau_0}{2}}^{\tau_0}\int_{(0,r]}h(\sigma,x)dxd\sigma \leq C_1\sqrt{r},
\qquad\forall r\in[0,R_*];\\
&C_1=8\, \sqrt{\frac{\tau_0}{2}} \left( \frac {\sqrt{M_1(h_0)}} {2}\tau_0 +\sqrt{M_0(h_0)}\right),\quad C_2=\frac{T_*(\alpha)}{2}(2R_*)^{1-\alpha},
\end{align*}
and that leads to a contradiction for $r$ small enough.

Consider now a general initial data $h_0$ such that $\int_{\{0\}}h_0(x)dx>0$. Let $h$ be a solution of (\ref{S1E16ha}) with initial data $h_0$ and define
$$
\tilde h(\tau )=h(\tau )-h_0(\{0\})\delta _0.
$$
Then, on the one hand, the initial data of $\tilde h$  satisfies $\tilde h(0, \{0\})=0$. On the other hand we claim that $\tilde h$ is still a solution of   (\ref{S1E16ha}). Notice indeed that $\tilde h_\tau \equiv h_\tau $ and, moreover,
$\widetilde{\mathscr{Q}}_3(\varphi,h(\tau))=\widetilde{\mathscr{Q}}_3(\varphi,\tilde h(\tau))$. Using the previous case
$$
\int  _{ \{0\}}\tilde h(\tau, x)dx>0,\quad\forall \tau >0,
$$
and then
$$
\int  _{ \{0\}} h(\tau, x)dx>\int  _{ \{0\}} h_0(x)dx,\quad\forall \tau >0.
$$
The Theorem follows using now the time translation invariance of the equation.
\end{proof}

The last result of this section describes the relation between the Lebesgue-Stieltjes measure associated to the (right continuous and strictly increasing) function
$m(\tau)=h(\tau,\{0\})$, and the equation for $h$ (\ref{S1E16ha}).

\begin{proposition}
\label{Stieltjes1}
Let $h$ be a solution of (\ref{S1E16ha}) for a initial data $h_0\in\mathscr{M}_+^1([0,\infty))$ with $N=M_0(h_0)>0$ and $E=M_1(h_0)>0$.
If we denote $m(\tau)=h(\tau,\{0\})$ and $\lambda$ is the Lebesgue-Stieltjes measure associated to $m$, then for all $\varphi_{\varepsilon}$ as in Remark \ref{TEST} and for all $\tau_1$ and $\tau_2$ with $0\leq\tau_1<\tau_2$:
\begin{align}
\label{mm1}
&m(\tau_2)-m(\tau_1)=\lambda((\tau_1,\tau_2]),\\
\label{Stieltjes3}
&\lambda((\tau_1,\tau_2])=\lim_{\varepsilon\to 0}\int_{\tau_1}^{\tau_2}\mathscr{Q}_3^{(2)}(\varphi_{\varepsilon},h(\tau))d\tau,
\end{align}
\begin{flalign}
\label{Stieltjes0}
&\text{and}&&0<\lambda((\tau_1,\tau_2]))<\infty.&&
\end{flalign} 
Furthermore, for all $\varphi_{\varepsilon}$ as in Remark \ref{TEST}
\begin{align}
\label{Stieltjes7}
\lim_{\varepsilon\to 0}\mathscr{Q}_3^{(2)}(\varphi_{\varepsilon},h)\in\mathscr{D}'(0,\infty),
\end{align}
and if we denote $m'$ the  derivative in the sense of Distributions of $m$, then 
\begin{align}
\label{Stieltjes2}
m'=\lambda=\lim_{\varepsilon\to 0}\mathscr{Q}_3^{(2)}(\varphi_{\varepsilon},h)\quad\text{in}\quad\mathscr{D}'(0,\infty).
\end{align}
\end{proposition}

\begin{proof}
By Lemma \ref{S6L1'}, $m$ is right continuous and nondecreasing on $[0,\infty)$. Then it has a Lebesgue-Stieltjes measure associated to it, $\lambda$, that satisfies (\ref{mm1}) (c.f.  \cite{Fo} Ch.1).

On the other hand, since $h$ is a solution of (\ref{S1E16ha}), using $\varphi_{\varepsilon}$ as in Remark \ref{TEST} and taking the limit $\varepsilon\to 0$, it follows from (\ref{limQ31b}) in Lemma \ref{convergence lemma} that for all $\tau_1$ and $\tau_2$ with $0\leq\tau_1<\tau_2$:
\begin{align}
\label{ZE3}
m(\tau_2)-m(\tau_1)=\lim_{\varepsilon\to 0}\int_{\tau_1}^{\tau_2}\mathscr{Q}_3^{(2)}(\varphi_{\varepsilon},h(\sigma))d\sigma,
\end{align}
and then (\ref{Stieltjes3}) follows from (\ref{mm1}). Moreover, since by Theorem \ref{S1T4h} $m$ is strictly increasing, then (\ref{Stieltjes0}) holds.

Notice that the limit in (\ref{ZE3}) is independent of the choice of the test function $\varphi_{\varepsilon}$. Indeed, if $\psi_{\varepsilon}$ is another test function as in Remark \ref{TEST}, since for all $\tau\geq 0$
\begin{align*}
\lim_{\varepsilon\to 0}\int_{[0,\infty)}\psi_{\varepsilon}(x)h(\tau,x)dx=m(\tau)=\lim_{\varepsilon\to 0}\int_{[0,\infty)}\varphi_{\varepsilon}(x)h(\tau,x)dx,
\end{align*}
it follows from (\ref{ZE3}) that for all $0\leq \tau_1\leq\tau_2$
\begin{align*}
\lim_{\varepsilon\to 0}\int_{\tau_1}^{\tau_2}\mathscr{Q}_3^{(2)}(\psi_{\varepsilon},h(\sigma))d\sigma=
\lim_{\varepsilon\to 0}\int_{\tau_1}^{\tau_2}\mathscr{Q}_3^{(2)}(\varphi_{\varepsilon},h(\sigma))d\sigma.
\end{align*}

Now, for all $\varphi_{\varepsilon}$ as in Remark \ref{TEST}, consider the absolutely continuous function 
\begin{align*}
\theta_{\varepsilon}(\tau)=\int_{[0,\infty)}\varphi_{\varepsilon}(x)h(\tau,x)dx.
\end{align*}
Then the equation in (\ref{AUXW}) reads $\theta_{\varepsilon}'(\tau)=\widetilde{\mathscr{Q}}_3(\varphi_{\varepsilon},h(\tau))
$.
Using integration by parts we deduce that for all $\varepsilon>0$:
\begin{align*}
-\int_0^{\infty}\phi'(\tau)\theta_{\varepsilon}(\tau)d\tau=\int_0^{\infty}\phi(\tau)\widetilde{\mathscr{Q}}_3(\varphi_{\varepsilon},h(\tau))d\tau
\quad\forall \phi\in C^{\infty}_c(0,\infty).
\end{align*}
Taking the limit $\varepsilon\to 0$ it then follows from Lemma \ref{convergence lemma} that
\begin{align*}
-\int_0^{\infty}\phi'(\tau)m(\tau)d\tau=\lim_{\varepsilon\to 0}\int_0^{\infty}\phi(\tau)\mathscr{Q}_3^{(2)}(\varphi_{\varepsilon},h(\tau))d\tau,
\end{align*}
hence, $m'=\lim_{\varepsilon\to 0}\mathscr{Q}_3^{(2)}(\varphi_{\varepsilon},h)$. On the other hand, by Fubini's theorem
\begin{align*}
\int_0^{\infty}\phi(\tau)d\lambda(\tau)=\int_0^{\infty}\int_0^{\tau}\phi'(\sigma)d\sigma d\lambda(\tau)
=-\int_0^{\infty}\phi'(\sigma)m(\sigma)d\sigma
\end{align*}
for all $\phi\in C^{\infty}_c(0,\infty)$ (cf. \cite{WR3}, Example 6.14), thus $m'=\lambda$.
\end{proof}

\section{Existence of solutions $G$, proof of  Theorem \ref{S1T1}.}
\label{sectionG}
\setcounter{equation}{0}
\setcounter{theorem}{0}

Given a initial data $G_0\in \mathscr M_+^1$ as in Theorem \ref{S1T1}, let $h\in C\big([0,\infty), \mathscr{M}_+([0,\infty))\big)$ satisfy (\ref{lip loc h})--(\ref{EE}),  (\ref{S5Ealpha }), given by (\ref{S5C52R}) and $H$ defined by (\ref{DEFH}) and satisfying  (\ref{lip loc H})--(\ref{EEH}), (\ref{S5EalphaR }) by Corollary \ref{S5C52R}.  It is natural, in view of the change of variables  (\ref{S1E45b})  to define now, 
\begin{equation}
\label{S6EG1}
G(t) =H(\tau ),\quad\tau =\int _0^tG(s, \{0\})ds.
\end{equation}
Notice nevertheless that since $G(s, \{0\})$ is still unknown,  (\ref{S6EG1}) does not define $G(t)$ actually.  What we know is rather, given $\tau >0$, what would be the value of $t$ such that 
\begin{equation}
\label{S6EtCh}
t =\int _0^\tau \frac {d\sigma  } {H(\sigma  , \{0\})},
\end{equation}
since we expect to have $G(s,\{0\})=H(\sigma , \{0\})$ for $s$ and $\sigma $ such that
\bean
\sigma =\int _0^sG(r, \{0\})dr,\quad\hbox{or}\quad s=\int _0^\sigma \frac {d\rho } {H(\rho , \{0\})}.
\eean
If $G$ is going to be defined in that way it is then necessary first to check that the range of values taken by the variable $t$ in (\ref{S6EtCh}) is all of $[0, \infty)$. By definition (\ref{DEFH}),
\begin{equation}
\label{{DEFH}2}
H(\tau,\{0\})=h(\tau, \{0\})-\int_0^\tau M_{1/2}(h(\sigma))d\sigma.
\end{equation}
Since both terms in the right hand side are nonnegative, $H(\tau, \{0\})$ has no a priori  definite sign. We must then consider that question in some detail.
Our first step is to prove the following

\begin{lemma}
\label{S6C1}
If $G_0(\{0\})>0$, then
\begin{align}
\label{TAU1}
&\tau_*=\inf\{\tau>0:H(\tau,\{0\})=0\}>0,\\
\label{TAU2}
&H(\tau_*,\{0\})=0,\\
\label{TAU3}
&H(\tau,\{0\})>0\qquad\forall\tau\in[0,\tau_*).
\end{align}
\end{lemma}
\begin{proof}
$H(0)=G_0$ by (\ref{S5IDh}), and then, using $\varphi_{\varepsilon}$ as in Remark \ref{TEST}, we deduce $H(0,\{0\})=G_0(\{0\})$, which is strictly positive by hypothesis. Then (\ref{TAU1}) follows from the right continuity of $H(\tau,\{0\})$ (cf. Corollary \ref{S6L1}).

In order to prove (\ref{TAU2}) we use a minimizing sequence $(\tau _n)_{ n\in \NN }$, i.e., $\tau _n\ge \tau _*$, $H(\tau _n,\{0\})=0$ for every $n\in \NN$,  and $\tau _n\to \tau _*$ as $n\to \infty$. Then from the right continuity (\ref{TAU2}) holds.

Let us prove now (\ref{TAU3}). If $H(\tau_0 , \{0\})<0$ for some $\tau_0 \in (0, \tau_*)$, then $\tau _0$ must be a left discontinuity point of $H(\tau, \{0\})$ and
\begin{equation*}
\limsup_{\delta\rightarrow 0^+}H(\tau_0-\delta,\{0\})> H(\tau_0,\{0\}),
\end{equation*}
and this would contradict (\ref{JIM}). That proves (\ref{TAU3}).
\end{proof}

It follows from Lemma \ref{S6C1} that the function:
\begin{equation}
t =\xi (\tau )=\int _0^\tau \frac {d\sigma  } {H(\sigma  , \{0\})}
\end{equation}
introduced  in (\ref{S6EtCh}) is well defined, monotone nondecreasing  and continuous on  the interval $ [0, \tau _*)$. We  then define,
\begin{align}
\label{S6DG2}
\forall t\in [0, \xi (\tau_*)):\quad G(t)=H(\xi ^{-1}(t)).
\end{align}
By (\ref{S6DG2}) and (\ref{{DEFH}2}), if $G(t)=G(t,\{0\})\delta _0+g(t)$ and $H(\tau)=H(\tau , \{0\})\delta _0+\tilde h(\tau )$, then
\begin{align}
&G(t,\{0\})=H(\tau , \{0\}), \label{S6DG2bis}\\
&\tilde h(\tau )=h(\tau )-h(\tau , \{0\})\delta_0, \label{S6DG20}\\
&g(t)=\tilde h(\tau ).\label{ght}
\end{align}

\begin{remark}
Formula (\ref{S6DG2}) defines the function $G$ at time $t\in (0, \xi (\tau _*))$ from the knowledge of the function $H(\tau )$ for $\tau >0$ such that $\tau =\xi ^{-1}(t)$. Moreover,
\begin{equation}
\forall t\in (0, \xi (\tau _*)): \quad \xi ^{-1}(t)=\int _0^t G(s, \{0\})ds.
\end{equation}
\end{remark}
We have now,

\begin{proposition}
\label{S6P3G1}
The function  $G$ defined by (\ref{S6DG2}) is  such that
\bear
G\in C\big([0, \xi (\tau _* )),\mathscr{M}_+^1([0,\infty))\big),\,\,\,G(0)=G_0
\eear 
and satisfies (\ref{S1ED3S}), (\ref{S1E16}), (\ref{S1E210}) and (\ref{S1E220}) on the time interval $[0, \xi (\tau_* ))$.
\end{proposition}

\begin{proof}
We first prove that $G(t)$ is a positive measure for all $t\in[0,\xi(\tau_*))$. By (\ref{TAU3}) and (\ref{S6DG2bis}) we have $G(t,\{0\})>0$ for all $t\in[0,\xi(\tau_*))$. Then, since $h(\tau)$ is a positive measure for all $\tau\in[0,\infty)$, we deduce from (\ref{ght}) and (\ref{S6DG20}) that $g(t)$ is a positive measure for all $t\in[0,\xi(\tau_*)).$ Hence $G(t)=G(t,\{0\})\delta_0+g(t)$ is also positive.

All the properties of $G(t)$ at $t\in [0, \xi (\tau_*))$ fixed follow from the corresponding property of $H(\tau )$ with $t=\xi(\tau )$. The only property where $t$ is not fixed are (\ref{S1ED2S}) and (\ref{S1ED3S}). Since
\begin{equation*}
\left|\frac {\partial G(t)} {\partial t}\right|= \left|\frac {\partial \tau } {\partial t}\frac {\partial H(\tau )} {\partial \tau }\right|
\le |H(\tau , \{0\})|\left|\frac {\partial H(\tau )} {\partial \tau }\right|
\end{equation*}
By definition,
\bean
|H(\tau , \{0\})|\le |h(\tau, \{0\})|+\int_0^\tau M_{1/2}(h(\sigma))d\sigma.
\eean
Since $h\in C([0, \infty), \mathscr M_+^1)$ it follows using also (\ref{MMI}),  (\ref{EE}) and H\"older inequality that 
$H(\tau , \{0\})\in L^\infty _{ loc }([0, \infty))$.Then, by  (\ref{lip loc h}), $G(t)$ is locally Lipschitz on $[0, \xi (\tau _*))$ and satisfies (\ref{S1ED3S}).
Since $H$ satisfies (\ref{AUXW}) the change of variables ensures that $G$ satisfies (\ref{S1E16}).
\end{proof}
We prove now the following property of the function $G$ defined in (\ref{S6DG2}).
\begin{proposition}
\label{origin G}
Let  $G$ be the function defined in  (\ref{S6DG2}) for $t\in (0, \xi (\tau _*))$. Then the map $t\mapsto G(t,\{0\})$ is right continuous and differentiable for almost every   $ t\in [0,\xi (\tau _*))$ and, for all $t_0\in (0, \xi (\tau _*))$ 
\begin{align}
\label{CMI}
G(t,\{0\})\geq G(t_0,\{0\})e^{-\int_{t_0}^tM_{1/2}(g(s))d s}\quad\forall t\in (t_0, \xi (\tau _*)).
\end{align} 
In particular, if $G(0, \{0\})>0$, then $G(t,\{0\})> 0$ for all $t\in (0, \xi (\tau _*))$.  
\end{proposition}
\begin{proof}
Using (\ref{S1E16}) and (\ref{S1EB1}) with $\varphi _{\varepsilon}$ as in Remark \ref{TEST}, we have
\begin{align}
\frac{d}{dt}\int_{[0,\infty)}\varphi_{\varepsilon}(x)G(t,x)d x+G(t,\{0\})M_{1/2}&(G(t))
=G(t,\{0\}) \widetilde{\mathscr Q}_3(\varphi _\varepsilon,G(t))
.\label{S6E78}
\end{align}
We use now that for all $\varepsilon >0$:
\begin{align}
\label{ZAS}
G(t,\{0\})\leq \int_{[0,\infty)}\varphi_{\varepsilon}(x)G(t,x)d x,
\end{align}
and we deduce from (\ref{S6E78}), using  $J(t)=\exp\left({\int_0^t M_{1/2}(G(s))d s}\right)$,
\begin{align}
\label{S6E79}
&\frac{d}{d t}\bigg(J(t)\int_{[0,\infty)}\varphi_{\varepsilon}(x)G(t,x)d x\bigg) \geq G(t,\{0\})J(t)  \widetilde{\mathscr Q}_3(\varphi _\varepsilon,G(t)).
\end{align}
By Lemma \ref{convex-positivity} the right hand side of (\ref{S6E79}) is nonnegative, and we deduce
\begin{equation*}
\label{S6E80}
J(t)\int_{[0,\infty)}\varphi_{\varepsilon}(x)G(t,x)d x \geq J(t_0)\int_{[0,\infty)}\varphi_{\varepsilon}(x)G(t_0,x)d x
\end{equation*}
for all $t\in (t_0, \xi (\tau _*))$ and all $\varepsilon >0$.  If we pass now to the limit as $\varepsilon \to 0$:
\begin{equation}
\label{S6E80}
J(t)G(t,\{0\}) \geq J(t_0)G(t_0,\{0\}),
\end{equation}
and this proves the estimate (\ref{CMI}). It  also follows from Lebesgue's Theorem that $J(t)G(t,\{0\})$ is differentiable for almost every $t\in (0, \xi (\tau _*))$ (cf. \cite{Roy},  Theorem 2). On the other hand, since $J(t)$ is a.e differentiable and $J(t)>0$ for all $t\in [0, \xi (\tau _*))$, we deduce that $G(t,\{0\})$ is also differentiable for almost every 
$t\in [0, \xi (\tau _*))$.

We prove now the  right continuity  of $G(t, \{0\})$. It follows from  (\ref{S6E80}),
\begin{equation*}
J(t+\delta)G(t+\delta,\{0\}) \geq J(t)G(t,\{0\}),\quad\forall \delta>0\,\,\forall t>0.
\end{equation*}
If we take inferior limits and use that $J$ is continuous and strictly positive we obtain,
\begin{equation}
\label{S6E578}
\liminf _{ \delta\to 0 }G(t+\delta,\{0\}) \geq G(t,\{0\}),\quad\forall t>0.
\end{equation}

Since $\mathcal L_0(\varphi_{\varepsilon})\ge 0$ by convexity (cf. Lemma \ref{convex-positivity}), we deduce 
\begin{align*}
\frac{d}{d t}&\int_{[0,\infty)}\varphi_{\varepsilon}(x)G(t,x)d x
\leq G(t,\{0\})\iint_{(0,\infty)^2}\frac{\Lambda(\varphi_{\varepsilon})(x,y)}{\sqrt{x y}}G(t,x)G(t,y)d xd y,
\end{align*}
and the argument  follows now as in the proof of the right continuity of $H$. 
From the inequality (\ref{ZAS}), the bound \eqref{lemma regularity 1} and the conservation of mass, we deduce for all  $t\in [0, \xi (\tau _*))$ fixed  and  $\delta\in[0, \xi (\tau _*)-t)$,
$$
G(t+\delta,\{0\})\leq\int_{[0,\infty)}\varphi_{\varepsilon}(x)G(t,x)d x
+\frac {2N^2\delta} {\varepsilon }\int_0^t G(s,\{0\})ds.
$$
If we take superior limits as $\delta\to 0$, and then let $\varepsilon \to 0$ we obtain, using  (\ref{S6L1E1}) with $G$ instead of $H$:
$$
\limsup _{ \delta\to 0 }G(t+\delta,\{0\})\leq G(t, \{0\}).
$$
and this combined with (\ref{S6E578}) proves that  $G(t, \{0\})$ is right continuous on $[0, \xi (\tau _*))$.
\end{proof}

In the next Lemma we prove that the function $G$ defined by (\ref{S6DG2}) is actually well defined for all $t>0$.
\begin{lemma}
\label{S6LG}
\begin{equation}
\lim _{ \tau \to \tau _*^- }\xi (\tau )=\infty.
\end{equation}
\end{lemma}
\begin{proof}
Since the function $\xi (\tau )$ is monotone  nondecreasing and continuous on $[0, \tau _*)$, its limit as $\tau \to \tau _*^-$ exists in $\overline{\RR}_+$. Let us call it $\ell$ and  suppose  $\ell\in \RR_+$. Now, from (\ref{CMI}) and the fact that $G$ satisfies :
$0\le M _{ 1/2 }(G(s))\le \sqrt {N E}$,
we deduce 
\begin{equation}
\label{S6C56E1}
\limsup _{t\to \ell^-  }G(t, \{0\})\ge e^{-\sqrt{NE} \ell}G(0, \{0\})>0,
\end{equation}
and by (\ref{JIM})
$$
H(\tau_*, \{0\})\ge \limsup _{\tau \to \tau _*^-  }H(\tau, \{0\})=\limsup _{ t\to \ell^- }G(t, \{0\})>0,
$$
and this contradicts (\ref{TAU2}). This proves that $\ell=\infty$.
\end{proof}

\begin{proof}[\upshape\bfseries{Proof of Theorem \ref{S1T1}}] 
 By Lemma \ref{S6LG} the function $G$ is defined for all $t>0$. As we have seen in the proof of Lemma \ref{S6LG}, $G(t)\in \mathscr M_+([0, \infty))$ for all $t>0$. It then follows from Proposition \ref{S6P3G1} that $G$ satisfies now all the conditions (\ref{S1ED2S})--(\ref{S1E16}) and (\ref{S1T1E0})--(\ref{S1E220}).
Property (\ref{S1E23}) follows from the corresponding estimate (\ref{MAh}) for $h$. Similarly, property (\ref{S5Ealphahh}) follows from the property (\ref{S5Ealpha }) of $h$.
We prove now the point (iv). Suppose then $\alpha \in (1, 3]$ and condition (\ref{PRO112}). For $\varphi(x)=x^{\alpha}$ we have,
$$
\mathscr{Q}_3^{(1)}(\varphi,G(t))=\left(\frac{\alpha-1}{\alpha+1}\right) M_{\alpha+\frac{1}{2}}(G(t)).
$$
On the other hand, for $0\le y\le x$,
$$
\Lambda(\varphi)(x,y)=x^{\alpha}\left(\big(1+z\big)^{\alpha}+\big(1-z\big)^{\alpha}-2\right),\quad z=\frac{y}{x}\in[0,1],
$$
If  $\alpha\in(1,2]$, for all $x\ge y >0$,
$$
\frac{\Lambda(\varphi)(x,y)}{\sqrt{xy}}\leq (2^{\alpha}-2) x^{\alpha-\frac{3}{2}}y^{\frac{1}{2}}\le  (2^{\alpha}-2) (xy)^{\frac{\alpha-1}{2}}.
$$
We deduce
$$
\mathscr{Q}_3^{(2)}(\varphi,G(t))\leq (2^{\alpha}-2) \Big( M_{\frac{\alpha-1}{2}}(G(t))\Big)^2.
$$
and obtain
\begin{align*}
\frac{d}{dt}M_{\alpha}(G(t))\leq G(t,\{0\})\left[C_{1,1}\Big(M_{\frac{\alpha-1}{2}}(G(t))\Big)^2-C_2M_{\alpha+\frac{1}{2}}(G(t))\right],
\end{align*}
where $C_{1,1}=2^{\alpha}-2$ and $C_2=(\alpha-1)/(\alpha+1).$
Using H\"{o}lder's inequality
\begin{align}
\label{CT98}
\frac{d}{dt}M_{\alpha}(G(t))\leq G(t,\{0\})\Big[C_{1,1}N^{3-\alpha}E^{\alpha-1}
-C_2E^{(2\alpha+1)/2}N^{(1-2\alpha)/2}\Big].
\end{align}
By (\ref{PRO112}), the right hand side of (\ref{CT98}) is negative, and then $M_{\alpha}(G(t))$ is decreasing on $(0,\infty)$.

For $\alpha\in[2,3]$ we use the estimate (\ref{MAQ}) with $C_{1,2}=\alpha(\alpha-1)$ instead of $C_{\alpha}$. Then we proceed as in the previous case to obtain
\begin{equation}
\label{CT99}
\frac{d}{dt}M_{\alpha}(G(t))\leq G(t,\{0\})\Big[C _{ 1,2 }N^{3-\alpha}E^{\alpha-1}
-C_2E^{(2\alpha+1)/2}N^{(1-2\alpha)/2}\Big].
\end{equation}
As before, (\ref{PRO112}) implies that the right hand side of (\ref{CT99}) is negative, and then $M_{\alpha}(G(t))$ is decreasing.
\end{proof}

\begin{proof}[\upshape\bfseries{Proof of Theorem \ref{S1Treg}}] 
By construction
\begin{align*}
G(t)=H(\tau)=h(\tau)-\left(\int_0^{\tau}M_{1/2}(h(\sigma))d\sigma\right)\delta_0,
\end{align*}
where $\tau$ and $t$ are related by
\begin{align}
\label{CHANGE}
t=\xi(\tau)=\int_0^{\tau}\frac{d\sigma}{H(\sigma,\{0\})};\qquad \tau=\xi^{-1}(t)=\int_0^t G(s,\{0\})ds.
\end{align}
Therefore $G(t,x)=h(\tau,x)$ for $x\in(0,\infty)$, and 
\begin{align*}
\int_0^{T}G(t,\{0\})\int_{(0,\infty)}x^{\alpha}G(t,x)dxdt
=\int_0^{\xi^{-1}(T)}\int_{(0,\infty)}x^{\alpha}h(\tau,x)dxd\tau.
\end{align*} 
The result  then follows from Proposition \ref{S5P2}.
\end{proof}

\begin{remark}
One could try  to directly solve the  system (\ref{PR}), (\ref{PR19}), written in $(g, n)$ variables. First, to obtain a sequence of solutions $(g_k, n_k)$ of an approximated system where the factor $x^{-1/2}$ is modified by truncation and regularization, and then pass to the limit. However, the limit obtained in that way, say $(g, n)$ is not a solution of (\ref{PR}), (\ref{PR19}). The reason is that all the solutions $g_k$ of the approximated system will be functions with a bounded moment of order $-1/2$. Then, the right hand side of the equation (\ref{S1E9}) is equal to $M _{ 1/2 }(g_k)$ and by passage to the limit the equation for $n$ will be $n'(t)=-n(t)M _{ 1/2 }(g(t))$, and  the total mass will not be conserved. 
\end{remark}

\section{Proofs of Theorems \ref{THn01}, \ref{MU1} and \ref{EQUIV}.}
\label{SectionK}
\setcounter{equation}{0}
\setcounter{theorem}{0}

We first prove Theorem \ref{THn01}. 

\begin{proof}[\upshape\bfseries{Proof of Theorem \ref{THn01} }]
We already know by Proposition \ref{origin G} and Lemma \ref{S6LG} that $n$ is right continuous and a.e. differentiable on $[0,\infty)$.
Then, by construction
\begin{align*}
G(t)=H(\tau)=h(\tau)-\left(\int_0^{\tau}M_{1/2}(h(\sigma))d\sigma\right)\delta_0,
\end{align*}
where $\tau\in[0,\tau^*)$ and $t\in[0,\infty)$ are related by (\ref{CHANGE}).
Hence
\begin{align}
\label{ZE2}
n(t)=m(\tau)-\int_0^{\tau}M_{1/2}(h(\sigma))d\sigma=m(\tau)-\int_0^t n(s)M_{1/2}(g(s))ds.
\end{align}
Since $n(0)=m(0)$, it then follows from Proposition \ref{Stieltjes1} that for all $t>0$:
\begin{align}
\label{ZEaa}
n(t)-n(0)+\int_0^t n(s)M_{1/2}(g(s))ds=\lambda((0,\tau]),
\end{align}
and using (\ref{CHANGE})
\begin{align}
\label{ZEbb}
\lambda((0,\tau])=\lim_{\varepsilon\to 0}\int_0^t n(s)\mathscr{Q}_3^{(2)}(\varphi_{\varepsilon},g(s))ds.
\end{align}
If we denote $\mu=\xi_{\#}\lambda$ (c.f. \cite{AMB}, Ch. 5), i.e., the push-forward  of $\lambda$ through the function $\xi:[0,\tau^*)\to[0,\infty)$ in (\ref{CHANGE}), then from the definition of $\mu$ we obtain
\begin{equation}
\label{Mul}
\mu((0,t])=\lambda((0,\tau])\qquad\forall t>0.
\end{equation}
Then (\ref{ZE02}) and  (\ref{ZE01}) follows from (\ref{ZEaa}), (\ref{ZEbb}) and (\ref{Mul}). Moreover, (\ref{ZE00}) follows from (\ref{Stieltjes0}) in Proposition \ref{Stieltjes1}. 
\end{proof}

The following properties of $n(t)$ follows by   the same arguments used in the proofs of properties (\ref{Stieltjes7}) and (\ref{Stieltjes2}) of Proposition \ref{Stieltjes1} 

\begin{proposition} 
Let $G$, $g$, and $n(t)$ be as in Theorem \ref{THn01}.  Then, for all $\varphi_{\varepsilon}$ as in Remark \ref{TEST}, the following limit exists 
in $\mathscr{D}'(0,\infty)$:
\begin{flalign}
\label{ZE04}
&&&\lim_{\varepsilon\to 0} n\mathscr{Q}_3^{(2)}(\varphi_{\varepsilon},g)=T(G),&&\\
\text{and}&&& 
\label{ZE05}
n'+nM_{1/2}(g)=T(G)\quad\text{in}\quad\mathscr{D}'(0,\infty).&&
\end{flalign}
\end{proposition}

\begin{proof}
Consider, for all  $\varphi_{\varepsilon}$ as in Remark \ref{TEST}, the absolutely continuous functions 
\begin{align}
\eta_{\varepsilon}(t)=\int_{[0,\infty)}\varphi_{\varepsilon}(x)G(t,x)dx.
\end{align}
Then equation (\ref{S1E16}) becomes $
\eta_{\varepsilon}'=n\mathscr{Q}_3(\varphi_{\varepsilon},g).$
Using integration by parts,
\begin{align*}
-\int_0^{\infty}\phi'(t)\eta_{\varepsilon}(t)dt=\int_0^{\infty}\phi(t)n(t)\mathscr{Q}_3(\varphi_{\varepsilon},g(t))dt
\quad\forall \phi\in C^{\infty}_c(0,\infty).
\end{align*}
Taking the limit $\varepsilon\to 0$ we deduce, using Lemma \ref{convergence lemma}, that
\begin{align*}
-\int_0^{\infty}\phi'(t)n(t)dt=&\lim_{\varepsilon\to 0}\int_0^{\infty}\phi(t)n(t)\mathscr{Q}_3^{(2)}(\varphi_{\varepsilon},g(t))dt\\
&-\int_0^{\infty}\phi(t)n(t)M_{1/2}(g(t))dt,
\end{align*}
and then (\ref{ZE04}), (\ref{ZE05}) follows. 
\end{proof}

\begin{remark}
If we take distributional derivatives in both sides of  (\ref{ZE02}) we obtain:
$$
n'+n M _{ 1/2 }(g)=\mu\quad\text{in}\quad\mathscr{D}'(0,\infty),$$
and by (\ref{ZE05}), $\mu =T(G)$.
\end{remark}

\begin{proof}[\upshape\bfseries{Proof of Theorem  \ref{MU1}}]
The statement of the Theorem follows from (\ref{Stieltjes0}) in Proposition \ref{Stieltjes1} and (\ref{Mul}).
\end{proof}

\begin{proof}[\upshape\bfseries{Proof of Theorem \ref{EQUIV}}]
Proof of part (i).
By Theorem \ref{THn01}, $n$ is given by (\ref{ZE02}) and (\ref{ZE01}). On the other hand, since $G$ satisfies (\ref{S1E16}), and for all
$\varphi\in C^1_b([0,\infty))$ such that $\varphi(0)=0$:
\begin{align}
\label{G=g}
\int_{[0,\infty)}\varphi(x)G(t,x)dx=\int_{[0,\infty)}\varphi(x)g(t,x)dx,
\end{align}
then $g$ satisfies (\ref{2E765}). In order to prove part (ii) we first show the existence of the limit in (\ref{ZE01}). To this end we write $\varphi_\varepsilon  = (1-\psi _\varepsilon )$, where  $\psi _\varepsilon$ is as in Remark \ref{TEST}. Then $\varphi _\varepsilon (0)=0$, and by (\ref{2E765}) and (\ref{S1EB1}),
using that $\mathscr{Q}_3(1-\psi_{\varepsilon},g)=\mathscr{Q}_3(1,g)-\mathscr{Q}_3(\psi_{\varepsilon},g)$, and $\mathscr{Q}_3(1,g)=0$, we deduce
\begin{align}
\label{Bht}
\int _0^tn(s)\widetilde{\mathscr Q}_3(\psi _\varepsilon,g(s))ds&=\int  _{ (0, \infty) }\varphi _\varepsilon (x)\left(g(0, x)-g(t, x)\right) dx\nonumber \\
&+\int _0^tn(s)M _{ 1/2 }(g(s))ds.
\end{align}
The existence of the limit in (\ref{ZE01}) follows and, if we pass to the limit,
\begin{align}
\lim _{ \varepsilon \to 0 }\int _0^tn(s) \mathscr Q_3^{(2)}(\psi _\varepsilon,g(s))ds&=\int  _{ (0, \infty) }(g(0, x)-g(t, x))dx\nonumber\\
&+\int _0^tn(s)M _{ 1/2 }(g(s))ds. \label{Bhtw}
\end{align}
We now check that, if $n$ satisfies the equation (\ref{ene1}) then  $G$ satisfies equation (\ref{S1E16}) for a.e. $t>0$ and for every $\varphi \in C_b^{1}([0, \infty))$. 
If $\varphi (0)=0$ this follows from (\ref{2E765}) and (\ref{G=g}).

For $\varphi (0)\not= 0$ we may assume without loss of generality that $\varphi (0)=1$, and write $\varphi =(\varphi -\psi _\varepsilon )+\psi _\varepsilon $, where 
$\psi _\varepsilon$ is as in Remark \ref{TEST}. 
Since  $(\varphi -\psi _\varepsilon )(0)=0$, using (\ref{2E765}) and  (\ref{S1ED3S})
\begin{align}
\label{S5P10E1'}
\int_{[0,\infty)}(\varphi-\psi_{\varepsilon})(x)g(t,x)dx&=\int_{[0,\infty)}(\varphi-\psi_{\varepsilon})(x)g(0,x)dx\nonumber\\
&+\int_0^t n(s)\widetilde{\mathscr{Q}}_3((\varphi-\psi_{\varepsilon}),g(s))ds.
\end{align}
In order to pas to the limit as $\varepsilon\to 0$, we first use
$\widetilde {\mathscr{Q}}_3((\varphi -\psi _\varepsilon ),g)=\widetilde {\mathscr{Q}}_3(\varphi,g)-\widetilde{\mathscr{Q}}_3(\psi_{\varepsilon},g).$
Then, since for all $t\geq 0$
\begin{align}
\label{gzero}
\lim_{\varepsilon\to 0}\int_{[0,\infty)}\psi_{\varepsilon}(x)g(t,x)dx=0,
\end{align}
and $n$ satisfies (\ref{ene1}), we deduce from (\ref{S5P10E1'}) and Lemma \ref{convergence lemma}:
\begin{align*}
\int_{[0,\infty)}\varphi(x)g(t,x)dx=&\int_{[0,\infty)}\varphi(x)g(0,x)dx+\int_0^t n(s)\widetilde{\mathscr{Q}}_3(\varphi,g(s))ds\\
&+n(0)-n(t)-\int_0^t n(s)M_{1/2}(g(s))ds.
\end{align*}
Since $\widetilde {\mathscr{Q}}_3(\varphi,G)-M_{1/2}(g)=\mathscr Q _3(\varphi,G)$,
it follows that $G$ satisfies 
\begin{align*}
\int_{[0,\infty)}\varphi(x)G(t,x)dx=\int_{[0,\infty)}\varphi(x)G(0,x)dx+\int_0^t n(s)\mathscr{Q}_3(\varphi,g(s))ds,
\end{align*}
thus (\ref{S1E16}) holds for a.e. $t>0$.

In order to check that $G$ satisfies (\ref{S1ED2S}) we first use (\ref{S1E16})  with $\varphi =1\in C_b^1([0, \infty))$. For that choice of $\varphi $ we have $\Lambda(\varphi)=\mathcal L_0(\varphi )\equiv 0$
and then:
$$
\int  _{ [0, \infty) }G(t, x)dx=\int  _{ [0, \infty) }G_0(x)dx.
$$
Because:
$$
\int  _{ [0, \infty) }x\,G(t, x)dx=\int  _{ [0, \infty) }x\,g(t, x)dx,
$$
$G$ satisfies (\ref{S1ED2S}) since by hypothesis so does $g$.
\end{proof}

\begin{remark} If $G$ is a weak radial solution of (\ref{PA}), (\ref{PB}), we know by  Theorem \ref{EQUIV} that  $g$ satisfies (\ref{2E765}). It is straightforward to check that it also satisfies,
\begin{equation*}
\frac{d}{dt}\int_{(0,\infty)}\varphi(x)g(t,x)dx=n(t)\widetilde{\mathscr{Q}} _3(\varphi,g(t))-\varphi (0)\frac{d}{dt}\mu((0,t]), \label{LE02'}
\end{equation*}
where $\mu$ is as in Theorem \ref{THn01}, and $\widetilde{\mathscr{Q}}_3$ is defined in (\ref{S1EB2})--(\ref{S1E21R}).
\end{remark}

\begin{proof}[\upshape\bfseries{Proof of Corollary \ref{S1C31}}]  If we prove that $n$ satisfies  (\ref{ene1}), the conclusion of the Corollary will follow from part (ii) of Theorem \ref{EQUIV}. 
By the hypothesis and part (ii) of Theorem \ref{EQUIV}, the limit in (\ref{ZE01}) exists, and (\ref{Bhtw}) holds, that we write:
\begin{align*}
\lim _{ \varepsilon \to 0 }\int _0^tn(s) \mathscr Q_3^{(2)}(\psi _\varepsilon,g(s))ds-\int _0^tn(s)M _{1/2}(g(s))ds=\\
=\int  _{ [0, \infty) }(G(0, x)-G(t, x))dx+n(t)-n(0).
\end{align*}
Using  the conservation of mass (\ref{MBC}) it follows that $n$ satisfies equation (\ref{ene1}).
\end{proof}

\begin{proposition}
\label{LM-1/2}
Let $G\in\mathscr{M}_+([0,\infty))$. If $G$ has no atoms on $(0,\infty)$ and $\int_{(0,\infty)}\frac{G(x)}{\sqrt{x}}dx<\infty$, 
then, for all $\varphi _\varepsilon $ as in Remarrk \ref{TEST},
\begin{equation*}
\mathscr{T}(G)=\lim_{\varepsilon\to 0}  \mathscr Q_3^{(2)}(\varphi_{\varepsilon},G)=0.
\end{equation*}
\end{proposition}

\begin{proof}
By definition 
$$
\mathscr{T}(G)=\lim_{\varepsilon\to 0}\iint_{(0,\infty)^2}\frac{\Lambda(\varphi_{\varepsilon})(x,y)}{\sqrt{xy}}G(x)G(y)dxdy,
$$
Since $\Lambda(\varphi_{\varepsilon})\leq 1$ for all $\varepsilon>0$ and 
\begin{align*}
\lim _{ \varepsilon \to 0 }\Lambda(\varphi_{\varepsilon})(x,y)=\mathds{1}_{\{x=y>0\}}(x,y)\qquad\forall (x,y)\in(0,\infty)^2,
\end{align*}
and $\int_{(0,\infty)}\frac{G(x)}{\sqrt{x}}dx<\infty$, then by dominated convergence
$$
\mathscr{T}(G)=\iint_{\{x=y>0\}}\frac{G(x)G(y)}{\sqrt{xy}}dxdy.
$$
Since $G$ has no atoms on $(0,\infty)$, i.e., $G(\{x\})=0$ for all $x>0$, by Fubini's theorem
\begin{align*}
\iint_{\{x=y>0\}}\frac{G(x)G(y)}{\sqrt{xy}}dxdy
&=\int_{(0,\infty)}\frac{G(x)}{x}G(\{x\})dx=0.
\end{align*}
\end{proof}

\begin{remark}
\label{HH}
From Proposition \ref{LM-1/2}, if  $M_{-1/2}(g)<\infty$ and $g$ has no atoms, then $\mu((0,t])=0$ for all $t>0$. 
If  $g\in L^1(0, \infty)$ and $x=0$  is a Lebesgue point of $g$ then $\mathscr{T}(g)=0$
(cf. \cite{Nouri1}) and again $\mu((0,t])=0$ for all $t>0$. 
If $g(x)=x^{-1/2}$, then $\mathscr{T}(g)=\pi^2/6$,
(cf. \cite{Lu3}), and a similar result holds if 
$\lim _{ x\to 0 }\sqrt x g(x)=C>0$ (cf. \cite{Spohn}). In that case,  $\mu((0,t])=\pi ^2/6\int _0^t n(s)ds$.
\end{remark}

\section{Proof of Theorem \ref{S1T5}}
\label{SectionD}
\setcounter{equation}{0}
\setcounter{theorem}{0}

\begin{proof}
By (\ref{CT98})  and (\ref{CT99}), we deduce that for all $t>t_0>0$:
\begin{align}
&\int _{ t_0 }^t G(s,\{0\})ds\leq \big(M_{\alpha}(G(t_0))-M_{\alpha}(G(t))\big)C(N,E,\alpha) \nonumber \\
\label{PRO116}
&C(N,E,\alpha)=\left[\left(\frac{\alpha-1}{\alpha+1}\right)E^{(2\alpha+1)/2}N^{(1-2\alpha)/2}-C_{1}N^{3-\alpha}E^{\alpha-1}\right]^{-1},
\end{align}
where $C_{1}=2^{\alpha}-2$ for $\alpha\in(1,2]$ and $C_1=\alpha(\alpha-1)$ for $\alpha\in[2,3]$.
Since by part (i), $0\leq M_{\alpha}(G(t_0))-M{\alpha}(G(t))\leq M_{\alpha}(G(t_0))$ for every $t>t_0$, we immediately deduce (\ref{S1ET5B}).

We prove now (\ref{S1ET5C}). 
Since, as we have seen in (\ref{S6E80}), the function $n(t)J(t)$ is monotone nondecreasing, from where, for all $t>0$ and $s\in (0, t)$:
\begin{equation*}
n(t)\ge e^{-\int _s^tM _{ 1/2 }(g(r))dr}n(s).
\end{equation*}
As we have $M _{ 1/2 }(g(r))\le \sqrt{NE}$ for all $r\ge 0$,
\begin{equation}
\label{S80E20}
n(t)\ge e^{-\sqrt {NE}(t-s)}n(s).
\end{equation}
By (\ref{S1ET5B}) we already have a sequence of times $\theta_k$ such that $\theta_k \to \infty $ and  $n(\theta_k)\to 0$ as $k\to \infty$.
Suppose that there exists, for some $\rho >0$, an increasing sequence of times $(s_k) _{ k\in \NN }$ such that $s_k\to\infty$ as $k\to\infty$ and :
\begin{align*}
\forall k,\;n(s_k)\ge \rho\quad\hbox{and}\quad s _{ k+1 }-s_k>\frac {\log 2} {\sqrt {NE}}.
\end{align*} 
Then, if we denote $t_k=s_k+\frac {\log 2} {\sqrt {NE}}$,
we deduce from (\ref{S80E20}) that for all $t\in (s_k, t_k)$:
\bean
n(t)\ge e^{-\sqrt {NE}(t-s_k)}n(s_k)\ge e^{-\sqrt {NE}(t_k-s_k)}\rho =\frac {\rho } {2}.
\eean
This would imply
\begin{equation*}
\int _0^\infty n(t)dt\ge \sum_{ k=0 } ^\infty \int  _{ s_k }^{t_k}n(t)dt=\infty,
\end{equation*}
and this contradiction proves (\ref{S1ET5C}).
\end{proof}

\section{Appendix}
\label{Appendix}
\setcounter{equation}{0}
\setcounter{theorem}{0}
We have gathered in this Section several results that are important and useful, but not directly related to the main results. For the sake of clarity, we present them in two different Sub Sections. In the first one, we find results that are used all along the manuscript, perhaps several times. In the second, we present results that are needed in Section \ref{model}. 

\subsection{A1}

\begin{lemma}[Convex-positivity]
\label{convex-positivity}
Let $\varphi\in C([0,\infty))$. If $\varphi$ is convex then $\Lambda(\varphi)(x,y)\geq 0$ for all $(x,y)\in[0,\infty)^2$ and $\mathcal{L}_0(\varphi)(x)\geq 0$ for all $x\in[0,\infty)$. If $\varphi$ is nonnegative and nonincreasing, then $\mathcal{L}(\varphi)(x)\leq 0$ for all $x\in[0,\infty)$.
\end{lemma}

\begin{proof}
Since $\Lambda(\varphi)(x,y)$ is symmetric we may reduce the proof to the case $0\leq y\leq x$. Putting
$x=\frac{x+y}{2}+\frac{x-y}{2},$
then by the very definition of convexity
$$\varphi(x)\leq\frac{\varphi(x+y)}{2}+\frac{\varphi(x-y)}{2},$$
therefore $\Lambda(\varphi)(x,y)\geq 0$.

The positivity of $\mathcal{L}_0(\varphi)$ is equivalent to prove 
\begin{align}
\label{positivity weak L}
\frac{1}{x}\int_0^x\varphi(y)d y\leq\frac{\varphi(0)+\varphi(x)}{2}\qquad\forall x\in[0,\infty).
\end{align}
Since for any $0\leq y\leq x$ we may trivially write $y=\left(1-\frac{y}{x}\right)\, 0+\frac{y}{x}\, x$, then by convexity 
$\varphi(y)\leq \left(1-\frac{y}{x}\right)\varphi(0)+\frac{y}{x}\varphi(x)$, which implies
\eqref{positivity weak L}.\\
If $\varphi$ is nonnegative and nonincreasing, then  $\mathcal{L}(\varphi)(x)\leq -x\,\varphi(x)\leq 0$ for all\ $x\in[0,\infty).$
\end{proof}

\begin{remark}
\label{concave-negativity}
By linearity and Lemma \ref{convex-positivity}, it follows that for all $\varphi\in C([0,\infty))$ concave,
$\Lambda(\varphi)(x,y)\leq 0$ for all $(x,y)\in[0,\infty)^2$ and $\mathcal{L}_0(\varphi)(x)\leq 0$ for all $x\in[0,\infty)$.
\end{remark}

\begin{lemma}
\label{lemma regularity}
Consider the operators $\Lambda(\cdot)$, $\mathcal{L}_0(\cdot)$ and $\mathcal{L}(\cdot)$ given in (\ref{S1E154}), (\ref{S1E155}) and (\ref{S1E21R}) respectively. Then
\begin{enumerate}[(i)]
\item
If $\varphi\in\emph{Lip}([0,\infty))$ with Lipschitz constant $L$, then  
\begin{align}
\label{lemma regularity 1}
\frac{|\Lambda(\varphi)(x,y)|}{\sqrt{xy}}\leq 2L\qquad\forall (x,y)\in[0,\infty)^2.
\end{align}
\item
If $\varphi\in C^1([0,\infty))$, then the map
$(x,y)\mapsto\frac{\Lambda(\varphi)(x,y)}{\sqrt{xy}}$ belongs to $C([0,\infty)^2)$ and 
\begin{align}
\label{lemma regularity 2}
\frac{\Lambda(\varphi)(x,y)}{\sqrt{xy}}=0\qquad\forall (x,y)\in\partial[0,\infty)^2.
\end{align}
\item 
If $\varphi\in C([0,\infty))$ then the maps $x\mapsto\frac{\mathcal{L}_0(\varphi)(x)}{\sqrt{x}}$ and  
 $x\mapsto\frac{\mathcal{L}(\varphi)(x)}{\sqrt{x}}$ belong to $C([0,\infty))$ and 
$\frac{\mathcal{L}_0(\varphi)(x)}{\sqrt{x}}= \frac{\mathcal{L}(\varphi)(x)}{\sqrt{x}}=0$ at $x=0$. If in addition 
$\varphi$ is bounded, then
\begin{align}
\frac{|\mathcal{L}_0(\varphi)(x)|}{\sqrt{x}}\leq 4\|\varphi\|_{\infty}\sqrt{x}\qquad\forall x\in[0,\infty), \label{lemma regularity 4}\\
\frac{|\mathcal{L}(\varphi)(x)|}{\sqrt{x}}\leq 3\|\varphi\|_{\infty}\sqrt{x}\qquad\forall x\in[0,\infty). \label{lemma regularity 4B}
\end{align}
\end{enumerate}
\end{lemma}

\begin{proof}
(i) By the symmetry of $\Lambda(\varphi)$ we can assume that $0\leq y\leq x$, and directly from the Lipschitz continuity
\begin{align*}
|\Lambda(\varphi)(x,y)|&\leq |\varphi(x+y)-\varphi(x)|+|\varphi(x-y)-\varphi(x)|\leq 2L\,y,
\end{align*}
which implies \eqref{lemma regularity 1}.\\
(ii) The only possible problem for the continuity is on the boundary of $[0,\infty)^2$.
Again by the symmetry of $\Lambda(\varphi)$ we can assume $0\leq y\leq x$. Then by the mean value theorem
$\Lambda(\varphi)(x,y)=y\,(\varphi'(\xi_1)-\varphi'(\xi_2))$ for some $\xi_1\in(x,x+y)$ and $\xi_2\in(x-y,x)$. 
Hence
\begin{align*}
\frac{\Lambda(\varphi)(x,y)}{\sqrt{xy}}\leq \varphi'(\xi_1)-\varphi'(\xi_2),
\end{align*}
and the continuity of $\frac{\Lambda(\varphi)(x,y)}{\sqrt{xy}}$ on $[0,\infty)^2$ and \eqref{lemma regularity 2} follow from the continuity of $\varphi'$.\\
(iii) The continuity of $\frac{\mathcal{L}_0(\varphi)(x)}{\sqrt{x}}$  and $\frac{\mathcal{L}(\varphi)(x)}{\sqrt{x}}$ are clear for $x>0$.  Using that $\frac{1}{x}\int_0^x \varphi(y)d y\rightarrow \varphi(0)$ as $x\rightarrow 0$ by Lebesgue differentiation Theorem, it follows the continuity at $x=0$ and that $\frac{\mathcal{L}_0(\varphi)(x)}{\sqrt{x}}=\frac{\mathcal{L}(\varphi)(x)}{\sqrt{x}}=0$ for $x=0$. The bounds (\ref{lemma regularity 4}) and (\ref{lemma regularity 4B})   are straightforward  for $\varphi\in C_b([0,\infty))$.
\end{proof}

\begin{lemma}
\label{regularised operators converge uniformly}
Consider the operators $\Lambda(\cdot)$ and $\mathcal{L}_0(\cdot)$ given in (\ref{S1E154}) and (\ref{S1E155}), and a sequence 
$(\phi_n)_{n\in\NN}\subset C_c([0,\infty))$ as in Cutoff \ref{cut-off}.
\begin{enumerate}[(i)]
\item
If $\varphi\in C^1([0,\infty))$ then $\Lambda(\varphi)(x,y)\phi_n(x)\phi_n(y)\xrightarrow[n\rightarrow\infty]{}\frac{\Lambda(\varphi)(x,y)}{\sqrt{xy}}$ uniformly on the compact sets of $[0,\infty)^2$.
\item
If $\varphi\in C([0,\infty))$ then $\mathcal{L}(\varphi)(x)\phi_n(x)\xrightarrow[n\rightarrow\infty]{}\frac{\mathcal{L}(\varphi)(x)}{\sqrt{x}}$
uniformly on the compact sets of $[0,\infty)$.
\end{enumerate}
\end{lemma}

\begin{proof}
(i) The pointwise convergence on $[0,\infty)^2$ is trivial since $\phi_n(x)\rightarrow x^{-1/2}$ as $n\rightarrow\infty$. Then, let $\varepsilon>0$ and $R>0$. For $n\geq R$ there holds $\phi_n(x)=x^{-1/2}$ for all 
$x\in[1/n,R]$, so we only need to show the uniform convergence on the regions $(x,y)\in[0,R]\times[0,1/n]$ and $(x,y)\in[0,1/n]\times[0,R]$. 
By the symmetry of $\Lambda(\varphi)$, we may study only one region. 

Using that $\frac{\Lambda(\varphi)(x,y)}{\sqrt{xy}}$ is continuous (hence uniformly continuous on compacts) and vanishes when  
$(x,y)\in\partial[0,\infty)^2$
(c.f. Lemma \ref{lemma regularity}), there holds for all $(x,y)\in[0,R]\times [0,1/n]$ that, for $n$ large enough,
\begin{align*}
\left|\frac{\Lambda(\varphi)(x,y)}{\sqrt{xy}}-\Lambda(\varphi)(x,y)\phi_n(x)\phi_n(y)\right|\leq\frac{|\Lambda(\varphi)(x,y)|}{\sqrt{xy}}\leq \varepsilon
\end{align*}

(ii) Let $\varepsilon>0$ and $R>0$. Since for $n\geq R$ there holds $\phi_n(x)=x^{-1/2}$ for all $x\in[1/n,R]$, we only need to prove the uniform convergence on the region $[0,1/n]$. Using that $\frac{\mathcal{L}(\varphi)(x)}{\sqrt{x}}$ is continuous (hence uniformly continuous on compacts) and vanishes when $x\rightarrow 0$ (cf. Lemma \ref{lemma regularity}), we have
\begin{align*}
\left|\frac{\mathcal{L}(\varphi)(x)}{\sqrt{x}}-\mathcal{L}(\varphi)(x)\phi_n(x)\right|\leq \frac{|\mathcal{L}(\varphi)(x)|}{\sqrt{x}}\leq\varepsilon
\qquad\forall x\in[0,1/n]
\end{align*}
for $n$ large enough.
\end{proof}

The following Lemma is about the approximation of a measure by functions, keeping the mass and the energy constants. It is taken from \cite{Lu1} with  minor modifications.
\begin{lemma}
\label{APROXDATA}
Let $h_0\in \mathscr{M}_+^{\alpha}([0,\infty))$ for some $\alpha\geq 1$.
Then there exists a sequence of functions $(f_n) _{n\in\NN}\subset C([0,\infty))\cap L^1\big(\RR_+,(1+x^{\alpha})dx\big)$ with $f_n>0$ such that
\begin{align}
\label{GROWTHA}
&\forall \varphi \in C([0,\infty)):\quad\sup_{x\geq 0}\frac{|\varphi(x)|}{1+x^{\alpha}}<\infty,\\
\label{APROXDATA1}
&\lim _{ n\to \infty  }\int_0^{\infty}\varphi(x)f_n(x)dx=\int_{[0,\infty)}\varphi(x)h_0(x)dx.
\end{align}
Moreover, if $M_1(h_0)>0$, then for all $n\in\NN$:
\begin{align}
\label{APROXDATA2}
M_0(f_n)=M_0(h_0)\qquad\text{and}\qquad M_1(f_n)=M_1(h_0).
\end{align}
\end{lemma}

\begin{proof}
For $a>0$ and $b>0$, let
\begin{align*}
J_{a,b}(x)=a e^{-bx^2},\qquad(x\geq 0)
\end{align*}
and let, for $n\in\NN$,
\begin{align*}
f_n(x)=e^n\int_0^{\frac{x}{1-e^{-n}}}J_{a,b}\left(e^n\big(x-y(1-e^{-n})\big)\right)h(y)dy,\qquad(x\geq0).
\end{align*}
Since $J_{a,b}$ is bounded and $M_0(h_0)<\infty$, then $f_n$ is well defined. The continuity and the strict positivity of $J_{a,b}$, together with $M_0(h)>0$, implies that $f_n$ is continuous and $f_n>0$.
Now for any $\varphi\in C([0,\infty))$ satisfying (\ref{GROWTHA}), using Fubini and the change of variables $z=e^n\big(x-y(1-e^{-n})\big)$:
\begin{align}
\label{APROXDATA3}
&\int_0^{\infty}\varphi(x)f_n(x)dx=\int_{[0,\infty)}I_n(\varphi)(y)h_0(y)dy,\\
\label{APROXDATA33}
&I_n(\varphi)(y)=\int_0^{\infty}\varphi\big(ze^{-n}+y(1-e^{-n})\big)J_{a,b}(z)dz.
\end{align}
Since by (\ref{GROWTHA}):
\begin{align*}
\big|\varphi\big(ze^{-n}+y(1-e^{-n})\big)\big|&\leq C\Big(1+\big(ze^{-n}+y(1-e^{-n})\big)^{\alpha}\Big)\\
&\leq C2^{\alpha}\big(1+y^{\alpha}+z^{\alpha}\big)\\
&\leq C2^{\alpha}\big(1+y^{\alpha}\big) \big(1+z^{\alpha}\big),
\end{align*}
and $h_0\in\mathscr{M}_+^{\alpha}([0,\infty))$, we deduce from (\ref{APROXDATA3})-(\ref{APROXDATA33}) that 
$f_n\in L^1\big(\RR_+,(1+x^{\alpha})\big)$. 
We also deduce using dominated convergence that
\begin{align}
&\lim _{ n\to \infty }I_n(\varphi)(y)=\varphi(y),\,\,\forall y\geq 0,\nonumber \\
\label{APROXDATA4}
&\lim _{ n\to \infty  }\int_{[0,\infty)}I_n(\varphi)(y)h_0(y)dy=\bigg(\int_0^{\infty} J_{a,b}(z)dz\bigg)\!\!\int_{[0,\infty)}\varphi(y)h_0(y)dy.
\end{align}
We now fix $a>0$ so that $\int_0^{\infty}J_{a,b}(z)dz=1$. Namely, $a=2\sqrt{b/\pi}$. Then (\ref{APROXDATA1}) follows from (\ref{APROXDATA3}) and (\ref{APROXDATA4}). 

If we chose $\varphi=1$ in (\ref{APROXDATA3})-(\ref{APROXDATA33}), the first part of (\ref{APROXDATA2}) follows.
If we chose now $\varphi(y)=y$ then
\begin{align*}
M_1(f_n)=M_1(h)+e^{-n}\left(\frac{M_0(h_0)}{\sqrt{b\pi}}-M_1(h_0)\right).
\end{align*}
We now fix $b=\pi^{-1}(M_0(h_0)/M_1(h_0))^2$ to obtain the second part of (\ref{APROXDATA2}).
\end{proof}

\begin{corollary}
\label{APD1}
Let $h_0\in\mathscr{M}^{\alpha} _+([0,\infty))$ for some $\alpha \ge 1$. Then, there exists a sequence of nonnegative functions 
$(h_{0,n})_{n\in\NN}\subset C_c([0,\infty))$ such that
\begin{align}
\label{APD56}
\limsup_{n\to\infty}M_{\alpha}(h_{0,n})\leq M_{\alpha}(h_0),
\end{align}
and for all $\varphi\in C_b([0,\infty))$
\begin{align}
\label{APD3}
&\lim _{ n\to \infty }\int_0^{\infty}\varphi(x)h_{0,n}(x)dx=\int_{[0,\infty)}\varphi(x)h_0(x)dx.
\end{align}

\end{corollary}

\begin{proof}
We consider the sequence $(f_n)$ given by Lemma \ref{APROXDATA} and a smooth cutoff $\zeta_n\in C([0,\infty))$ such that $0\leq\zeta_n\leq 1$, $\zeta_n(x)=1$ for $x\in[0,n]$ and $\zeta_n(x)=0$ for $x\geq n+1$. Then we define for all $n\in\NN$:
\begin{align}
\label{APD33}
h_{0,n}(x)=f_n(x)\zeta_n(x)
\end{align}
It then follows that $h_{0,n}$ is continuous, nonnegative and with compact support. 
The property (\ref{APD56}) follows directly from (\ref{APROXDATA1}) in Lemma \ref{APROXDATA} since $h_{0,n}\leq f_n$.
Now let $\varphi\in C_b([0,\infty))$. Since $f_n$ satisfies (\ref{APROXDATA1}), in order to prove (\ref{APD3}) it is sufficient to prove
\begin{align}
\label{APD4}
\lim _{ n\to \infty }\bigg|\int_0^{\infty}\varphi(x)h_{0,n}(x)dx-\int_0^{\infty}\varphi(x)f_n(x)(x)dx\bigg|=0,
\end{align}
and (\ref{APD4}) follows from
\begin{align*}
\lim_{n\to\infty}\int_n^{\infty}\varphi(x)f_n(x)dx\leq\lim_{n\to\infty} \frac{\|\varphi\|_{\infty}M_1(f_n)}{n}=0.
\end{align*}
\end{proof}

\begin{definition}
\label{definition operators strong}
Let $h$, $\phi_n$ and $\varphi$ be real-valued functions with domain $\RR_+$. Then, let
\begin{align}
\widetilde{\mathscr{Q}}_{3,n}(\varphi,h)=\mathscr{Q}_{3,n}^{(2)}(\varphi,h)-\widetilde{\mathscr{Q}}_{3,n}^{(1)}(\varphi,h), \label{Aq3tilden}
\end{align}
where
\begin{align}
\label{Q3n2}
&\mathscr{Q}_{3,n}^{(2)}(\varphi,h)=\int_0^{\infty}\!\!\!\int_0^{\infty} \Lambda(\varphi) (x, y)\phi_n(x)\phi_n(y)h(x)h(y)dxdy,\\
\label{Q3n1w}
&\widetilde{\mathscr{Q}}_{3,n}^{(1)}(\varphi,h)=\int_0^{\infty}\mathcal {L}(\varphi )(x)\phi_n(x)h(x)dx,
\end{align}
and let, for $x\in\RR_+$:
\begin{align}
\label{A1E32}
J_{3,n}(h)(x)&=K_n(h)(x)+L_n(h)(x)-h(x)A_n(h)(x),
\end{align} 
where
\begin{align}
K_n(h)(x)&=\int_0^x h(x-y)h(y)\phi_n(x-y)\phi_n(y)d y \nonumber\\
&+2\int_x^{\infty} h(y)h(y-x)\phi_n(y)\phi_n(y-x)d y,\label{A1E33}\\
L_n(h)(x)&=2\int_x^{\infty}h(y)\phi_n(y)d y,\label{A1E34}\\
A_n(h)(x)&=\phi_n(x)\Big(x+4\int_0^x h(y)\phi_n(y)d y\Big).\label{A1E35}
\end{align}
\end{definition}

\begin{lemma}
\label{convergence lemma}
Let $G\in\mathscr{M}_+([0,\infty))$, $\varphi_{\varepsilon}$ as in Remark \ref{TEST}, and $\phi_n$ as in Cutoff \ref{cut-off}. Then
\begin{align}
&G(\{0\})=\lim_{\varepsilon\to 0}\int_{[0,\infty)}\varphi_{\varepsilon}(x)G(x)dx,\label{convergence 1}\\
&\lim_{\varepsilon\to 0}\widetilde{\mathscr{Q}}_{3,n}^{(1)}(\varphi_{\varepsilon},G)=0\qquad\forall n\in\NN.\label{convergence 3}
\end{align}
If in addition $G$ has no singular part in $(0,\infty)$, then
\begin{align}
\lim_{\varepsilon\to 0}\mathscr{Q}_{3,n}^{(2)}(\varphi_{\varepsilon},G)=0\qquad\forall n\in\NN.\label{convergence 4}
\end{align}
Furthermore, if $G\in\mathscr{M}_+^{1/2}([0,\infty))$, then
\begin{align}
\label{limQ31a}
&\lim_{\varepsilon\to 0}\mathscr{Q}_3^{(1)}(\varphi_{\varepsilon},G)=M_{1/2}(G),\\
\label{limQ31b}
&\lim_{\varepsilon\to 0}\widetilde{\mathscr{Q}}_3^{(1)}(\varphi_{\varepsilon},G)=0,
\end{align}
where $\mathscr{Q}_3^{(1)}$ and $\widetilde{\mathscr{Q}}_3^{(1)}$ are defined in (\ref{S1E1Q31}) and (\ref{S1E20R}) respectively.
\end{lemma}

\begin{proof}
The proof only uses dominated convergence. Since $\varphi_{\varepsilon}\leq 1$ for all $\varepsilon>0$, and $M_0(G)<\infty$, and 
$\varphi_{\varepsilon}\to\mathds{1}_{\{0\}}$ as $\varepsilon\to 0$, then (\ref{convergence 1}) holds.
Then, since for all $x\in[0,\infty)$ it follows from dominated convergence that
\begin{align}
\label{LIN1}
\lim_{\varepsilon\to 0}\mathcal{L}_0(\varphi_{\varepsilon})(x)=x
\qquad\text{and}\qquad
\lim_{\varepsilon\to 0}\mathcal{L}(\varphi_{\varepsilon})(x)=0,
\end{align}
and $\phi_n$ is compactly supported, then (\ref{convergence 3}) follows. 
Also, since for all $(x,y)\in[0,\infty)^2$, $\Lambda(\varphi_{\varepsilon})(x,y)\leq 1$ for all $\varepsilon>0$, and 
\begin{align*}
\lim_{\varepsilon\to 0}\Lambda(\varphi_{\varepsilon})(x,y)=\mathds{1}_{\{x=y>0\}}(x,y),
\end{align*}
then
\begin{align*}
\lim_{\varepsilon\to 0}\mathscr{Q}_{3,n}^{(2)}(\varphi_{\varepsilon},G)=\iint_{\{x=y>0\}}\phi_n(x)\phi_n(y)G(x)G(y)d xd y,
\end{align*}
Using that $G$ has no singular part on $(0,\infty)$, (\ref{convergence 4}) follows.

Lastly, since 
\begin{align}
\widetilde{\mathscr{Q}}_3^{(1)}(\varphi_{\varepsilon},G)\leq \mathscr{Q}_3^{(1)}(\varphi_{\varepsilon},G)=\int_{(0,\infty)}\frac{\mathcal{L}_0(\varphi_{\varepsilon})(x)}{\sqrt{x}}G(x)dx,
\end{align}
and by (\ref{lemma regularity 4})
\begin{align*}
\int_{(0,\infty)}\frac{|\mathcal{L}_0(\varphi_{\varepsilon})(x)|}{\sqrt{x}}G(x)dx\leq 4M_{1/2}(G)\qquad\forall\varepsilon>0.
\end{align*}
then (\ref{limQ31a}) and (\ref{limQ31b}) follows from (\ref{LIN1}) and dominated convergence.
\end{proof}

\begin{lemma}
\label{well defined operators}
Consider $n\in\NN$, $\phi_n\in C_c([0,\infty))$ nonnegative and  $\rho\in L^1_{loc}(\RR_+)$ nonnegative. Then for every nonnegative functions 
$h$, $h_1$ and $h_2$ in $L^{\infty}(\RR_+)$, the functions $K_n(h)$, $L_n(h)$, $A_n(h)$ and $hA_n(h)$ are also nonnegative, belong to 
$L^{\infty}(\RR_+)\cap L^1_{\rho}(\RR_+)$, and there exists a positive constant $C(n,\rho)$ such that:
\begin{align}
&\|K_n(h_1)-K_n(h_2)\|_{L^{\infty}\cap L^1_{\rho}}\leq C(n,\rho)\|h_1\|_{\infty}\|h_1-h_2\|_{\infty}\label{SAE100}\\
&\|L_n(h)\|_{L^{\infty}\cap L^1_{\rho}}\leq C(n,\rho)\|h\|_{\infty}\label{bound L}\\
&\|A_n(h)\|_{L^{\infty}\cap L^1_{\rho}}\leq C(n,\rho)\big(1+\|h\|_{\infty}\big)\label{SaE120}\\
&\|A_n(h_1)-A_n(h_2)\|_{L^{\infty}\cap L^1_{\rho}}\leq C(n,\rho)\|h_1-h_2\|_{\infty}.\label{SaE121}
\end{align}
\begin{flalign}
\text{Moreover}&&J_{3,n}(h)\in L^{\infty}(\RR_+)\cap L^1_{\rho}(\RR_+).&&
\end{flalign}
\end{lemma}

\begin{proof}
The positivity of the operators is clear from their definitions. Notice that since $\phi_n$ is bounded and compactly supported on $\RR_+$ and $\rho\in L^1_{loc}(\RR_+)$, there exist two positive constants $C(n)$ and $C(n,\rho)$ such that
\begin{align*}
&\sup_{x\geq 0}\int_0^{\infty} \phi_n(|x-y|)\phi_n(y)d y\leq C(n),\\
&\int_0^{\infty}\!\!\!\int_0^{\infty}\rho(x)\phi_n(|x-y|)\phi_n(y)d yd x\leq C(n,\rho).
\end{align*}

\noindent1. Estimates for $K_n$. For all $x\geq 0$:
\begin{align*}
K_n(h)(x)&\leq 3\|h\|_{\infty}^2\int_0^{\infty} \phi_n(|x-y|)\phi_n(y)d y\leq 3\|h\|_{\infty}^2C(n),
\end{align*}
and
\begin{align*}
\|K_n(h)\|_{L^1_{\rho}}
&\leq 3\|h\|_{\infty}^2\int_0^{\infty}\!\!\! \int_0^{\infty}  \rho(x)\phi_n(|x-y|)\phi_n(y)d yd x\leq 3\|h\|_{\infty}^2 C(n,\rho).
\end{align*}
Then for all $x\geq 0$:
\begin{align}
\label{KNs}
&\big|K_n(h_1)(x)-K_n(h_2)(x)\big|\\
&\leq3\int_0^{\infty}\phi_n(|x-y|)\phi_n(y)\big|h_1(|x-y|)h_1(y)-h_2(|x-y|)h_2(y)\big|d y.\nonumber
\end{align}
Without loss of generality we assume that $\|h_1\|_{\infty}\geq \|h_2\|_{\infty}$. Using 
\begin{align*}
&\big|h_1(|x-y|)h_1(y)-h_2(|x-y|)h_2(y)\big|\leq 2\|h_1\|_{\infty}\|h_1-h_2\|_{\infty}
\end{align*}
in (\ref{KNs}) then (\ref{SAE100}) follows.\\
2. Estimates for $L_n$. Since $\phi_n$ is bounded and compactly supported and $\rho\in L^1_{loc}(\RR_+)$, there exist two positive constants $C(n)$ and 
$C(n,\rho)$ such that
\begin{align*}
\int_0^{\infty}\phi_n(x)d x\leq C(n)\quad\text{and}\quad
\int_0^{\infty}\rho(x)\int_x^{\infty}\phi_n(y) dy dx\leq C(n,\rho)
\end{align*}
and (\ref{bound L}) follows.\\
3. Estimates for $A_n$.
The estimate  (\ref{SaE120}) follows from
\begin{align*}
\|A_n(h)\|_{\infty}&\leq \|x\,\phi_n(x)\|_{\infty}+4\|\phi_n\|_{\infty}^2\|h\|_{\infty}|\supp(\phi_n)|\leq C(n)(1+\|h\|_{\infty}),
\end{align*}
and
\begin{align*}
\|A_n(h)\|_{L^1_{\rho}}&\leq \int_0^{\infty}\rho(x)\,x\,\phi_n(x)d x+4\,\|h\|_{\infty}\int_0^{\infty}\rho(x)\,\phi_n(x)\int_0^x \phi_n(y)d yd x\\
&\leq C(n,\rho)(1+\|h\|_{\infty}).
\end{align*}
For all $x\geq 0$,
\begin{align*}
|A_n(h_1)(x)-A_n(h_2)(x)|&\leq 4\|h_1-h_2\|_{\infty}\phi_n(x)\int_0^x\phi_n(y)d y\\
&\leq C(n) \|h_1-h_2\|_{\infty}.
\end{align*}
We also have, 
\begin{align*}
\|A_n(h_1)-A_n(h_2)\|_{L^1_{\rho}}&\leq 4\|h_1-h_2\|_{\infty}\int_0^{\infty}\rho(x)\phi_n(x)\int_0^x\phi_n(y)d yd x\\
&\leq C(n,\rho)\,\|h_1-h_2\|_{\infty},
\end{align*}
and then, (\ref{SaE121}) follows.\\
4. Since $h\in L^{\infty}(\RR_+)$ and $A_n(h)\in L^{\infty}(\RR_+)\cap L^1_{\rho}(\RR_+)$, then $hA_n(h)\in L^{\infty}(\RR_+)\cap L^1_{\rho}(\RR_+)$.\\
5. It also follows from points 1 to 4 that $J_{3,n}(h)$ has the desired regularity.
\end{proof}

\subsection{A2}

\begin{lemma}
\label{representation of Deltavarphi}
Let $\varphi\in C ^{1.1}([0,\infty))$. Then, for all $(x_1,x_2,x_3)\in[0,\infty)^3$ such that $x_1+x_2\geq x_3$:
\begin{align*}
&\Delta\varphi(x_1,x_2,x_3)
=(x_1-x_3)(x_2-x_3)\times  \nonumber \\
&\qquad \times \int_0^1\int_0^1\varphi''\big(x_3+t(x_1-x_3)+s(x_2-x_3)\big)dsdt.
\end{align*}
Moreover, if $\varphi\in C^{1.1}_b([0,\infty))$, then for all $(x_1,x_2,x_3)\in [0,\infty)^3$
\begin{align}
|\Delta\varphi(x_1,x_2,x_3)| \leq\min\left\{A, B, C, D\right\}. \label{S2E2}
\end{align}
\begin{flalign}
\text{where} &&A&=4\|\varphi\|_{\infty},\quad B=2\|\varphi'\|_{\infty}|x_1-x_3|,\quad C=2\|\varphi'\|_{\infty}|x_2-x_3|,\nonumber\\
&&D&=\|\varphi''\|_{\infty}|x_1-x_3||x_2-x_3|.\nonumber 
\end{flalign}
\end{lemma}

\begin{proof}
Let $(x_1,x_2,x_3)\in[0,\infty)^3$ be such that $x_1+x_2\geq x_3$. By the fundamental Theorem of calculus
\begin{align*}
\Delta &\varphi(x_1,x_2,x_3)=\big[\varphi(x_4)-\varphi(x_2)\big]-\big[\varphi(x_1)-\varphi(x_3)\big]\\
&  =\int_0^1\frac{d}{dt}\varphi\big(x_2+t(x_1-x_3)\big)dt-\int_0^1\frac{d}{dt}\varphi\big(x_3+t(x_1-x_3)\big)dt\\
&=(x_1-x_3)\int_0^1\big[\varphi'\big(x_2+t(x_1-x_3)\big)-\varphi'\big(x_3+t(x_1-x_3)\big)\big]dt\\
&=(x_1-x_3)\int_0^1\int_0^1\frac{d}{ds}\varphi'\big(x_3+t(x_1-x_3)+s(x_2-x_3)\big)dsdt\\
&=(x_1-x_3)(x_2-x_3)\int_0^1\int_0^1\varphi''\big(x_3+t(x_1-x_3)+s(x_2-x_3)\big)dsdt.
\end{align*}
Assume now that $\varphi\in C^{1.1}_b([0,\infty))$. Using the first, the third, and the fifth line above, estimate (\ref{S2E2}) follows. 
\end{proof}

We now consider the function $w$ given in (\ref{S1E6'}) and define
\begin{align}
\label{S2E3}
W(x_1,x_2,x_3)=\left\{
\begin{array}{ll}
\frac{w(x_1,x_2,x_3)}{\sqrt{x_1x_2x_3}}&\!\!\text{if}\quad(x_1,x_2,x_3)\in (0,\infty)^3\\
\frac{1}{\sqrt{x_1x_2}}&\!\!\!\!\!\!\text{if}\quad x_3=0,\quad (x_1,x_2)\in(0,\infty)^2\\
\frac{1}{\sqrt{x_ix_3}}&\!\!\!\!\!\!\text{if}\; x_j=0,\, x_i>x_3>0;\,\{i, j\}=\{1,2\} \\
0&\!\!\!\text{otherwise}.
\end{array}
\right.
\end{align}

We then have:
\begin{lemma}
\label{S2L1}
Consider the function $\Phi_{\varphi}=W\Delta\varphi$, where $\Delta\varphi$  and $W$  are defined in (\ref{S2E1}) and (\ref{S2E3}) respectively.
\begin{enumerate}[(i)]
\item
If $\varphi\in C^{1.1}([0,\infty))$ then $\Phi_{\varphi}\in C([0,\infty)^3)$.
\item
If $\varphi\in C^{1.1}_b([0,\infty))$ then $\Phi_{\varphi}\in C_0([0,\infty)^3)$. In particular $\Phi_{\varphi}$ is uniformly continuous on $[0,\infty)^3$.
\end{enumerate}
\end{lemma}

\begin{proof}
\textbf{Proof of (i).} By definition $\Phi _\varphi\in C ((0, \infty)^3)$. Therefore it only remains to study the behaviour of $\Phi _\varphi$ in a neighborhood of the boundary $\partial [0, \infty)^3$ of $[0, \infty)^3$. First we show that $\Phi_{\varphi}$ is continuous on $\partial [0, \infty)^3$. \\
Thanks to the symmetry of $\Phi_{\varphi}$ in the $x_1$, $x_2$ variables, we just need to prove:\\
(i)for all $(x_1,x_2)\in (0,\infty)^2$, 
\begin{align}
\label{boundary1}
\Phi_{\varphi}(x_1,x_2,0)=\frac{\Delta\varphi(x_1,x_2,0)}{\sqrt{x_1x_2}}\longrightarrow 0
\end{align}
whenever $x_1\rightarrow 0$ or $x_2\rightarrow 0$ or $(x_1,x_2)\rightarrow (0,0)$, and\\
(ii) for all $x_1>x_3>0$,
\begin{align}
\label{boundary2}
\Phi_{\varphi}(x_1,0,x_3)=\frac{\Delta\varphi(x_1,0,x_3)}{\sqrt{x_1x_3}}\longrightarrow 0
\end{align}
whenever $x_1\rightarrow x_3$ or $x_3\rightarrow 0$ or $(x_1,x_3)\rightarrow (0,0)$.

By (\ref{S2E2}) $|\Delta\varphi(x_1,x_2,0)|\leq \|\varphi''\|_{\infty}x_1x_2$ for all $(x_1,x_2)\in (0,\infty)^2$, which implies (\ref{boundary1}). Also $|\Delta\varphi(x_1,0,x_3)|\leq \|\varphi''\|_{\infty}x_3(x_1-x_3)$ for all $x_1>x_3>0$. Hence
\begin{align*}
\frac{|\Delta\varphi(x_1,0,x_3)|}{\sqrt{x_1x_3}}\leq  \|\varphi''\|_{\infty}\sqrt{\frac{x_3}{x_1}}(x_1-x_3)\leq \|\varphi''\|_{\infty}(x_1-x_3),
\end{align*}
which implies (\ref{boundary2}).

Then we prove that for any $x\in\partial[0,\infty)^3$ and for any 
$(x_n)_{n\in\NN}\subset (0,\infty)^3$ such that $x_n\rightarrow x$, then $\Phi_{\varphi}(x_n)\rightarrow\Phi_{\varphi}(x)$ as $n\to\infty$. Let us denote
$$\Omega=\{(x_1,x_2,x_3)\in(0,\infty)^3:x_1+x_2\leq x_3\}.$$ Since $x_4$ is defined as $x_4=(x_1+x_2-x_3)_+$, then for all $(x_1,x_2,x_3)\in(0,\infty)^3$,
$$(x_1,x_2,x_3)\in\Omega\quad\text{if and only if}\quad x_4=0.$$
It might happen that the sequence $(x_n)_{n\in\NN}$ ``jumps'' from $\Omega$ to $\Omega^c$. If in every neighbourhood of $x$ the sequence has points in both regions, then we may consider two subsequences, each one contained in one region only. For the sequel, the main estimate is the following:
if we denote $x_n=(x_1^n,x_2^n,x_3^n)$ and $w(x_n)=\min\left\{\sqrt{x_1^n},\sqrt{x_2^n},\sqrt{x_3^n},\sqrt{x_4^n}\right\}$, then by (\ref{S2E2}) 
\begin{align}
\label{main estimate to prove continuity}
|\Phi_{\varphi}(x_n)|\leq\|\varphi''\|_{\infty}
\frac{w(x_n)}{\sqrt{x_1^nx_2^nx_3^n}}
\big|x_1^n-x_3^n\big|\big|x_2^n-x_3^n\big|.
\end{align}
We study case by case depending on where $x$ lies.

Case $x=(0,0,0)$. If $(x_n)\subset\Omega$ then $x_4^n=0$, 
$w(x_n)=\sqrt{x_4^n}=0$ and thus $\Phi_{\varphi}(x_n)=0=\Phi_{\varphi}(x)$.\\
If $\{x_n\}\subset\Omega^c$ then $x_4^n>0$ and we study case by case depending on the relative order of $x_1^n$, $x_2^n$, and $x_3^n$. 
Since $\Phi_{\varphi}$
is symmetric in the $x_1$, $x_2$ variables, we may assume without loss of generality that $x_1^n\leq x_2^n$.  
Note by \eqref{main estimate to prove continuity} that we also may assume $x_3^n\neq x_1^n$, $x_3^n\neq x_2^n$; otherwise the result follows directly.

If $x_1^n\leq x_2^n<x_3^n$,
then $w(x_n)=\sqrt{x_4^n}$ and by \eqref{main estimate to prove continuity}
\begin{align*}
|\Phi_{\varphi}(x_n)|&\leq\|\varphi''\|_{\infty}
\frac{\sqrt{x_4^n}}{\sqrt{x_1^nx_2^nx_3^n}}\big(x_3^n-x_1^n\big)\big(x_3^n-x_2^n\big)\\
&\leq\|\varphi''\|_{\infty}\left(\frac{\sqrt{x_4^n}\big(x_3^n\big)^{3/2}}{\sqrt{x_1^nx_2^n}}+\frac{\sqrt{x_4^n x_1^n x_2^n}}{\sqrt{x_3^n}}\right)\\
&\leq\|\varphi''\|_{\infty}\left( \frac{\big(x_3^n\big)^{3/2}}{\sqrt{x_2^n}}+\sqrt{x_1^n x_2^n}\right).
\end{align*}
Since $x_n\to x=0$, then $\sqrt{x_1^n x_2^n}\to 0$. Moreover, since $x_n\in\Omega^c$ and $x_1^n\leq x_2^n$, then 
$x_3^n<2 x_2^n$, and so 
$$ \frac{\big(x_3^n\big)^{3/2}}{\sqrt{x_2^n}}\leq 2^{3/2} x_2^n\longrightarrow 0\qquad\text{as}\quad n\rightarrow\infty.$$

If $x_1^n<x_3^n< x_2^n$, then $w(x_n)=\sqrt{x_1^n}$ and by \eqref{main estimate to prove continuity}
\begin{align*}
|\Phi_{\varphi}(x_n)|&\leq\|\varphi''\|_{\infty}\frac{(x_2^n-x_3^n)(x_3^n -x_1^n)}{\sqrt{x_2^n x_3^n}}\\
&\leq \|\varphi''\|_{\infty}\left(\sqrt{x_2^n x_3^n}+\frac{x_1^n\sqrt{x_3^n}}{\sqrt{x_2^n}}\right)\\
&\leq \|\varphi''\|_{\infty}\left(\sqrt{x_2^n x_3^n}+\sqrt{x_1^n x_3^n}\right)\longrightarrow 0\qquad\text{as}\quad n\rightarrow\infty.
\end{align*}

Lastly, if $x_3^n<x_1^n\leq x_2^n$, then $w(x_n)=\sqrt{x_3^n}$ and by \eqref{main estimate to prove continuity}
\begin{align*}
|\Phi_{\varphi}(x_n)|&\leq\|\varphi''\|_{\infty}\frac{(x_1^n-x_3^n)(x_2^n-x_3^n)}{\sqrt{x_1^n x_2^n}}\\
&\leq \|\varphi''\|_{\infty}\left(\sqrt{x_1^n x_2^n}+\frac{\big(x_3^n\big)^2}{\sqrt{x_1^n x_2^n}}\right)\\
&\leq 2\|\varphi''\|_{\infty}\left(\sqrt{x_1^n x_2^n}+x_1\right)\longrightarrow 0\qquad\text{as}\quad n\rightarrow\infty.
\end{align*}
Hence, in the three cases above $\Phi_{\varphi}(x_n)\rightarrow 0=\Phi_{\varphi}(x).$

Case $x=(x_1,0,0)$ with $x_1>0$. 
Then $w(x_n)=
\min\left\{\sqrt{x_2^n},\sqrt{x_3^n}\right\}$ for $n$ large enough.
On the other hand
\begin{align*}
\big|x_2^n-x_3^n\big|&=\big(\sqrt{x_2^n}+\sqrt{x_3^n}\big)\big|\sqrt{x_2^n}-\sqrt{x_3^n}\big|\\
&\leq 2\max\left\{\sqrt{x_2^n},\sqrt{x_3^n}\right\}\big|\sqrt{x_2^n}-\sqrt{x_3^n}\big|.
\end{align*}
Since $\min\left\{\sqrt{x_2^n},\sqrt{x_3^n}\right\}\max\left\{\sqrt{x_2^n},\sqrt{x_3^n}\right\}=\sqrt{x_2^nx_3^n}$, then by \eqref{main estimate to prove continuity} 
\begin{align*}
|\Phi_{\varphi}(x_n)|\leq2\|\varphi''\|_{\infty}\frac{\big|x_1^n-x_3^n\big|}{\sqrt{x_1^n}}\big|\sqrt{x_2^n}-\sqrt{x_3^n}\big|
\end{align*}
for $n$ large enough. It then follows $\Phi_{\varphi}(x_n)\rightarrow 0=\Phi_{\varphi}(x)$ as $n\to\infty$.

The case $x=(0,x_2,0)$ with $x_2>0$ is analogous to the previous one thanks to the symmetry of $\Phi_{\varphi}$ in the $x_1$, $x_2$ variables.

Case $x=(0,0,x_3)$ with $x_3>0$. Then $x_n\in\Omega$ for $n$ large enough, $x_4^n=0$ and
$w(x_n)=\sqrt{x_4^n}=0$. Thus
$\Phi_{\varphi}(x_n)=0=\Phi_{\varphi}(x)$ for $n$ large enough.

Case $x=(0,x_2,x_3)$ with $x_2>0$ and $x_3>0$. If $x_2>x_3$ then
$w(x_n)=\sqrt{x_1^n}$ for $n$ large enough and 
\begin{align*}
|\Phi_{\varphi}(x_n)-\Phi_{\varphi}(x)|=\left|\frac{1}{\sqrt{x_2^nx_3^n}}\Delta\varphi(x_1^n,x_2^n,x_3^n)
-\frac{1}{\sqrt{x_2x_3}}\Delta\varphi(0,x_2,x_3)\right|,
\end{align*}
which clearly goes to zero as $n\rightarrow\infty$. If $x_2<x_3$ then $x_4^n=0$ for $n$ large enough and 
$w(x_n)=\sqrt{x_4^n}=0$, thus $\Phi_{\varphi}(x_n)=0=\Phi_{\varphi}(x)$.
If $x_2=x_3$ and $(x_n)\subset\Omega$ for $n$ large enough, then $x_4^n=0$, thus
$\Phi_{\varphi}(x_n)=0=\Phi_{\varphi}(x)$.\\ 
If $x_2=x_3$ and $(x_n)\subset\Omega^c$ for $n$ large enough, then
$w(x_n)
=\min\left\{\sqrt{x_1^n},\sqrt{x_4^n}\right\}$, and by \eqref{main estimate to prove continuity} 
\begin{align*}
|\Phi_{\varphi}(x_n)|\leq\|\varphi''\|_{\infty}
\frac{\min\left\{\sqrt{x_1^n},\sqrt{x_4^n}\right\}}{\sqrt{x_1^nx_2^nx_3^n}}
\big|x_1^n-x_3^n\big|\big|x_2^n-x_3^n\big|.
\end{align*}
On the one hand 
\begin{align*}
\big|x_1^n-x_3^n\big|
\leq 2\max\left\{\sqrt{x_1^n},\sqrt{x_3^n}\right\}\big|\sqrt{x_1^n}-\sqrt{x_3^n}\big|.
\end{align*}
On the other hand 
$\min\left\{\sqrt{x_1^n},\sqrt{x_4^n}\right\}\leq\min\left\{\sqrt{x_1^n},\sqrt{x_3^n}\right\}$ for $n$ large enough. 
Since $\min\left\{\sqrt{x_1^n},\sqrt{x_3^n}\right\}\max\left\{\sqrt{x_1^n},\sqrt{x_3^n}\right\}=\sqrt{x_1^n x_3^n}$, then
\begin{align*}
|\Phi_{\varphi}(x_n)|\leq2\|\varphi''\|_{\infty}\frac{\big|x_2^n-x_3^n\big|}{\sqrt{x_2^n}}\big|\sqrt{x_1^n}-\sqrt{x_3^n}\big|,
\end{align*}
which goes to zero as $n\rightarrow\infty$ since $x_2=x_3$. Thus $\Phi_{\varphi}(x_n)\rightarrow 0=\Phi_{\varphi}(x)$.

The case $x=(x_1,0,x_3)$ with $x_1>0$ and $x_3>0$ is analogous to the previous one thanks to the symmetry of $\Phi_{\varphi}$ in the $x_1$, $x_2$ variables.

Case $x=(x_1,x_2,0)$ with $(x_1,x_2)\in (0,\infty)^2$. Then $w(x_n)=\sqrt{x_3^n}$ for $n$ large enough and
\begin{align*}
|\Phi_{\varphi}(x_n)-\Phi_{\varphi}(x)|=\left|\frac{1}{\sqrt{x_1^nx_2^n}}\Delta\varphi(x_1^n,x_2^n,x_3^n)
-\frac{1}{\sqrt{x_1x_2}}\Delta\varphi(x_1,x_2,0)\right|,
\end{align*}
which clearly goes to zero as $n\rightarrow\infty$.

\textbf{Proof of (ii).}
By part (i) $\Phi_{\varphi}\in C([0,\infty)^3)$. Let us show now that for any given $\varepsilon>0$ there exists $R(\varepsilon)>0$ such that 
$|\Phi_{\varphi}(x)|\leq\varepsilon$ for all $x\in[0,\infty)^3\setminus[0,R(\varepsilon)]^3$. 

Given $R>0$ and $\alpha>0$, let $(x_1,x_2,x_3)\in[0,\infty)^3\setminus[0,R]^3$ and denote
$x_i=\min\{x_1,x_2,x_3\}$, $x_k=\max\{x_1,x_2,x_3\}$ and $x_j$ neither $x_i$ nor $x_k$.
Notice that $x_k>R$ and the function $W$ defined in (\ref{S2E3}) satisfies $W(x_1,x_2,x_3)\leq\frac{1}{\sqrt{x_jx_k}}$.
If $x_i>\alpha$ or $x_j>\alpha$ then by (\ref{S2E2}) 
\begin{align*}
|\Phi_{\varphi}(x_1,x_2,x_3)|\leq\frac{|\Delta\varphi(x_1,x_2,x_3)|}{\sqrt{x_j x_k}}\leq\frac{4\|\varphi\|_{\infty}}{\sqrt{\alpha R}}\leq\varepsilon,
\end{align*}
provided $R\geq\frac{16\|\varphi\|^2_{\infty}}{\alpha\varepsilon^2}.$
If $x_i\leq\alpha$ and $x_j\leq\alpha$ we study case by case depending on the relative position of $x_1$, $x_2$, $x_3$. Since $\Phi_{\varphi}$ is symmetric in variables $x_1$ and $x_2$, we may assume without loss of generality that $x_2\leq x_1$.
If $x_k=x_1$, using (\ref{S2E2})
\begin{align*}
|\Phi_{\varphi}(x_1,x_2,x_3)|&\leq\frac{2\|\varphi'\|_{\infty}(x_j-x_i)}{\sqrt{x_1 x_j}}
\leq\frac{2\|\varphi'\|_{\infty}\sqrt{x_j}}{\sqrt{x_1}}
\leq\frac{2\|\varphi'\|_{\infty}\sqrt{\alpha}}{\sqrt{R}}\leq\varepsilon,
\end{align*}
provided $R\geq\frac{4\|\varphi'\|^2_{\infty}\alpha}{\varepsilon^2}.$
If $x_k=x_3$ and $x\in\Omega$ then $x_4=0$ and $\Phi_{\varphi}(x)=0$.
If $x_k=x_3$ and $x\in\Omega^c$, then
$x_1\geq R/2$ and
\begin{align*}
|\Phi_{\varphi}(x_1,x_2,x_3)|&\leq\frac{4\|\varphi\|_{\infty}}{\sqrt{x_1x_3}}\leq\frac{4\sqrt{2}\|\varphi\|_{\infty}}{R}\leq\varepsilon,
\end{align*}
provided $R\geq \frac{4\sqrt{2}\|\varphi\|_{\infty}}{\varepsilon}$.

Finally, if we chose $R\geq\max\left\{\frac{16\|\varphi\|^2_{\infty}}{\alpha\varepsilon^2},\frac{4\|\varphi'\|^2_{\infty}\alpha}{\varepsilon^2},
\frac{4\sqrt{2}\|\varphi\|_{\infty}}{\varepsilon}\right\}$ 
then $\Phi_{\varphi}\in C_0([0,\infty)^3)$ and in particular, $\Phi_{\varphi}$ is uniformly continuous in $[0,\infty)^3$.
\end{proof}

\noindent
\textbf{Acknowledgments.}
The research of the first author is supported by the Basque Government through the BERC 2014-2017 program, by the Spanish Ministry of Economy and Competitiveness MINECO: BCAM Severo Ochoa accreditation SEV-2013-0323, and by MTM2014-52347-C2-1-R of DGES. The research of the second author is supported by grants MTM2014-52347-C2-1-R of DGES and  IT641-13 of the Basque Government. The authors acknowledge the valuable remarks and helpful comments received from A. H. M. Kierkels and Pr. J. J. L. Vel\'azquez, as well as the hospitality of  IAM at the University of Bonn.

\bibliographystyle{plain}
\bibliography{Biblio.bib}

\begin{thebibliography}{10}

\bibitem{AMB}
L.~Ambrosio, N.~Gigli, and G.~Savare.
\newblock {\em Gradient Flows: In Metric Spaces and in the Space of Probability
  Measures}.
\newblock Lectures in Mathematics. ETH Z{\"u}rich. Birkh{\"a}user Basel, 2005.

\bibitem{Bob}
A.~V. Bobylev.
\newblock Moment inequalities for the {B}oltzmann equation and applications to
  spatially homogeneous problems.
\newblock {\em Journal of Statistical Physics}, 1997.

\bibitem{Chiara}
C.~Boccato, C.~Brennecke, S.~Cenatiempo, and B.~Schlein.
\newblock Bogoliubov theory in the {G}ross-{P}itaevskii limit.
\newblock {\em arXiv:1801.01389 [math-ph]}, page 101, 2018.

\bibitem{BOG}
V.~I. Bogachev.
\newblock {\em Measure Theory}.
\newblock Number vol. 2. Springer Berlin Heidelberg, 2007.

\bibitem{BE}
M.~Briant and A.~Einav.
\newblock On the {C}auchy problem for the homogeneous {B}oltzmann--{N}ordheim
  equation for bosons: local existence, uniqueness and creation of moments.
\newblock {\em Journal of Statistical Physics}, 2016.

\bibitem{Eckern}
U.~Eckern.
\newblock Relaxation processes in a condensed {B}ose gas.
\newblock {\em Journal of Low Temperature Physics}, 54:333--359, 1984.
\newblock 10.1007/BF00683281.

\bibitem{ESY}
L.~Erd\"os, B.~Schlein, and H.-T. Yau.
\newblock Derivation of the {G}ross-{P}itaevskii equation for the dynamics of
  {B}ose-{E}instein condensate.
\newblock {\em Annals of Mathematics}, 172(1):291--370, 2010.

\bibitem{EV1}
M.~Escobedo and J.~J.~L. Vel\'azquez.
\newblock Finite time blow-up and condensation for the bosonic {N}ordheim
  equation.
\newblock {\em Invent. Math.}, 200(3):761--847, 2015.

\bibitem{Fo}
G.B. Folland.
\newblock {\em {R}eal {A}nalysis: modern techniques and their applications}.
\newblock Pure and Applied Mathematics. Wiley, 1999.

\bibitem{Zoller}
C.~W. Gardiner and P.~Zoller.
\newblock Quantum kinetic theory {V}. {Q}uantum kinetic master equation for
  mutual interaction of condensate and noncondensate.
\newblock {\em Phys. Rev. A}, 61:033601, 2000.

\bibitem{GNZB}
A.~Griffin, T.~Nikuni, and E.~Zaremba.
\newblock {\em Bose-Condensed Gases at Finite Temperatures}.
\newblock Cambridge University Press, 2009.

\bibitem{IG}
M.~Imamovic-Tomasovic and A.~Griffin.
\newblock Quasiparticle kinetic equation in a trapped {B}ose gas at low
  temperatures.
\newblock {\em J. Low Temp. Phys.}, 122:617--655, 2001.

\bibitem{KIER}
A.~H.~M Kierkels.
\newblock {\em {O}n a {K}inetic {E}quation {A}rising in {W}eak {T}urbulence
  {T}heory for the {N}onlinear {S}chr{\"o}dinger {E}quation}.
\newblock PhD thesis, Universit\"{a}t Bonn, 2016.

\bibitem{AV1}
A.~H.~M. Kierkels and J.~J.~L. Vel{\'a}zquez.
\newblock On the {T}ransfer of {E}nergy {T}owards {I}nfinity in the {T}heory of
  {W}eak {T}urbulence for the {N}onlinear {S}chr{\"o}dinger {E}quation.
\newblock {\em Journal of Statistical Physics}, 2015.

\bibitem{Kirkpatrick}
T.~R. Kirkpatrick and J.~R. Dorfman.
\newblock Transport in a dilute but condensed nonideal {B}ose gas: Kinetic
  equations.
\newblock {\em Journal of Low Temperature Physics}, 58(3):301--331, 02 1985.

\bibitem{Lu5}
W.~Li and X.~Lu.
\newblock Global existence of solutions of the {B}oltzmann equation for
  {B}ose-{E}instein particles with anisotropic initial data.
\newblock {\em Preprint arXiv:1706.06235v3 [math.AP]}, page~48, 2017.

\bibitem{lieb}
E.H. Lieb, R.~Seiringer, J.~P. Solovej, and J.~Yngvason.
\newblock {\em The mathematics of the Bose Gas and its condensation}.
\newblock Oberwolfach Seminars. Springer Basel AG, 2005.

\bibitem{Lu1}
X.~Lu.
\newblock On isotropic distributional solutions to the {B}oltzmann equation for
  {B}ose-{E}instein particles.
\newblock {\em J. Statist. Phys.}, 116(5-6):1597--1649, 2004.

\bibitem{Lu2}
X.~Lu.
\newblock The {B}oltzmann equation for {B}ose-{E}instein particles: velocity
  concentration and convergence to equilibrium.
\newblock {\em J. Stat. Phys.}, 119(5-6):1027--1067, 2005.

\bibitem{Lu3}
X.~Lu.
\newblock The {B}oltzmann equation for {B}ose-{E}instein particles:
  condensation in finite time.
\newblock {\em J. Stat. Phys.}, 150(6):1138--1176, 2013.

\bibitem{Lu4}
X.~Lu.
\newblock The {B}oltzmann equation for {B}ose-{E}instein particles: regularity
  and condensation.
\newblock {\em J. Stat. Phys.}, 156(3):493--545, 2014.

\bibitem{Nordheim1}
L.~W. Nordheim.
\newblock On the kinetic method in the new statistics and its application in
  the electron theory of conductivity.
\newblock {\em Proc. R. Soc. Lond. A}, 119:689--698, 1928.

\bibitem{Nouri1}
A.~Nouri.
\newblock Bose-{E}instein condensates at very low temperatures: a mathematical
  result in the isotropic case.
\newblock {\em Bull. Inst. Math. Acad. Sin. (N.S.)}, 2(2):649--666, 2007.

\bibitem{Roy}
H.~L. Royden.
\newblock {\em Real {A}nalysis}.
\newblock Macmillan, 2nd edition, 1968.

\bibitem{WR3}
W.~Rudin.
\newblock {\em Functional Analysis}.
\newblock International Series in Pure and Applied Mathematics. McGraw-Hill
  Inc., New York, second edition, 1991.

\bibitem{ST1}
D.~V. Semikoz and I.~I. Tkachev.
\newblock Kinetics of {B}ose condensation.
\newblock {\em Phys. Rev. Lett.}, 74(16):3093--3097, 1995.

\bibitem{Spohn}
H.~Spohn.
\newblock Kinetics of the {B}ose-{E}instein condensation.
\newblock {\em Physica D: Nonlinear Phenomena}, 239(10):627 -- 634, 2010.

\bibitem{Stoof2}
H.~T.~C. Stoof.
\newblock Field theory for trapped atomic gases.
\newblock In R.~Kaiser, C.~Westbrook, and F.~David, editors, {\em Coherent
  atomic matter waves}, pages 219--315, Berlin, Heidelberg, 2001. Springer
  Berlin Heidelberg.

\bibitem{Svis}
B.~V. Svistunov.
\newblock Highly nonequilibrium {B}ose condensation in a weakly interacting
  gas.
\newblock {\em Journal of the Moscow Physical Society}, 1:373--390, 1991.

\bibitem{ZNG}
E.~Zaremba, T.~Nikuni, and A.~Griffin.
\newblock Dynamics of trapped {B}ose gases at finite temperatures.
\newblock {\em Journal of Low Temperature Physics}, 116:277--345, 1999.
\newblock 10.1023/A:1021846002995.

\end{thebibliography}
\end{document}